\documentclass[10pt]{amsart}

\usepackage{amsmath,amssymb,amsfonts,fullpage,amssymb,amsthm}

\usepackage{verbatim}
\usepackage[usenames]{color}
\usepackage{hyperref}
\usepackage{url}

\newtheorem{thm}{Theorem}
\newtheorem{prop}[thm]{Proposition} 
\newtheorem{lem}[thm]{Lemma}
\newtheorem{cor}[thm]{Corollary}
\newtheorem{clm}{Claim}

\newtheorem{fact}[thm]{Fact}

\theoremstyle{remark}
\newtheorem*{rem}{Remark}
\newtheorem*{attribution}{Attributions}
\newtheorem*{Ack}{Acknowledgments}

\newtheorem*{subclaimn}{Subclaim}

\theoremstyle{definition}
\newtheorem{defn}[thm]{Definition}

\newtheorem*{notation}{Notation}
\newtheorem{notationn}{Notation}
\newtheorem{example}[thm]{Example} 
\newtheorem{examples}[thm]{Examples} 
\newtheorem{question}[thm]{Question}
\newtheorem*{question1}{Question 1}
\newtheorem*{question2}{Question 2}
\newtheorem*{question3}{Question 3}

\theoremstyle{remark}

\newcommand{\fin}{\mathrm{fin}}

\newcommand{\lra}{\leftrightarrow}

\newcommand{\al}{\alpha}
\newcommand{\om}{\omega}
\newcommand{\vp}{\varphi}
\newcommand{\sse}{\subseteq}
\newcommand{\contains}{\supseteq}

\DeclareMathOperator{\Iso}{Iso}
 
\newcommand{\re}{\restriction}

\newcommand{\bA}{\mathbb{A}}

\newcommand{\bN}{\mathbb{N}}

\newcommand{\bsA}{\boldsymbol{A}}
\newcommand{\bsB}{\boldsymbol{B}}
\newcommand{\bsC}{\boldsymbol{C}}
\newcommand{\bsD}{\boldsymbol{D}}
\newcommand{\bsE}{\boldsymbol{E}}
\newcommand{\bsF}{\boldsymbol{F}}
\newcommand{\bsS}{\boldsymbol{S}}
\newcommand{\bsT}{\boldsymbol{T}}
\newcommand{\bsU}{\boldsymbol{U}}
\newcommand{\bsV}{\boldsymbol{V}}
\newcommand{\bsW}{\boldsymbol{W}}
\newcommand{\bsX}{\boldsymbol{X}}
\newcommand{\bsY}{\boldsymbol{Y}}
\newcommand{\bsZ}{\boldsymbol{Z}}

\DeclareMathOperator{\depth}{depth}

\DeclareMathOperator{\FIN}{FIN}

\newcommand{\ra}{\rightarrow}

\newcommand{\Llra}{\Longleftrightarrow}
\newcommand{\lgl}{\langle}
\newcommand{\rgl}{\rangle}

\newcommand{\Pudlak}{Pudl{\'{a}}k}
\newcommand{\Rodl}{R{\"{o}}dl}
\newcommand{\Erdos}{Erd{\H{o}}s}

\newcommand{\Fraisse}{Fra{\"{i}}ss{\'{e}}}
\newcommand{\Sokic}{Soki{\'{c}}}
\newcommand{\Nesetril}{Ne{\v{s}}et{\v{r}}il}

\definecolor{darkgreen}{rgb}{0,0.6,0}
\newcommand{\noprint}[1]{\relax}

\title{Topological Ramsey spaces from Fra\"{i}ss\'{e}\ classes, Ramsey-classification theorems, and initial structures in the Tukey types of p-points}

\author{Natasha
  Dobrinen}
\address{Department of Mathematics\\
  University of Denver \\
   2280 S Vine St\\ Denver, CO \ 80208 U.S.A.}  
\email{natasha.dobrinen@du.edu}
\urladdr{\url{http://web.cs.du.edu/~ndobrine}} 
\thanks{Dobrinen  was partially supported by  National Science Foundation Grant DMS-1301665  and  Simons Foundation Collaboration Grant 245286}

\author{Jos\'e G. Mijares}
\address{Department of Mathematics\\
  University of Denver \\
   2280 S Vine St\\ Denver, CO \ 80208 U.S.A.}
\email{Jose.MijaresPalacios@du.edu}
\thanks{Mijares  was partially supported Dobrinen's  Simons Foundation Collaboration Grant 245286}

\author{Timothy Trujillo}
\address{Applied Mathematics and Statistics\\
 Colorado School of Mines \\
 Golden,  CO \ 80401 U.S.A.}
\email{trujillo@mines.edu}


\dedicatory{Dedicated to James Baumgartner, whose depth and insight  continue to inspire}

\begin{document}

\maketitle

\begin{abstract}
A general method for constructing a new class of topological Ramsey spaces is presented.
Members of such spaces are infinite sequences of 
 products of Fra\"{i}ss\'{e} classes of finite relational structures satisfying the Ramsey property.
The Product Ramsey Theorem of Soki\v{c} is extended to equivalence relations for  
finite products of structures from  \Fraisse\ classes of finite relational structures satisfying the Ramsey property and the Order-Prescribed Free Amalgamation Property.
This is essential to
proving Ramsey-classification theorems for equivalence relations on fronts, generalizing the \Pudlak-\Rodl\ Theorem to this class of topological Ramsey spaces.

To each topological Ramsey space in this framework corresponds an associated ultrafilter  satisfying some weak partition property.
By using the correct \Fraisse\ classes,  we construct topological Ramsey spaces which are dense in the partial orders  of Baumgartner and Taylor in \cite{Baumgartner/Taylor78}  
   generating  p-points  which are $k$-arrow but not $k+1$-arrow, and in a
partial order of Blass in \cite{Blass73} producing a diamond shape in the Rudin-Keisler structure of p-points.
Any space in our framework in which blocks are products of $n$ many structures
produces ultrafilters with initial Tukey structure exactly the  Boolean algebra $\mathcal{P}(n)$.
If the number of \Fraisse\ classes on each block grows without bound, then 
 the Tukey types of the p-points below the space's associated ultrafilter have the structure exactly $[\om]^{<\om}$.
In contrast,  
the set of isomorphism types
of any
  product of finitely many \Fraisse\ classes 
of finite relational structures satisfying the Ramsey property and the OPFAP,
  partially ordered by embedding,
 is realized as the initial Rudin-Keisler structure of some p-point generated by a space  constructed from our template.

\end{abstract}

\section{Introduction}\label{sec.intro}

The Tukey theory of ultrafilters has 
recently seen much progress, developing into a full-fledged area of research drawing on set theory, topology, and Ramsey theory.
Interest in Tukey reducibility on ultrafilters stems both from the fact that it is a weakening of the well-known Rudin-Keisler reducibility as well as the fact that it is a useful tool for classifying partial orderings.

Given ultrafilters $\mathcal{U},\mathcal{V}$, we say that $\mathcal{U}$ is {\em Tukey reducible to} $\mathcal{V}$ (written $\mathcal{U}\le_T\mathcal{V}$)
if there is a function $f:\mathcal{V}\ra\mathcal{U}$ which sends filter bases of $\mathcal{V}$ to filter bases of $\mathcal{U}$.
We say that $\mathcal{U}$ and $\mathcal{V}$ are {\em Tukey equivalent} if both $\mathcal{U}\le_T\mathcal{V}$ and $\mathcal{V}\le_T\mathcal{U}$.
The collection of all ultrafilters Tukey equivalent to $\mathcal{U}$ is called the {\em Tukey type} of $\mathcal{U}$.

The question of which structures embed into the Tukey types of ultrafilters on the natural numbers was addressed to some extent in \cite{Dobrinen/Todorcevic10}.
In that paper, the following were shown to be consistent with ZFC: chains of length $\mathfrak{c}$ embed into the Tukey types of p-points; diamond configurations embed into the Tukey types of p-points;
and there are $2^{\mathfrak{c}}$ many Tukey-incomparable selective ultrafilters.
However, \cite{Dobrinen/Todorcevic10} left open the question of which structures appear as {\em initial Tukey structures} in the Tukey types of ultrafilters, where 
by  an initial Tukey structure we mean a collection of  Tukey types of nonprincipal ultrafilters which is closed under Tukey reducibility.

The first progress in this direction was made in
\cite{Raghavan/Todorcevic12},
where applying a canonical Ramsey theorem of \Pudlak\ and \Rodl\ (see Theorem \ref{thm.PR}),
Todorcevic showed that every nonprincipal ultrafilter Tukey reducible to a Ramsey ultrafilter is in fact Tukey equivalent to that Ramsey ultrafilter.
Thus, 
the initial Tukey structure below a Ramsey ultrafilter is simply a singleton.

Further progress on initial Tukey structures was made by Dobrinen and Todorcevic  in
 \cite{Dobrinen/Todorcevic11} and \cite{Dobrinen/Todorcevic12}.
To each topological Ramsey space, there is a naturally associated ultrafilter obtained by forcing with the topological Ramsey space partially ordered modulo finite initial segments.  The  properties of the associated ultrafilters are inherited from the properties of the topological Ramsey space (see Section \ref{sec.uf}).
In \cite{Dobrinen/Todorcevic11}, a dense subset of a partial ordering of Laflamme  from \cite{Laflamme89} which forces a weakly Ramsey ultrafilter was pared down to reveal the inner structure responsible for the desired properties to be  that of a topological Ramsey space, $\mathcal{R}_1$.
In fact,  Laflamme's partial ordering is exactly that of an earlier example of Baumgartner and Taylor in \cite{Baumgartner/Taylor78} (see Example \ref{ex.BT4.84.9}).
By proving and applying  a new Ramsey classification theorem, generalizing the \Pudlak-\Rodl\ Theorem for canonical equivalence relations on barriers,
it was shown in  \cite{Dobrinen/Todorcevic11} that the ultrafilter associated with $\mathcal{R}_1$ has exactly one Tukey type of nonprincipal ultrafilters strictly below it, namely that of the projected Ramsey ultrafilter, and similarly for Rudin-Keisler reduction.
Thus, the initial   Tukey and Rudin-Keisler structures of nonprincipal ultrafilters reducible to the ultrafilter associated with $\mathcal{R}_1$ are both exactly a chain of length $2$.

In \cite{Dobrinen/Todorcevic12},
this work was extended to a new class of topological Ramsey spaces $\mathcal{R}_{\al}$, which are obtained as particular dense sets of forcings of Laflamme in \cite{Laflamme89}.
In \cite{Dobrinen/Todorcevic12}, it was proved that the structure of the Tukey types of ultrafilters Tukey reducible to the ultrafilter associated with $\mathcal{R}_{\al}$ is exactly a decreasing chain of order-type $\al+1$.
Likewise for the initial Rudin-Keisler structure.
As before, this result was obtained by proving new Ramsey-classification theorems for canonical equivalence relations on barriers and applying them to deduce the Tukey and Rudin-Keisler structures below the ultrafilter associated with $\mathcal{R}_{\al}$.

All of the results in \cite{Raghavan/Todorcevic12}, \cite{Dobrinen/Todorcevic11} and \cite{Dobrinen/Todorcevic12}   produced initial Tukey and Rudin-Keisler structures which are linear orders, precisely, decreasing chains of some countable successor ordinal length. 
This led to the following questions, which motivated the present and forthcoming work.

\begin{question}\label{q.1}
What are the possible initial Tukey structures for ultrafilters on a countable base set?
\end{question}

\begin{question}\label{q.2}
What are the possible initial Rudin-Keisler structures for ultrafilters on a countable base set?
\end{question}

\begin{question}\label{q.3}
 For a given ultrafilter $\mathcal{U}$, what is the structure of the Rudin-Keisler ordering of the isomorphism classes of ultrafilters Tukey reducible to $\mathcal{U}$?
\end{question}

Question \ref{q.3} was answered in  \cite{Raghavan/Todorcevic12}, \cite{Dobrinen/Todorcevic11} and \cite{Dobrinen/Todorcevic12}  by showing that each Tukey type below the associated ultrafilter consists of iterated Fubini products of p-points obtained from projections of the ultrafilter forced by the space.

Related to these questions are the following two motivating questions.
Before \cite{Dobrinen/Todorcevic12}, there were relatively few examples in the literature of topological Ramsey spaces.  The constructions in that paper led to considering what other new topological Ramsey spaces can be formed. 
Our general construction method presented in Section \ref{sec.genseq} is a step toward answering the following larger question.

\begin{question}\label{q.constructionscheme}
What general construction schemes are there for constructing new topological Ramsey spaces?
\end{question}

We point out some recent work in this vein constructing new types of topological Ramsey spaces. 
In \cite{Mijares/Padilla13}, Mijares and Padilla construct new spaces of infinite polyhedra, and in \cite{Mijares/Torrealba13},
Mijares and Torrealba construct spaces whose members are
 countable metric spaces with rational valued metrics.
These spaces answer  questions in Ramsey theory regarding homogeneous structures and random objects.
One of aims  of the present work is to find a general framework for ultrafilters satisfying partition properties in terms of topological Ramsey spaces.
See also \cite{Dobrinen14} for a new construction scheme.

\begin{question}\label{q.frameworkuf}
Is each ultrafilter on some countable base satisfying some partition relations actually an ultrafilter associated with some topological Ramsey space (or something close to a topological Ramsey space)?
Is there some general framework of topological Ramsey spaces into which many or all examples of ultrafilters with partition properties fit?
\end{question}

Some recent work of Dobrinen in \cite{Dobrinen14} constructs high-dimensional extensions of the Ellentuck  space.
These topological Ramsey spaces  generate  ultrafilters which are not p-points but which have strong partition properties;
precisely these spaces yield the ultrafilters generic for the forcings $\mathcal{P}(\om^n)/$Fin$^{\otimes n}$, $2\le n<\om$.
The structure of the spaces aids in 
finding their initial Tukey structures via new extensions of the \Pudlak-\Rodl\ Theorem for these spaces.

It turns out that whenever an ultrafilter is associated with some topological Ramsey space, the ultrafilter has complete combinatorics, meaning that in the presence of a supercompact cardinal, the ultrafilter is generic over $L(\mathbb{R})$.  
This was proved by Di Prisco, Mijares, and Nieto in \cite{DiPrisco/Mijares/Nieto15}, building on work of Todorcevic in 
 \cite{Farah98} for the Ellentuck space.
Thus, finding a general framework for ultrafilters with partition properties in terms of ultrafilters associated with  topological Ramsey spaces has the benefit of providing a large class of forcings with complete combinatorics.

In this
 paper we provide a general scheme for 
constructing new topological Ramsey spaces.
This construction scheme uses products of finite ordered relational structures from  \Fraisse\ classes with the Ramsey property. 
The details are set out in Section \ref{sec.genseq}.
The goal of this construction scheme is 
several-fold.
We aim to construct topological Ramsey spaces with associated ultrafilters which have initial Tukey structures which are not simply linear orders.
This is achieved by allowing ``blocks" of the members of the Ramsey space to consist of products of structures, rather than trees as was the case in \cite{Dobrinen/Todorcevic12}.
In particular, for each $n<\om$, we construct a hypercube space $\mathcal{H}^n$  which produces an ultrafilter with initial Tukey and Rudin-Keisler structures exactly that of the Boolean algebra $\mathcal{P}(n)$.
See  Example \ref{ex.H} and 
Theorems \ref{Tim-InitialStructures} and  \ref{TukeyReducibleTheorem2}.

We also  seek to use topological Ramsey spaces to
provide a unifying framework for p-points satisfying weak partition properties.
This is the focus  in Section \ref{sec.uf}.
All of the p-points of Baumgartner and Taylor in \cite{Baumgartner/Taylor78}  fit into our scheme,
in particular, the $k$-arrow, not $(k+1)$-arrow p-points which they construct.
In the other direction,  for many collections of weak partition properties,
we show 
there is a topological Ramsey space with associated ultrafilter simultaneously satisfying those properties.

The general Ramsey-classification Theorem \ref{canonical R} in Section \ref{sec.canon}  hinges on 
Theorem \ref{FiniteERProducts} in Section \ref{sec.ProductER}, which 
 generalizes the \Erdos-Rado Theorem (see Theorem \ref{thm.finiteER})
in two ways: by extending it from finite linear orders to 
 \Fraisse\ classes of finite ordered relational structures with the Ramsey property and the Order-Prescribed  Free Amalgamation Property (see Definition \ref{def.opfap}),
and by extending it to 
finite products of members of such classes.
Theorem \ref{FiniteERProducts} also extends the Product Ramsey Theorem of Soki\v{c} (see Theorem \ref{thm.Sokic}) from finite colorings  to equivalence relations, but at the expense of restricting to a certain subclass of those \Fraisse\ classes for which his theorem holds.
Theorem \ref{FiniteERProducts} is applied in 
Section \ref{sec.canon} to prove
 Theorem \ref{canonical R},
which generalizes the Ramsey-classification theorems in \cite{Dobrinen/Todorcevic11} for equivalence relations on fronts
to the setting of the topological Ramsey spaces in this paper.
Furthermore, we show that the Abstract Nash-Williams Theorem (as opposed to the Abstract Ellentuck Theorem) suffices for the proof.

  Section \ref{sec.basic} contains theorems general to all topological Ramsey spaces $(\mathcal{R},\le,r)$, not just those constructed from a generating sequence.
In this section, general notions of a filter being selective or  Ramsey for the space $\mathcal{R}$ are put forth.
The main result of this section,
Theorem \ref{p-point Tukey theorem},
shows that
 Tukey reductions for ultrafilters Ramsey for a topological Ramsey space
can be assumed to be continuous with respect to the metric topology on the Ramsey space.
In particular, it is shown that any cofinal map from an ultrafilter Ramsey for $\mathcal{R}$ is 
continuous on some base for that ultrafilter, and even better, is {\em basic} (see Definition \ref{defn.basic}).
This section also contains a general method for analyzing 
ultrafilters Tukey reducible to some ultrafilter Ramsey for $\mathcal{R}$ in terms of fronts and canonical functions.
(See  Proposition \ref{Proposition 5.5 Analogue}  and neighboring text.)

Theorems \ref{canonical R} and \ref{p-point Tukey theorem} are applied 
in Section \ref{sec.initstruc}
to answer Questions \ref{q.1} - \ref{q.3}.
All initial Tukey and Rudin-Keisler structures associated with the ultrafilters generated by the class of topological Ramsey spaces constructed in this paper are found.
 Theorem \ref{Tim-InitialStructures},
shows that whenever $n$  \Fraisse\ classes 
are used to generate a topological Ramsey space, then the initial Tukey structure below the associated ultrafilter is exactly the Boolean algebra $\mathcal{P}(n)$.
When infinitely many \Fraisse\ classes are used, then the initial Tukey structure of the p-points below the associated forced filter is exactly $([\om]^{<\om},\sse)$.
In Theorem
\ref{TukeyReducibleTheorem},
we find the exact structure of the Rudin-Keisler types inside the Tukey types of ultrafilters reducible to the associated filter.
Theorem \ref{TukeyReducibleTheorem2} 
shows that if $\mathcal{R}$ is a topological Ramsey space constructed from some \Fraisse\ classes $\mathcal{K}_j$, $j\in J$, and $\mathcal{C}$ is a Ramsey filter on $(\mathcal{R}, \le )$,
then 
 the Rudin-Keisler ordering of the p-points Tukey reducible to $\mathcal{C}$ is isomorphic to the 
collection of all (equivalence classes of) finite products of members of the classes $\mathcal{K}_j$, partially ordered under embeddability.

\begin{attribution}
The work in Sections 3 - 5 is due to Dobrinen.
Section \ref{sec.canon} comprises joint work of Dobrinen and Mijares.
Sections  \ref{sec.basic}
and \ref{sec.initstruc} are joint work of Dobrinen and Trujillo,
building on some of the work in Trujillo's thesis.
The main results in this paper for the special case of the  space $\mathcal{H}^2$   constitute work of Trujillo  in  his PhD thesis \cite{TrujilloThesis}.
\end{attribution}

\begin{Ack}
The authors gratefully acknowledge input from the first anonymous  referee pointing out an oversight in the first draft which led us to formulate the OPFAP.
We also thank the second anonymous referee for pointing out some typos. 
Many thanks go to   Miodrag Socki\v{c} for his thorough reading of previous drafts, catching typos and some errors which have been fixed.
\end{Ack}


\section{Background on topological Ramsey spaces, notation,\\ and classical canonization theorems}\label{sec.2}

In 
\cite{TodorcevicBK10}, Todorcevic distills the key properties of the Ellentuck space into  four axioms which guarantee that a space is a topological Ramsey space.
For the sake of clarity, we reproduce his definitions here.
The following can all be found at the beginning of Chapter 5 in
\cite{TodorcevicBK10}.

The  axioms \bf A.1 \rm - \bf A.4 \rm
are defined for triples
$(\mathcal{R},\le,r)$
of objects with the following properties.
$\mathcal{R}$ is a nonempty set,
$\le$ is a quasi-ordering on $\mathcal{R}$,
 and $r:\mathcal{R}\times\om\ra\mathcal{AR}$ is a mapping giving us the sequence $(r_n(\cdot)=r(\cdot,n))$ of approximation mappings, where
$\mathcal{AR}$ is  the collection of all finite approximations to members of $\mathcal{R}$.
For $a\in\mathcal{AR}$ and $A,B\in\mathcal{R}$,
\begin{equation}
[a,B]=\{A\in\mathcal{R}:A\le B\mathrm{\ and\ }(\exists n)\ r_n(A)=a\}.
\end{equation}

For $a\in\mathcal{AR}$, let $|a|$ denote the length of the sequence $a$.  Thus, $|a|$ equals the integer $k$ for which $a=r_k(a)$.
For $a,b\in\mathcal{AR}$, $a\sqsubseteq b$ if and only if $a=r_m(b)$ for some $m\le |b|$.
$a\sqsubset b$ if and only if $a=r_m(b)$ for some $m<|b|$.
For each $n<\om$, $\mathcal{AR}_n=\{r_n(A):A\in\mathcal{R}\}$.
\vskip.1in

\begin{enumerate}
\item[\bf A.1]\rm
\begin{enumerate}
\item
$r_0(A)=\emptyset$ for all $A\in\mathcal{R}$.\vskip.05in
\item
$A\ne B$ implies $r_n(A)\ne r_n(B)$ for some $n$.\vskip.05in
\item
$r_n(A)=r_m(B)$ implies $n=m$ and $r_k(A)=r_k(B)$ for all $k<n$.\vskip.1in
\end{enumerate}
\item[\bf A.2]\rm
There is a quasi-ordering $\le_{\mathrm{fin}}$ on $\mathcal{AR}$ such that\vskip.05in
\begin{enumerate}
\item
$\{a\in\mathcal{AR}:a\le_{\mathrm{fin}} b\}$ is finite for all $b\in\mathcal{AR}$,\vskip.05in
\item
$A\le B$ iff $(\forall n)(\exists m)\ r_n(A)\le_{\mathrm{fin}} r_m(B)$,\vskip.05in
\item
$\forall a,b,c\in\mathcal{AR}[a\sqsubset b\wedge b\le_{\mathrm{fin}} c\ra\exists d\sqsubset c\ a\le_{\mathrm{fin}} d]$.\vskip.1in
\end{enumerate}
\end{enumerate}

We abuse notation and 
for $a\in\mathcal{AR}$ and $A\in\mathcal{R}$, we  write 
$a\le_{\mathrm{fin}}A$ to denote that  there is some $n<\om$ such that $a\le_{\mathrm{fin}} r_n(A)$.
$\depth_B(a)$ denotes the least $n$, if it exists, such that $a\le_{\mathrm{fin}}r_n(B)$.
If such an $n$ does not exist, then we write $\depth_B(a)=\infty$.
If $\depth_B(a)=n<\infty$, then $[\depth_B(a),B]$ denotes $[r_n(B),B]$.

\begin{enumerate}
\item[\bf A.3] \rm
\begin{enumerate}
\item
If $\depth_B(a)<\infty$ then $[a,A]\ne\emptyset$ for all $A\in[\depth_B(a),B]$.\vskip.05in
\item
$A\le B$ and $[a,A]\ne\emptyset$ imply that there is $A'\in[\depth_B(a),B]$ such that $\emptyset\ne[a,A']\sse[a,A]$.\vskip.1in
\end{enumerate}
\end{enumerate}

If $n>|a|$, then  $r_n[a,A]$ denotes the collection of all $b\in\mathcal{AR}_n$ such that $a\sqsubset b$ and $b\le_{\mathrm{fin}} A$.

\begin{enumerate}
\item[\bf A.4]\rm
If $\depth_B(a)<\infty$ and if $\mathcal{O}\sse\mathcal{AR}_{|a|+1}$,
then there is $A\in[\depth_B(a),B]$ such that
$r_{|a|+1}[a,A]\sse\mathcal{O}$ or $r_{|a|+1}[a,A]\sse\mathcal{O}^c$.\vskip.1in
\end{enumerate}

The topology on $\mathcal{R}$ is given by the basic open sets
$[a,B]$.
This topology is called the  {\em Ellentuck} topology on $\mathcal{R}$;
it extends the usual metrizable topology on $\mathcal{R}$ when we consider $\mathcal{R}$ as a subspace of the Tychonoff cube $\mathcal{AR}^{\bN}$.
Given the Ellentuck topology on $\mathcal{R}$,
the notions of nowhere dense, and hence of meager are defined in the natural way.
Thus, we may say that a subset $\mathcal{X}$ of $\mathcal{R}$ has the {\em property of Baire} iff $\mathcal{X}=\mathcal{O}\cap\mathcal{M}$ for some Ellentuck open set $\mathcal{O}\sse\mathcal{R}$ and Ellentuck meager set $\mathcal{M}\sse\mathcal{R}$.

\begin{defn}[\cite{TodorcevicBK10}]\label{defn.5.2}
A subset $\mathcal{X}$ of $\mathcal{R}$ is {\em Ramsey} if for every $\emptyset\ne[a,A]$,
there is a $B\in[a,A]$ such that $[a,B]\sse\mathcal{X}$ or $[a,B]\cap\mathcal{X}=\emptyset$.
$\mathcal{X}\sse\mathcal{R}$ is {\em Ramsey null} if for every $\emptyset\ne [a,A]$, there is a $B\in[a,A]$ such that $[a,B]\cap\mathcal{X}=\emptyset$.

A triple $(\mathcal{R},\le,r)$ is a {\em topological Ramsey space} if every property of Baire subset of $\mathcal{R}$ is Ramsey and if every meager subset of $\mathcal{R}$ is Ramsey null.
\end{defn}

The following result can be found as Theorem
5.4 in \cite{TodorcevicBK10}.

\begin{thm}[Abstract Ellentuck Theorem]\label{thm.AET}\rm \it
If $(\mathcal{R},\le,r)$ is closed (as a subspace of $\mathcal{AR}^{\bN}$) and satisfies axioms {\bf A.1}, {\bf A.2}, {\bf A.3}, and {\bf A.4},
then every property of Baire subset of $\mathcal{R}$ is Ramsey,
and every meager subset is Ramsey null;
in other words,
the triple $(\mathcal{R},\le,r)$ forms a topological Ramsey space.
\end{thm}

For a topological Ramsey space,
certain types of subsets of the collection of approximations $\mathcal{AR}$  have the Ramsey property.

\begin{defn}[\cite{TodorcevicBK10}]\label{defn.5.16}
A family $\mathcal{F}\sse\mathcal{AR}$ of finite approximations is
\begin{enumerate}
\item
{\em Nash-Williams} if $a\not\sqsubseteq b$ for all $a\ne b\in\mathcal{F}$;
\item
{\em Sperner} if $a\not\le_{\mathrm{fin}} b$ for all $a\ne b\in\mathcal{F}$;
\item
{\em Ramsey} if for every partition $\mathcal{F}=\mathcal{F}_0\cup\mathcal{F}_1$ and every $X\in\mathcal{R}$,
there are $Y\le X$ and $i\in\{0,1\}$ such that $\mathcal{F}_i|Y=\emptyset$.
\end{enumerate}
\end{defn}

The Abstract Nash-Williams Theorem (Theorem 5.17  in \cite{TodorcevicBK10}), which follows from the Abstract Ellentuck Theorem,
will suffice for the arguments in this paper.

\begin{thm}[Abstract Nash-Williams Theorem]\label{thm.ANW}
Suppose $(\mathcal{R},\le,r)$ is a closed triple that satisfies {\bf A.1} - {\bf A.4}. Then
every Nash-Williams family of finite approximations is Ramsey.
\end{thm}

\begin{defn}\label{def.frontR1}
Suppose $(\mathcal{R},\le,r)$ is a closed triple that satisfies {\bf A.1} - {\bf A.4}.
Let $X\in\mathcal{R}$.
A family $\mathcal{F}\sse\mathcal{AR}$ is a {\em front} on $[0,X]$ if
\begin{enumerate}
\item
For each $Y\in[0,X]$, there is an $a\in \mathcal{F}$ such that $a\sqsubset Y$; and
\item
$\mathcal{F}$ is Nash-Williams.
\end{enumerate}
$\mathcal{F}$ is a {\em barrier}  if (1) and ($2'$) hold,
where
\begin{enumerate}
\item[(2$'$)]
$\mathcal{F}$ is Sperner.
\end{enumerate}
\end{defn}

The quintessential example of a topological Ramsey space is the {\em Ellentuck space}, which  is the triple $([\om]^{\om},\sse,r)$.
Members $X\in[\om]^{\om}$ are considered as infinite increasing sequences of natural numbers, $X=\{x_0,x_1,x_2,\dots\}$.
For each $n<\om$, the $n$-th approximation to $X$ is $r_n(X)=\{x_i:i<n\}$; in particular, $r_0(X)=\emptyset$.
The basic open sets of the Ellentuck topology are sets of the form 
$[a,X]=\{Y\in [\om]^{\om}: a\sqsubset Y$ and $Y\sse X\}$.
Notice that the Ellentuck topology is finer than the metric topology on $[\om]^{\om}$.

In the case of the Ellentuck space,
the Abstract Ellentuck Theorem says the following:
Whenever a subset $\mathcal{X}\sse[\om]^{\om}$
has the property of Baire in the Ellentuck topology, 
then that set is {\em Ramsey},
meaning that every open set contains a basic open set either contained in $\mathcal{X}$ or else disjoint from $\mathcal{X}$.
This was proved  by Ellentuck in \cite{Ellentuck74}.

The first theorem to extend Ramsey's Theorem from finite-valued functions to countably infinite-valued functions was a theorem of \Erdos\ and Rado.  
They found that in fact, given any equivalence relation on $[\om]^n$,
there is an infinite subset on which the equivalence relation is canonical - one of exactly $2^n$ many equivalence relations.
We shall state the finite version of their theorem, as it is all that is used in this paper (see Section \ref{sec.ProductER}).

Let $n\le l$.
For each $I\sse n$,
the equivalence relation $E_I$ on $[l]^n$ is defined as follows:
For $b,c\in [l]^n$,  
$$b\ E_I\ c\Llra \forall i\in I (b_i=c_i),$$
where $\{b_0,\dots,b_{n-1}\}$ and $\{c_0,\dots,c_{n-1}\}$ are the strictly increasing enumerations of $b$ and $c$, respectively.
An equivalence relation $E$ on $[l]^n$ is {\em canonical} if and only if there is some $I\sse n$ for which $E= E_I$.

\begin{thm}[Finite \Erdos-Rado Theorem, \cite{Erdos/Rado50}]\label{thm.finiteER}
Given $n\le l$, there is an $m>l$ such that for each equivalence relation $E$ on $[m]^n$,
there is a subset $s\sse m$ of cardinality $l$ such that $E\re [s]^n$ is canonical;
that is, there is a set $I\sse n$ such that $E\re [s]^n=E_I\re [s]^n$.
\end{thm}

\Pudlak\ and \Rodl\ later extended this theorem to equivalence relations on general barriers on the Ellentuck space.
To state their theorem, we need the following definition.

\begin{defn}\label{def.frontEllentuck}
A map $\vp$ from a front $\mathcal{F}$ on the Ellentuck space into $\om$ is called {\em irreducible} if 
\begin{enumerate}
\item $\vp$ is {\em inner}, meaning that $\vp(a)\sse a$ for all $a\in\mathcal{F}$; and
\item$\vp$ is {\em Nash-Williams}, meaning that $\vp(a)\not\sqsubset\vp(b)$ for all $a,b\in\mathcal{F}$ such that $\vp(a)\ne\vp(b)$.
\end{enumerate}
\end{defn}

Given a front $\mathcal{F}$ and an $X\in[\om]^{\om}$,
we let $\mathcal{F}\re X$ denote $\{a\in \mathcal{F}:a\sse X\}$.
Given an equivalence relation $E$ on a barrier $\mathcal{F}$,
we say that an irreducible map $\vp$ {\em represents} $E$ on $\mathcal{F}\re X$ if for all $a,b\in\mathcal{F}\re X$, we have $a\ E\ b\lra \vp(a)=\vp(b)$.

The following theorem of \Pudlak\ and \Rodl\ is the basis for all subsequent canonization theorems for fronts on the general topological Ramsey spaces considered in the papers
\cite{Dobrinen/Todorcevic11} and \cite{Dobrinen/Todorcevic12}.

\begin{thm}[\Pudlak/\Rodl, \cite{Pudlak/Rodl82}]\label{thm.PR}
For any barrier $\mathcal{F}$ on the Ellentuck space and any equivalence relation on $\mathcal{F}$,
there is an $X\in[\om]^{\om}$ and an irreducible map $\vp$ such that
the equivalence relation restricted to $\mathcal{F}\re X$ is represented by $\vp$.
\end{thm}

Theorem \ref{thm.PR} was generalized to a class of topological Ramsey spaces whose members are trees
in  \cite{Dobrinen/Todorcevic11} and  \cite{Dobrinen/Todorcevic12}.
In Section \ref{sec.canon},
we shall generalize this theorem to the broad class of topological Ramsey  spaces  defined in the next section.


\section{A general method for constructing  topological Ramsey spaces\\ using \Fraisse\ theory}\label{sec.genseq}

We review only the facts of \Fraisse\ theory for ordered  relational structures which are necessary to this article.
More general background  on \Fraisse\ theory can be found in \cite{Kechris/Pestov/Todorcevic05}.
We shall call $L=\{<\}\cup\{R_i\}_{i\in I}$  an {\em ordered relational signature} if it consists of the order relation symbol $<$ and  a (countable) collection of  {\em relation symbols} $R_i$, where for each $i\in I$, we let $n(i)$ denote the {\em arity} of $R_i$.
A {\em structure} for $L$ is of the form $\bsA=\lgl |\bsA|,<^{\bsA}, \{R_i^{\bsA}\}_{i\in I}\rgl$,
where $|\bsA|\ne\emptyset$ is the {\em universe} of $\bsA$, $<^{\bsA}$ is a linear ordering of $|\bsA|$, and for each $i\in I$, $R_i^{\bsA}\sse A^{n(i)}$.
An {\em embedding} between structures $\bsA,\bsB$ for $L$ is an injection $\iota:|\bsA|\ra|\bsB|$ such that
for any two $a,a'\in|\bsA|$,
$a <^{\bsA} a'\leftrightarrow \iota(a) <^{\bsB}\iota(a')$,
 and for all $i\in I$, $R_i^{|\bsA|}(a_1,\dots,a_{n(i)})\leftrightarrow R_i^{|\bsB|}(\iota(a_1),\dots,\iota(a_{n(i)}))$.
If $\iota$ is the identity map, then we say that $\bsA$ is a {\em substructure} of $\bsB$.
We say that $\iota$ is an {\em isomorphism} if $\iota$ is an onto embedding.
We write  $\bsA\le\bsB$ to denote that $\bsA$ can be embedded into $\bsB$; and we  write $\bsA\cong\bsB$ to denote that $\bsA$ and $\bsB$ are isomorphic.

A class $\mathcal{K}$ of finite structures for an ordered relational signature $L$ is called {\em hereditary} if whenever $\bsB\in\mathcal{K}$ and  $\bsA\le\bsB$, then also $\bsA\in\mathcal{K}$.
$\mathcal{K}$ satisfies the {\em joint embedding property} if for any $\bsA,\bsB\in\mathcal{K}$,
there is a $\bsC\in\mathcal{K}$ such that $\bsA\le\bsC$ and $\bsB\le\bsC$.
We say that $\mathcal{K}$ satisfies the {\em amalgamation property} if for any embeddings 
$f:\bsA\ra\bsB$ and $g:\bsA\ra\bsC$, with $\bsA,\bsB,\bsC\in\mathcal{K}$,
there is a $\bsD\in\mathcal{K}$ and  there are embeddings $r:\bsB\ra\bsD$ and $s:\bsC\ra\bsD$ such that
$r\circ f = s\circ g$.
$\mathcal{K}$ satisfies the {\em strong amalgamation property}   there are embeddings
 $r:\bsB\ra\bsD$ and $s:\bsC\ra\bsD$ such that
$r\circ f = s\circ g$ and additionally,
$r(\bsB)\cap s(\bsC)=r\circ f(\bsA)= s\circ g(\bsA)$.
A class of finite structures $\mathcal{K}$ is called a {\em \Fraisse\ class of ordered relational structures} for an ordered  relational signature $L$ if it is hereditary, satisfies the joint embedding and amalgamation properties, contains (up to isomorphism) only countably many structures, and contains structures of arbitrarily large finite cardinality.

Let $\mathcal{K}$ be a hereditary class of finite structures for an ordered  relational signature $L$.
For $\bsA,\bsB\in\mathcal{K}$ with $\bsA\le\bsB$, we use ${\bsB\choose\bsA}$ to denote the set of all substructures of $\bsB$ which are isomorphic to $\bsA$.
Given structures $\bsA\le\bsB\le\bsC$ in $\mathcal{K}$, we write
$$
\bsC\ra(\bsB)_k^{\bsA}
$$
to denote that for each coloring of ${\bsC\choose \bsA}$ into $k$ colors, there is a $\bsB' \in {\bsC\choose\bsB}$ such that
${\bsB'\choose\bsA}$ is homogeneous, i.e.\ monochromatic, meaning that
 every member of ${\bsB'\choose\bsA}$ has the same color.
We say that $\mathcal{K}$ has the {\em Ramsey property} if and only if for any two structures $\bsA\le\bsB$ in $\mathcal{K}$ and any natural number $k\ge 2$,
there is a $\bsC\in\mathcal{K}$ with $\bsB\le\bsC$ such that 
$\bsC\ra (\bsB)^{\bsA}_k$.

For finitely many \Fraisse\ classes $\mathcal{K}_j$, $j\in J$ for some $J<\om$, 
we write ${(\bsB_j)_{j\in J} \choose (\bsA_j)_{j\in J}}$ to denote the set of all sequences $(\bsD_j)_{j\in J}$ such that for each $j\in J$, $\bsD_j\in{\bsB_j\choose\bsA_j}$.
For structures $\bsA_j\le\bsB_j\le\bsC_j\in\mathcal{K}_j$, $j\in J$, we write 
$$
(\bsC_j)_{j\in J}\ra ((\bsB_j)_{j\in J})_k^{(\bsA_j)_{j\in J}}
$$
to denote that for each coloring of the members of   ${(\bsC_j)_{j\in J}\choose (\bsA_j)_{j\in J}}$ into $k$ colors, there is $(\bsB'_j)_{j\in J}\in{(\bsC_j)_{j\in J}\choose (\bsB_j)_{j\in J}}$ such that 
all members of ${(\bsB'_j)_{j\in J}\choose(\bsA_j)_{j\in J}}$ have the same color; that is, the set 
 ${(\bsB'_j)_{j\in J}\choose(\bsA_j)_{j\in J}}$ is homogeneous.
We subscribe to the  usual convention that when no $k$ appears in the expression,  it is assumed that $k=2$.

We point out that by Theorem A of \Nesetril\ and \Rodl\ in \cite{Nesetril/Rodl83}, there is a large class of \Fraisse\ classes of finite ordered relational structures with the Ramsey property.
In particular, the collection of all finite linear orderings,
the collection of all finite ordered $n$-clique free graphs, 
and the collection of all finite ordered complete graphs are examples of \Fraisse\ classes fulfilling our requirements.
Moreover, finite products of members of such classes preserve the Ramsey property, as we now see.
The following  theorem for products of Ramsey classes of finite objects is due to \Sokic\ and can be found in his PhD thesis. 

\begin{thm}[Product Ramsey Theorem, \Sokic\ \cite{SokicPhDThesis}]\label{thm.Sokic}
Let $s$ and $k$ be fixed natural numbers and let $\mathcal{K}_j$, $j\in s$, be a sequence of Ramsey classes of finite objects.
Fix two sequences $(\bsB_j)_{j\in s}$ and $(\bsA_j)_{j\in s}$ such that for each $j\in s$, we have $\bsA_j,\bsB_j\in\mathcal{K}_j$ and $\bsA_j\le\bsB_j$.
Then there is a sequence $(\bsC_j)_{j\in s}$ such that $\bsC_j\in\mathcal{K}_j$ for each $j\in s$, 
and
$$
(\bsC_j)_{j\in s}\ra ((\bsB)_{j\in s})_k^{(\bsA_j)_{j\in s}}.
$$
\end{thm}

We now present our notion of a {\em generating sequence}.
Such sequences will be used to generate new topological Ramsey spaces.

\begin{defn}[Generating Sequence]\label{defn.A_k}
Let $1\le J\le\om$ and
$\mathcal{K}_j$, $j\in J$, be a collection
of  \Fraisse\ classes  of  finite ordered relational structures with   the Ramsey property.
For each $k\in\om$,
if $J<\om$ then let $J_k=J$, and if $J=\om$ then let $J_k=k+1$.

For each $k<\om$ and $j\in J_k$,
 suppose $\bsA_{k,j}$ is some fixed member of $\mathcal{K}_j$,
and
let $\bsA_k$ denote the sequence $(\bsA_{k,j})_{j\in J_k}$.
We say that 
$\lgl \bsA_k:k<\om\rgl$ is a {\em generating sequence}  
if and only if
\begin{enumerate}
\item
For each $j\in J_0$, $|\bsA_{0,j}|=1$.
\item
For each $k<\om$ and all $j\in J_k$,
$\bsA_{k,j}$ is a substructure of  $\bsA_{k+1,j}$.
\item
For each $j\in J$ and each structure $\bsB\in\mathcal{K}_j$,
there is a $k$ such that 
$\bsB\le \bsA_{k,j}$.
\item (Pigeonhole)
For each pair $k<m<\om$,  there is an $n>m$ such that 
$$
(\bsA_{n,j})_{j\in J_k}\ra {(\bsA_{m,j})_{j\in J_k}}^{(\bsA_{k,j})_{j\in J_k}}.
$$
\end{enumerate}
\end{defn}

\begin{rem}
Note that (3) implies that for each $j\in J$ and each $\bsB\in\mathcal{K}_j$, $\bsB\le \bsA_{k,j}$ for all but finitely many $k$.
\end{rem}

We now define the new class of topological Ramsey spaces which are the focus of this article.

\begin{defn}[The spaces $\mathcal{R}(\lgl\bsA_k:k<\om\rgl)$]\label{def.XinR}
Let $1\le J\le\om$ and
$\mathcal{K}_j$, $j\in J$, be a collection
of  \Fraisse\ classes  of  finite ordered relational structures with  the Ramsey property.
Let $\lgl \bsA_k:k<\om\rgl$ be any generating sequence.
Let $\bA=\lgl \lgl k,\bsA_k\rgl:k<\om\rgl$.
$\bA$ is the maximal member of $\mathcal{R}(\lgl\bsA_k:k<\om\rgl)$.

We define
$B$ to be a member of $\mathcal{R}(\lgl\bsA_k:k<\om\rgl)$ if and only if
$B=\lgl \lgl n_k, \bsB_k\rgl:k<\om\rgl$,
where 
\begin{enumerate}
\item
$(n_k)_{k<\om}$ is some strictly increasing sequence of natural numbers;   and 
\item
For each $k<\om$, $\bsB_k$ is some sequence $(\bsB_{k,j})_{j\in J_k}$, where 
for each $j\in J_k$,
 $\bsB_{k,j}\in{\bsA_{n_k,j}\choose \bsA_{k,j}}$.
\end{enumerate}
We use $B(k)$ to denote  $\lgl n_k,\bsB_k\rgl$, the {\em $k$-th block} of $B$.
Let $\mathcal{R}(k)$ denote $\{B(k):B\in\mathcal{R}(\lgl\bsA_k:k<\om\rgl)\}$,
the collection of all  $k$-th blocks of members of $\mathcal{R}(\lgl\bsA_k:k<\om\rgl)$.
The $n$-th approximation of $B$ is
$r_n(B):=\lgl B(0),\dots,B(n-1)\rgl$.
In particular, $r_0(B)=\emptyset$.
Let 
$\mathcal{AR}_n=\{r_n(B):B\in\mathcal{R}(\lgl\bsA_k:k<\om\rgl)\}$, 
the collection of all $n$-th approximations to members of $\mathcal{R}(\lgl\bsA_k:k<\om\rgl)$.
Let  $\mathcal{AR}=\bigcup_{n<\om}\mathcal{AR}_n$,
the collection of all finite approximations to members of $\mathcal{R}(\lgl\bsA_k:k<\om\rgl)$.

Define the partial order  $\le$ on $\mathcal{R}(\lgl\bsA_k:k<\om\rgl)$ as follows.
For $B=\lgl \lgl m_k,\bsB_k\rgl:k<\om\rgl$ and $C=\lgl\lgl n_k,\bsC_k\rgl:k<\om\rgl$,
 define
$C\le B$ if and only if 
for each $k$ there is an $l_k$ such that 
$n_k=m_{l_k}$
and
 for all $j\in J_k$,  $\bsC_{k,j}\in{\bsB_{l_k,j} \choose \bsA_{k,j}}$.

Define the partial order $\le_{\mathrm{fin}}$ on $\mathcal{AR}$ as follows:
For $b=\lgl \lgl m_k,\bsB_k\rgl:k<p\rgl$ and $c=\lgl\lgl n_k,\bsC_k\rgl:k<q\rgl$,
define $c\le_{\mathrm{fin}} b$ if and only if there are $C\le B$ and  $k\le l$ such that $c=r_q(C)$,
$b=r_p(B)$, and for each $k<q$, $n_k= m_{l_k}$ for some $l_k<p$.
\end{defn}

For $c\in\mathcal{AR}$ and $B\in \mathcal{R}$,
$\depth_B(c)$ denotes the minimal $d$ such that $c\le_{\mathrm{fin}} r_d(B)$, if such a $d$ exists; otherwise $\depth_B(c)=\infty$.
Note that  for 
$c=\lgl\lgl n_k,\bsC_k\rgl:k<q\rgl$,
$\depth_{\mathbb A}(c)$ is equal to  $n_{q-1}+1$.
The {\em length} of $c$, denoted by $|c|$, is the minimal $q$ such that $c= r_q(c)$.
For $b,c\in\mathcal{AR}$, we write $b\sqsubseteq c$ if and only if there is a $p\le |c|$ such that $b=r_p(c)$.
In this case, we say that $b$ is an {\em initial segment} of $c$.
We use $b\sqsubset c$ to denote that $b$ is a proper initial segment of $c$; that is $b\sqsubseteq c$ and $b\ne c$.

\begin{rem}
The members of $\mathcal{R}(\lgl \bsA_k:k<\om\rgl)$ are infinite squences $B$ which are isomorphic to the maximal member $\bA$, in the sense that for each $k$-th block $B(k)=\lgl n_k, \bsB_k\rgl$, each of the structures $\bsB_{k,j}$ is isomorphic to $\bsA_{k,j}$.
This idea, of forming  a topological Ramsey space by taking the collection of all infinite sequences coming from within some fixed sequence and preserving the same form as this fixed sequence,
is extracted from the Ellentuck space itself, and was
first extended to more generality in \cite{Dobrinen/Todorcevic11}.
\end{rem}

The above method of construction yields a new class of topological Ramsey spaces. The proof below is jointly written with Trujillo.

\begin{thm}\label{thm.tRs}
Let $1\le J\le\om$ and
$\mathcal{K}_j$, $j\in J$, be a collection
of  \Fraisse\ classes  of  finite ordered relational structures with  the Ramsey property.
For each generating sequence $\lgl \bsA_k:k<\om\rgl$,
the space $(\mathcal{R}(\lgl \bsA_k:k<\om\rgl),\le,r)$ 
satisfies axioms \bf A.1 \rm - \bf A.4 \rm and is closed in $\mathcal{AR}^{\om}$,
and hence,
is a
 topological Ramsey space.
\end{thm}

\begin{proof}
Let $\mathcal{R}$ denote $\mathcal{R}(\lgl \mathbf{A}_{k}:k<\omega\rgl)$. 
$\mathcal{R}$ is identified with the subspace of the Tychonov power $\mathcal{AR}^{\omega}$ consisting of all sequences $\left<a_{n}, n<\omega\right>$ for which there is a $B\in \mathcal{R}$ such that for each $n<\omega$, $a_{n} = r_{n}(B)$. 
$\mathcal{R}$ forms a closed subspace of $\mathcal{AR}^{\omega}$, since for each sequence $\left<a_{n}, n<\omega\right>$ with the properties that  each  $a_{n} \in \mathcal{AR}_{n}$
and $a_n\sqsubset a_{n+1}$,
 then $\left< \lgl\depth_{A} (a_{n+1}), a_{n+1}(n)\rgl : n<\omega \right>$ is a member of $\mathcal{R}$.
It is routine to check that axioms \bf A.1 \rm  and  \bf A.2 \rm hold.

${\bf A.3}$
(1) If $\depth_{B}(a) = n <\omega$, then $a \le_{\fin} r_{n}(B)$. 
If $C\in [\depth_{B}(a), B]$, then $r_{n}(B)=r_{n}(C)$ and for each $k>n$, there is an $m_{k}$ such that
$( \mathbf{C}_{m_{k},j})_{j\in J_{k}}\le
 (\mathbf{B}_{k,j})_{j\in J_{k}}$. 
For each $i\ge|a|$, let $D(i)$ be an element of $\mathcal{R}(i)$ such that  $(\mathbf{D}_{i,j})_{j\in J_{i}}$ is a substructure of $(\mathbf{C}_{i,j})_{j\in J_{i}}$ isomorphic to $(\mathbf{A}_{i, j})_{j\in J_{i}}$.
 Let $D = a ^{\frown}  \left< D(i) :  |a|\le i<\om \right> \in \mathcal{R}$. 
Then $ D \in [a, B]$, so $[a,B] \not = \emptyset$.

(2) Suppose that $B\le C$ and $[a,B]\not= \emptyset$. Let $n=\depth_{C}(a)$. Then $n <\infty$ since $B\le C$. Let $D= r_{n}(C) ^{\frown} \left< B(n+i): i<\omega\right>$. Then $D \in [ \depth_{C}(a), C]$ and $\emptyset \not=[a,D]\subseteq[a,B]$.

${\bf A.4}$
Suppose that 
$B = \left< (n_k, \mathbf{B}_k) : k<\omega\right>$,
$\depth_{B}(a)<\infty$, and $\mathcal{O}\subseteq \mathcal{AR}_{|a|+1}$.
Let $n=|a|$.
By (4) in the definition of a generating sequence, there is  a strictly increasing sequence $(k_i)_{i\ge n}$ such  that 
$(\bsA_{k_i,j})_{j\in J_n}\ra {(\bsA_{i,j})_{j\in J_n}\choose
(\bsA_{n,j})_{j\in J_n}}$, for each $i\ge n$.
For each $i\ge n$, choose some $(\bsC_{i,j})_{j\in J_i}$ in ${(\bsB_{k_i,j})_{j\in J_i}\choose (\bsA_{i,j})_{j\in J_i}}$
such that the collection
$$
\{\lgl  n_{k_i}, (\bsX_{i,j})_{j\in J_n}\rgl:
 (\bsX_{i,j})_{j\in J_n}\in {(\bsC_{i,j})_{j\in J_n}\choose (\bsA_{n,j})_{j\in J_n}}\}$$
 is homogeneous for $\mathcal{O}$.
Infinitely many of these $(\bsC_{i,j})_{j\in J_i}$ will agree about being in or out of $\mathcal{O}$.
Thus, for some subsequence
$(k_{i_l})_{l\ge n}$,
there are 
 $(\bsD_{l,j})_{j\in J_l}\in {(\bsC_{i_l,j})_{j\in J_l}\choose (\bsA_{l,j})_{j\in J_l}}$
such that
letting $D=a^{\frown}\lgl \lgl
n_{k_{i_l}}
 ,(\bsD_{l,j})_{j\in J_l}\rgl:l\ge n\rgl$,
we have that
 $r_{n+1}[a,D]$ is either contained in or disjoint from $\mathcal{O}$.
\end{proof}

We fix the following notation, which is used throughout this paper.

\begin{notationn}\label{notn.lots}
For $a\in\mathcal{AR}$ and $B\in\mathcal{R}$,
we write $a\le_{\fin} B$ to mean that there is some $A\in\mathcal{R}$ such that $A\le B$ and $a=r_n(A)$ for some $n$.
For $\mathcal{H}\sse\mathcal{AR}$ and $B\in\mathcal{R}$,
let $\mathcal{H}|B$ denote the collection of all $a\in\mathcal{H}$ such that $a\le_{\fin} B$.

For $n<\om$,
$\mathcal{R}(n)=\{C(n):C\in\mathcal{R}\}$, and 
$\mathcal{R}(n)|B=\{C(n):C\le B\}$.
$B/a$ denotes the tail of $B$ which is above every block in $a$.
$\mathcal{R}(n)|B/a$ denotes the members of $\mathcal{R}(n)|B$ which are above $a$.
\end{notationn}


\section{Ultrafilters associated with topological Ramsey spaces
constructed from \\ generating sequences
and their partition properties}\label{sec.uf}

In this section, we show that many examples of ultrafilters satisfying partition properties can be 
seen to arise as ultrafilters associated with some
topological Ramsey spaces constructed from  a generating sequence.
In particular, the ultrafilters of Baumgartner and Taylor  in Section 4 of \cite{Baumgartner/Taylor78} arising from norms fit into this framework.
We begin by reviewing some 
 important types of ultrafilters.
All of the following definitions  can found in \cite{Bartoszynski/JudahBK}.
Recall the standard notation $\sse^*$,
where for $X,Y\sse\om$, we write
$X\sse^* Y$ to denote that $|X\setminus Y|<\om$.

\begin{defn}\label{defn.uftypes}
Let $\mathcal{U}$ be a nonprincipal ultrafilter.
\begin{enumerate}
\item
$\mathcal{U}$ is {\em selective} if 
for every function $f:\om\ra\om$,
there is an $X\in\mathcal{U}$ such that either $f\re X$ is constant or $f\re X$ is one-to-one.
\item
$\mathcal{U}$ is {\em Ramsey} if for each $2$-coloring $f:[\om]^2\ra 2$, there is  an $X\in\mathcal{U}$ such that $f\re [X]^2$ takes on exactly one color.
This is denoted by $\om\ra(\mathcal{U})^2$.
\item
$\mathcal{U}$ is a {\em p-point} if for every family $\{X_n:n<\om\}\sse\mathcal{U}$
there is an $X\in\mathcal{U}$ such that $X\sse^* X_n$ for each $n<\om$.
\item
$\mathcal{U}$ is a {\em q-point} if for each partition of $\om$ into finite pieces $\{I_n:n<\om\}$,
there is an $X\in\mathcal{U}$ such that $|X\cap I_n|\le 1$ for each $n<\om$.
\item
$\mathcal{U}$ is {\em rapid} if 
for each function $f:\om\ra\om$,
there exists an $X\in\mathcal{U}$ such that $|X\cap f(n)|\le n$ for each $n<\om$.
\end{enumerate}
\end{defn}

It is well-known that for ultrafilters on $\om$, being Ramsey is equivalent to being selective, and that an ultrafilter is Ramsey if and only if it is both a p-point and a q-point.
Every q-point is rapid.

Let $(\mathcal{R},\le, r)$ be any topological Ramsey space.
Recall that a subset $\mathcal{C}\sse\mathcal{R}$ is a
{\em filter} on $(\mathcal{R},\le)$ 
if $\mathcal{C}$ is {\em closed upwards}, meaning that whenever $X\in\mathcal{C}$ and $X\le Y$, then also $Y\in\mathcal{C}$;
and for every pair $X,Y\in\mathcal{C}$, there is a $Z\in\mathcal{C}$ such that $Z\le X,Y$.

\begin{defn}\label{def.Ramseymaxfilter}
A filter $\mathcal{C}$ on a topological Ramsey space $\mathcal{R}$ is called {\em Ramsey for $\mathcal{R}$}
if $\mathcal{C}$ is a maximal filter and
for each $n<\om$ and each $\mathcal{H}\sse\mathcal{AR}_n$,
there is a member $C\in\mathcal{C}$ such that  either 
$\mathcal{AR}_n|C\sse\mathcal{H}$ or else
$\mathcal{AR}_n|C\cap\mathcal{H}=\emptyset$.
\end{defn}

Note that a filter which is Ramsey for $\mathcal{R}$ is a  maximal filter on $(\mathcal{R},\le)$, meaning that  for each $X\in\mathcal{R}\setminus \mathcal{C}$,
the filter generated by $\mathcal{C}\cup\{X\}$ is all of $\mathcal{R}$.

\begin{fact}\label{fact.Ramseyuf}
Let $\lgl\bsA_n:n<\om\rgl$ be any generating sequence with $1\le J<\om$.
Each  filter $\mathcal{C}$ which is Ramsey for $\mathcal{R}(\lgl\bsA_n:n<\om\rgl)$ generates an ultrafilter on the base set $\mathcal{AR}_1$,
namely the ultrafilter, denoted $\mathcal{U}_{\mathcal{R}}$, generated by the collection $\{\mathcal{AR}_1|C:C\in\mathcal{C}\}$.
\end{fact}

\begin{proof}
Let $\mathcal{U}$ denote the collection of $\mathcal{G}\sse\mathcal{AR}_1$ such that $\mathcal{G}\contains \mathcal{AR}_1|C$ for some $C\in\mathcal{C}$.
Certainly $\mathcal{U}$ is a filter on $\mathcal{AR}_1$, since $\mathcal{C}$ is a filter on $\mathcal{R}(\lgl\bsA_n:n<\om\rgl)$.
To see that $\mathcal{U}$ is an ultrafilter,
let $\mathcal{H}\sse\mathcal{AR}_1$ be given.
Since $\mathcal{C}$ is Ramsey for $\mathcal{R}(\lgl\bsA_n:n<\om\rgl)$, 
there is a $C\in\mathcal{C}$ such that either $\mathcal{AR}_1|C\sse\mathcal{H}$ or else $\mathcal{AR}_1|C\cap\mathcal{H}=\emptyset$.
In the first case, $\mathcal{H}\in\mathcal{U}$; in the second case, $\mathcal{AR}_1\setminus\mathcal{H}\in\mathcal{U}$.
\end{proof}

One of the motivations for generating sequences was to provide a construction scheme for ultrafilters which are p-points satisfying some partition relations.
At this point, we show how some historic examples of such ultrafilters can be seen to arise as ultrafilters associated with some topological Ramsey space constructed from a generating sequence, thus providing a general framework for such ultrafilters.

\begin{example}[A weakly Ramsey, non-Ramsey ultrafilter, \cite{Baumgartner/Taylor78}, \cite{Laflamme89}]\label{ex.BT4.84.9}
In \cite{Dobrinen/Todorcevic11} a topological Ramsey space called $\mathcal{R}_1$ was extracted from a forcing of Laflamme which forces a weakly Ramsey ultrafilter which is not Ramsey.
That forcing of Laflamme is the same as the example of 
Baumgartner and Taylor in Theorems 4.8 and 4.9 in \cite{Baumgartner/Taylor78}.
$\mathcal{R}_1$ is exactly $\mathcal{R}(\lgl\bsA_n:n<\om\rgl)$, where 
 each  $\bsA_n=\lgl n, <\rgl$, the linear order of cardinality $n$.
$\mathcal{R}_1$ is dense in the forcing given by Baumgartner and Taylor.
Thus, their ultrafilter can be seen to be generated by the topological Ramsey space $\mathcal{R}_1$.
\end{example}

The next set of examples of ultrafilters which are generated by our topological Ramsey spaces are the $n$-arrow, not $(n+1)$-arrow ultrafilters of Baumgartner and Taylor.

\begin{defn}[\cite{Baumgartner/Taylor78}]\label{defn.arrow}
An ultrafilter $\mathcal{U}$ is {\em $n$-arrow} if $3\le n<\om$ and for every function $f:[\om]^2\ra 2$,
either there exists a set $X\in\mathcal{U}$ such that $f([X]^2)=\{0\}$,
or else there is a set $Y\in [\om]^n$ such that $f([Y]^2)=\{1\}$.
$\mathcal{U}$ is an {\em arrow} ultrafilter if $\mathcal{U}$ is $n$-arrow for each $n\le 3<\om$.
\end{defn}

Theorem 4.11 in \cite{Baumgartner/Taylor78} of Baumgartner and Taylor
shows that for each $2\le n<\om$, there are p-points which are $n$-arrow but not $(n+1)$-arrow.
(By default, every ultrafilter is $2$-arrow.)
As the ultrafilters of Laflamme in \cite{Laflamme89} with partition relations had led to the formation of new topological Ramsey spaces and their analogues of the \Pudlak-\Rodl\ Theorem
in \cite{Dobrinen/Todorcevic11} and \cite{Dobrinen/Todorcevic12}, Todorcevic suggested  that these arrow  ultrafilters  with asymmetric partition relations might lead to interesting new Ramsey-classification theorems.
It turns out that the constructions of Baumgartner and Taylor can be thinned to see that there is a generating sequence with associated topological Ramsey space producing their ultrafilters.
In fact, our idea of using \Fraisse\ classes of relational structures to construct topological Ramsey spaces was gleaned from their theorem.

\begin{example}[Spaces $\mathcal{A}_n$, generating $n$-arrow, not $(n+1)$-arrow p-points]\label{ex.A}
For a fixed $n\ge 2$, let $J=1$ and $\mathcal{K}=\mathcal{K}_0$ denote the \Fraisse\ class of all finite $(n+1)$-clique-free ordered graphs.
By Theorem A  of \Nesetril\ and \Rodl\ in \cite{Nesetril/Rodl83}, $\mathcal{K}$ has the Ramsey property.
Choose any generating sequence $\lgl \bsA_k :k<\om\rgl$.
One can check, by a proof similar to that given in Theorem 4.11 of \cite{Baumgartner/Taylor78}, that any ultrafilter on $\mathcal{AR}_1$ which is Ramsey for $\mathcal{R}(\lgl \bsA_k:k<\om\rgl)$ is an $n$-arrow p-point which is not $(n+1)$-arrow.

Let $\mathcal{U}_{\mathcal{A}_n}$ denote any ultrafilter on $\mathcal{AR}_1$ which is Ramsey for $\mathcal{A}_n$.
It will follow from 
Theorem \ref{TukeyReducibleTheorem2} 
that the initial Rudin-Keisler structure of the p-points Tukey reducible to $\mathcal{U}_{\mathcal{A}_n}$
is exactly that of the collection of isomorphism classes of members of $\mathcal{K}_0$, partially ordered by embedability.
Further,
Theorem \ref{Tim-InitialStructures}
  will show 
that the initial Tukey structure below $\mathcal{U}_{\mathcal{A}_n}$ is exactly a chain of length 2.
\end{example}

\begin{rem}
In fact,  Theorem A in \cite{Nesetril/Rodl83}  of \Nesetril\ and \Rodl\  provides a large collection of  \Fraisse\ classes
of finite ordered relational structures which omit subobjects which are irreducible.
Generating sequences can be taken from any of these, 
resulting in  new
 topological Ramsey spaces and associated ultrafilters.
(See \cite{Nesetril/Rodl83} for the relevant definitions.)
\end{rem}

The next  collection of topological Ramsey spaces we will call {\em hypercube spaces}, $\mathcal{H}^n$, $1\le n<\om$.
The idea for the space $\mathcal{H}^2$ was gleaned from Theorem 9 of Blass in \cite{Blass73}, where he shows that, assuming Martin's Axiom, there is a p-point with two Rudin-Keisler incomparable p-points Rudin-Keisler reducible to it.
The partial ordering he uses has members which are  infinite unions  of  $n$-squares.
That example was enhanced in \cite{Dobrinen/Todorcevic10} to show that, assuming CH, there is a p-point with two Tukey-incomparable p-points Tukey reducible to it.
A closer look at the partial ordering of Blass  reveals inside  essentially a product of two copies of the topological Ramsey space $\mathcal{R}_1$ from \cite{Dobrinen/Todorcevic11}.
Our space $\mathcal{H}^2$ was constructed in order to construct or force a p-point which has initial Tukey structure  exactly the Boolean algebra $\mathcal{P}(2)$.
The spaces $\mathcal{H}^n$ were then the logical next step in constructing p-points with initial Tukey structure exactly $\mathcal{P}(n)$.

We point out that the space $\mathcal{H}^1$ is exactly the space $\mathcal{R}_1$ in \cite{Dobrinen/Todorcevic11}.

\begin{rem}
The space $\mathcal{H}^2$ was investigated  in \cite{TrujilloThesis}.
All the  results in this paper pertaining to the space $\mathcal{H}^2$ are due to Trujillo.
\end{rem}

\begin{example}[Hypercube Spaces $\mathcal{H}^n$, $1\le n<\om$]\label{ex.H}
Fix $1\le n<\om$, and let $J=n$.
For each $k<\om$ and $j\in n$, let $\bsA_{k,j}$ be any linearly ordered set of size $k+1$.
Letting $\bsA_k$ denote the sequence $(\bsA_{k,j})_{j\in n}$,
we see that $\lgl \bsA_k:k<\om\rgl$ is a generating sequence, where   each
 $\mathcal{K}_j$ is  the class of finite linearly ordered sets.
Let
$\mathcal{H}^n$ denote $\mathcal{R}(\lgl \bsA_k:k<\om\rgl)$.
It will follow from Theorem \ref{Tim-InitialStructures}
that the initial Tukey structure below $\mathcal{U}_{\mathcal{H}^n}$ is exactly that of the Boolean algebra $\mathcal{P}(n)$.
\end{example}

Many other examples of topological Ramsey spaces
are obtained in this manner, simply letting $\mathcal{K}_n$ be a \Fraisse\ class of finite ordered relational structures with the Ramsey property.

We now look at the most basic example of a topological Ramsey space generated by infinitely many \Fraisse\ classes.
When $J=\om$, $\mathcal{AR}_1$ no longer suffices as a base for an ultrafilter.
In fact, any filter which is Ramsey for this kind of space codes a Fubini product of the ultrafilters associated with $\mathcal{K}_j$ for each index $j\in \om$.
However,  the notion of a  filter Ramsey for such a space is still well-defined.

\begin{example}[The infinite Hypercube Space $\mathcal{H}^{\om}$]\label{ex.Homega}
Let $J=\om$.
For each $k<\om$ and $j\in k$, let $\bsA_{k,j}$ be any linearly ordered set of size $k+1$.
Letting $\bsA_k$ denote the sequence $(\bsA_{k,j})_{j\in k+1}$,
we see that $\lgl \bsA_k:k<\om\rgl$ is a generating sequence for the \Fraisse\ classes $\mathcal{K}_j$ being the class of finite linearly ordered sets.
Let
$\mathcal{H}^{\om}$ denote $\mathcal{R}(\lgl \bsA_k:k<\om\rgl)$.
It will be shown in Theorem \ref{Tim-InitialStructures} that
the  structure of the Tukey types of p-points Tukey reducible to  any  filter $\mathcal{C}_{\mathcal{H}^{\om}}$ which is Ramsey for $\mathcal{H}^{\om}$ is exactly $[\om]^{<\om}$.
The space $\mathcal{H}^{\om}$ is the first example of a topological Ramsey space which has associated filter $\mathcal{C}_{\mathcal{H}^{\om}}$
with 
 infinitely many Tukey-incomparable Ramsey ultrafilters Tukey reducible to it. 
\end{example}

We point out that, taking $J=\om$ and each $\mathcal{K}_j$, $j\in\om$, to be the \Fraisse\ class of finite ordered $(j+3)$-clique-free graphs,
the resulting topological Ramsey space codes the Fubini product seen in Theorem 3.12 in \cite{Baumgartner/Taylor78} of Baumgartner and Taylor  which produces an ultrafilter which is $n$-arrow for all $n$.

We conclude this section by showing how the partition properties of ultrafilters Ramsey for some space constructed from a generating sequence can be read off from  the \Fraisse\ classes.
Recall the following notation for partition relations.
For $k>l$, any $m\ge 2$, and an ultrafilter $\mathcal{U}$,
\begin{equation}
\mathcal{U}\ra(\mathcal{U})^m_{k,l}
\end{equation}
denotes that for any  $U\in\mathcal{U}$ and any partition of $[U]^m$ into $k$ pieces,
there is a subset $V\sse U$ in $\mathcal{U}$ such that $[V]^m$ is contained in at most $l$ pieces of the partition.
We shall say that the {\em Ramsey degree for $m$-tuples} for $\mathcal{U}$ is $l$, denoted $R(\mathcal{U},m)=l$,
if $\mathcal{U}\ra(\mathcal{U})^m_{k,l}$ for each $k\ge l$,
but $\mathcal{U}\not\ra(\mathcal{U})^m_{k,l-1}$.

It is straightforward to calculate the Ramsey degrees of ultrafilters Ramsey for topological Ramsey spaces constructed from a generating sequence, given knowledge of the \Fraisse\ classes used in the construction.
For a given \Fraisse\ class $\mathcal{K}$, for each $s\ge 1$,
let $\Iso(\mathcal{K},s)$ denote the number of isomorphism classes in $\mathcal{K}$ of structures with universe of size $s$.
Let $S(m)$ denote the collection of all finite sequences $\vec{s}=\lgl s_0,\dots,s_{l-1}\rgl\in (\om\setminus\{0\})^{<\om}$ such that $s_0+\dots +s_{l-1}=m$.

\begin{fact}\label{fact.partitionrels}
Let $J=1$, $\mathcal{K}$ be a \Fraisse\ class of finite ordered relational structures with the Ramsey property,
and $\mathcal{U}_{\mathcal{K}}$ be an ultrafilter Ramsey for $\mathcal{R}(\lgl \bsA_k:k<\om\rgl)$ for some generating sequence for $\mathcal{K}$.
Then for each $m\ge 2$,
\begin{equation}
R(\mathcal{U}_{\mathcal{K}},m)=
 \Sigma_{s\in S(m)}\Pi_{i<|s|}\Iso(\mathcal{K},s_i).
\end{equation}
\end{fact}

\begin{examples}\label{ex.J=1Ramseydegrees}
For an ultrafilter $\mathcal{U}_{\mathcal{H}^1}$ Ramsey for the space $\mathcal{H}^1$, 
we have $R(\mathcal{U}_{\mathcal{H}^1},2)=2$,
$R(\mathcal{U}_{\mathcal{H}^1},3)=4$,
$R(\mathcal{U}_{\mathcal{H}^1},4)=8$, and in general, 
$R(\mathcal{U}_{\mathcal{H}^1},m)=2^{m-1}$.

For an ultrafilter $\mathcal{U}_{\mathcal{A}_2}$ Ramsey for the space $\mathcal{A}_2$, we have 
$R(\mathcal{U}_{\mathcal{A}_2},2)=3$,
$R(\mathcal{U}_{\mathcal{A}_2},3)=12$, and
$R(\mathcal{U}_{\mathcal{A}_2},4)=35$.
In fact, for each $n\ge 3$, 
$R(\mathcal{U}_{\mathcal{A}_n},2)=3$, since the only relation is the edge relation.
The numbers $R(\mathcal{U}_{\mathcal{A}_n},m)$  can be calculated from  the recursive formula in Fact \ref{fact.partitionrels},
but as they grow quickly, we leave this to the interested reader.
\end{examples}

When $J=2$,
the Ramsey degrees are again calculated from knowledge of the \Fraisse\ classes $\mathcal{K}_0$ and $\mathcal{K}_1$.

\begin{fact}\label{fact.RamseydegJ=2}
For $\mathcal{R}$ a topological Ramsey space constructed from a generating sequence for \Fraisse\ classes $\mathcal{K}_j$, $j\in 2$,
letting $\mathcal{U}_{\mathcal{R}}$ be an ultrafilter Ramsey for $\mathcal{R}$, we have
\begin{equation}
R(\mathcal{U}_{\mathcal{R}},2)=1+
\Iso(\mathcal{K}_0,2)+\Iso(\mathcal{K}_1,2) + 2 \Iso(\mathcal{K}_0,2)\cdot \Iso(\mathcal{K}_1,2).
\end{equation}
\end{fact}

The $1$ comes from the fact that a pair can come from different blocks; for a pair coming from the same block,
 $\Iso(\mathcal{K}_0,2)$  takes care of the case when the pair has the same second dimensional coordinate,  $\Iso(\mathcal{K}_1,2)$ takes care of the case  when the pair has the same first dimensional coordinate, and  $ 2 \Iso(\mathcal{K}_0,2)\cdot \Iso(\mathcal{K}_1,2) $ is the number of possible different colors for pairs which are diagonal to each other.

For larger $J$ and $m$, the Ramsey degrees can be obtained in a similar manner as above.
For example,
$R(\mathcal{U}_{\mathcal{H}^2},3)=24$.
We leave the reader with the following:
$R(\mathcal{U}_{\mathcal{H}^2},2)=5$,
$R(\mathcal{U}_{\mathcal{H}^3},2)=14$,
and we conjecture that in general,
$R(\mathcal{U}_{\mathcal{H}^n},2)=\frac{3^n-1}{2}+1$.


\section{Canonical equivalence relations for products of  structures \\ from \Fraisse\ classes
of finite ordered relational structures}\label{sec.ProductER}

In the main theorem of this section,  Theorem \ref{FiniteERProducts}, we extend  the finite  \Erdos-Rado Theorem \ref{thm.finiteER} to finite products of sets as well as finite products of 
 members of  \Fraisse\ classes of finite ordered relational structures with the Ramsey property and an additional property which we shall call the {\em Order-prescribed Free Amalgamation Property}, defined below.
In particular, this extends the 
Product Ramsey Theorem \ref{thm.Sokic} from finite colorings to equivalence relations for  \Fraisse\ classes with the aforementioned properties.
Theorem \ref{FiniteERProducts} will follow from Theorem \ref{thm.ERProducts}, 
which gives canonical equivalence relations for blocks from topological Ramsey spaces constructed from generating sequences for these special types of \Fraisse\ classes.
We proceed in this manner for two reasons.
First, the 
strength of topological Ramsey space theory, and in particular the
availability  of the Abstract Nash-Williams Theorem, greatly streamlines the proof.
Second, our desired application of 
Theorem \ref{FiniteERProducts} is in the proof of Theorem \ref{canonical R} in Section \ref{sec.canon}
to  find the canonical equivalence relations on fronts for
 topological Ramsey spaces constructed from a generating sequence.

Recall that  $|\bsB|$ denotes the universe of the structure $\bsB$, and $\|\bsB\|$ denotes the cardinality of the universe of $\bsB$.
For a structure $\bsX_{j}\in\mathcal{K}_j$, we shall let $\{x_{j}^p:p<\|\bsX_{j}\|\}$ denote the members of the universe $|\bsX_{j}|$ of $\bsX_{j}$, enumerated in $<$-increasing order.

\begin{defn}[Order-Prescribed Free Amalgamation Property  (OPFAP)]\label{def.opfap}
An ordered relational \Fraisse\ class $\mathcal{K}$ has the 
{\em Order-Prescribed   Free Amalgamation Property} if 
the following holds.
Suppose $\bsX,\bsY,\bsZ$ are structures in $\mathcal{K}$ 
with embeddings 
 $e:\bsZ\ra\bsX$ and $f:\bsZ\ra\bsY$.
Let $K=\|\bsX\|$, $L=\|\bsY\|$, and $M=\|\bsZ\|$.
Let $\{x^k:k\in K\}$ denote $|\bsX|$  and $\{y^l:l\in L\}$ denote  $|\bsY|$, the universes of $\bsX$ and $\bsY$, respectively.
Let $K'=\{k'_m:m\in M\}\sse K$ and $L'=\{l'_m:m\in M\}\sse L$ be the subsets such that 
$\bsX\re K'=e(\bsZ)$ and 
$\bsY\re L'= f(\bsZ)$.

Let $\vec{\rho}:K\times L\ra\{<,=,>\}$ be any function such that 
\begin{enumerate}
\item[(a)]
For each $m\in M$,
$\vec{\rho}(k'_m,l'_m)=\, =$;
\item[(b)]
For each $m\in M$, $k<k'_m$ and $l>l'_m$ implies $\vec{\rho}(k,l)=\, <$;
and $k>k'_m$ and $l<l'_m$ implies $\vec{\rho}(k,l)=\, >$;
\item[(c)]
For all $(k,l)\in (K\setminus K')\times(L\setminus L')$,
$\vec{\rho}(k,l)\ne\,  =$;
\item[(d)]
 $\vec{\rho}(k,l)=\, <$ implies  for all $l'>l$ and $k'<k$, 
$\vec{\rho}(k',l')=\, <$;\\
and  $\vec{\rho}(k,l)=\, >$ 
implies  for all $l'<l$ and $k'>k$, 
$\vec{\rho}(k',l')=\, >$.
\end{enumerate}

Then there is a  free amalgamation  $(g,h,\bsW)$
of $(\bsZ,e,\bsX,f,\bsY)$ 
and there is a function $\sigma: K+L\ra \|\bsW\|$ such that the following hold:
\begin{enumerate}
\item
$\sigma\re K$ and $\sigma\re[K,K+L)$ are strictly increasing;
\item
$\bsW\re \sigma'' K =g(\bsX)$ and
$\bsW\re \sigma'' [K,K+L)=h(\bsY)$;
\item
For all $m\in M$, $\sigma(k'_m)=\sigma(K+l'_m)$,
and 
$\bsW\re \sigma''\{k'_m :m\in M\}=
\bsW\re \sigma''\{K+l'_m:m\in M\}=g\circ e(\bsZ)=h\circ f(\bsZ)\cong\bsZ$;
\item
For all $(k,l)\in K\times L$,
$w^{\sigma(k)}\,\vec{\rho}\, w^{\sigma(K+l)}$.
\end{enumerate}
Hence, $\bsW$ contains copies of $\bsX$ and $\bsY$ 
which appear as substructures of $\bsW$ in the order prescribed by $\vec{\rho}$.
\end{defn}

In words, the Order-Prescribed  Free Amalgamation Property says that given any structure $\bsZ$ appearing as a substructure of both $\bsX$ and $\bsY$,
one can find a strong amalgamation $\bsW$ of $\bsX$ and $\bsY$  so that the members of the universes of the copies of $\bsX$ and $\bsY$ in $\bsW$ lying between the members of the universe  of the copy of $\bsZ$  can lie in any  order which we prescribed ahead of time, 
and  the only relations between members of the copies of $\bsX$ and $\bsY$  in $\bsW$ are those in $\bsZ$, that is, the amalgamation is free.

\begin{rem}
Note that (3) in Definition \ref{def.opfap} implies that there are no transitive relations on $\mathcal{K}$.
Thus, any \Fraisse\ class which has a transitive relation does not satisfy the OPFAP.
We point out that the classes of ordered finite graphs, ordered finite $K_n$-free graphs, 
and  more generally, the classes of ordered set-systems omitting some collection of  irreducible structures (see \cite{Nesetril/Rodl83})
all satisfy the OPFAP.
\end{rem}

\begin{defn}\label{defn.finite.prod.canonical}
Let  $\mathcal{K}_j$, $j\in J<\om$ be
 \Fraisse\ classes of finite ordered relational structures with the Ramsey property and the OPFAP.
For each $j\in J$, let $\bsA_j,\bsB_j\in\mathcal{K}_j$ such that $\bsA_j\le\bsB_j$.
Given a subset $I_j\sse\|\bsA_j\|$ and $\bsX_j,\bsY_j\in{\bsB_j \choose\bsA_j}$,
we write $|\bsX_j|\, E_{I_j}\, |\bsY_j|$ if and only if for all $i\in I_j$, $x^i_j=y^i_j$.

An equivalence relation $E$ on ${(\bsB_j)_{j\in J}\choose (\bsA_j)_{j\in J}}$ is {\em canonical} if and only if
for each $j\in J$, there is a set $I_j\sse \|\bsA_j\|$ such that
for all $(\bsX_j)_{j\in J},(\bsY_j)_{j\in J}\in {(\bsB_j)_{j\in J}\choose (\bsA_j)_{j\in J}}$,
\begin{equation}
(\bsX_j)_{j\in J}\, E\, (\bsY_j)_{j\in J}\longleftrightarrow
\forall j\in J,\ |\bsX_j|\, E_{I_j}\, |\bsY_j|.
\end{equation}
When $E$ is canonical, given by $E_{I_j}$, $j\in J$, then we shall write $E=E_{(I_j)_{j\in J}}$.
\end{defn}

\begin{thm}\label{FiniteERProducts}
Let  $\mathcal{K}_j$, $j\in J<\om$, be
 \Fraisse\ classes of ordered relational structures with the Ramsey property  and the OPFAP.
For each $j\in J$, let $\bsA_j,\bsB_j\in\mathcal{K}_j$ be such that $\bsA_j\le\bsB_j$.
Then for each $j\in J$, there is a $\bsC_j\in\mathcal{K}_j$ such that for each equivalence relation $E$ on ${(\bsC_j)_{j\in J}\choose(\bsA_j)_{j\in J}}$,
there is a sequence $(\bsB'_j)_{j\in J}\in{(\bsC_j)_{j\in J}\choose(\bsB_j)_{j\in J}}$ such that 
$E$ restricted to ${(\bsB'_j)_{j\in J}\choose(\bsA_j)_{j\in J}}$ is canonical. 
\end{thm}

Theorem \ref{FiniteERProducts} will follow immediately from the next theorem.

\begin{thm}\label{thm.ERProducts}
Let
$\lgl \bsA_k:k<\om\rgl$ be a generating sequence associated to some \Fraisse\ classes of finite ordered relational structures $\mathcal{K}_j$, $j\in J$, each satisfying the Ramsey property  and the OPFAP.
Let $n<\om$ and $L\sse J_n$ be given,
 and let $E$ be 
 an equivalence relation on  
$\bigcup_{k\ge n}{ (\bsA_{k,j})_{j\in L}\choose (\bsA_{n,j})_{j\in L}}$
such that
$E\, \sse \bigcup_{k\ge n}{ (\bsA_{k,j})_{j\in L}\choose (\bsA_{n,j})_{j\in L}}\times { (\bsA_{k,j})_{j\in L}\choose (\bsA_{n,j})_{j\in L}}$.
Then
there is a  $C\in \mathcal{R}(\lgl \bsA_k:k<\om\rgl)$ and there are index sets
$I_{j}\sse \|\bsA_{n,j}\|$ such that 
for all  
$k\ge n$,
$E=E_{(I_j)_{j\in L}}$ when restricted to 
${(\bsC_{k,j})_{j\in L}\choose (\bsA_{n,j})_{j\in L}}$.
That is, 
for each $k\ge n$,
 and each pair $(\bsX_{n,j})_{j\in L},(\bsY_{n,j})_{j\in L}\in {(\bsC_{k,j})_{j\in L}\choose (\bsA_{n,j})_{j\in L}}$,
$$
\left((\bsX_{n,j})_{j\in L}\, E\, (\bsY_{n,j})_{j\in L} \longleftrightarrow
\forall j\in L,\ 
|\bsX_{n,j}|\, E_{I_{j}}\, |\bsY_{n,j}|\right).
$$
\end{thm}

\begin{proof}
Let $J\le\om$ and $\mathcal{K}_j$, $j\in J$, be a collection of \Fraisse\ classes of finite ordered relational structures with the Ramsey property
 and the Order-Prescribed  Free Amalgamation Property.
Let 
$\lgl \bsA_k:k<\om\rgl$ be a generating sequence associated with the $\mathcal{K}_j$, $j\in J$,
and let $\mathcal{R}$ denote the topological Ramsey space $\mathcal{R}(\lgl \bsA_k:k<\om\rgl)$.
Recall that $J_n=J$ if $J<\om$, and $J_n=n$ if $J=\om$.

Before beginning the inductive proof, we establish some terminology and  notation, and  Lemma \ref{lem.1} below.
Given  $n<\om$, for each  $j\in J_n$  
let $K_{j}$ denote $\|\bsA_{n,j}\|$, the cardinality of the universe  of the structure $\bsA_{n,j}$.
For a given structure $\bsX\in\mathcal{K}_j$,
  let $\{x^0,\dots,x^{\|\bsX\|-1}\}$ denote  $|\bsX|$, the universe of $\bsX$, enumerated in increasing order.
For $M\sse\|X\|$, let $\bsX\re M$ denote the substructure of $\bsX$ on universe $\{x^k:k\in M\}$.
For each $j\in J_n$, 
let Amalg$(n,j)$ denote the collection of all $\bsX\in\mathcal{K}_j$ such that $\bsX$ is an amalgamation of two copies of $\bsA_{n,j}$.
By this we mean precisely that there are set of indices $M_0,M_1\sse\|\bsX\|$ such that 
$M_0$ and $M_1$ each have cardinality $K_{j}$,
$M_0\cup M_1=\|\bsX\|$,
and 
$\bsX\re M_0\cong\bsX\re M_1\cong\bsA_{n,j}$.

By the definition of a generating sequence, given $n<\om$   there is an $m>n$  such that for each $j\in J_n$, every structure $\bsX\in$ Amalg$(n,j)$ 
embeds into $\bsA_{m,j}$.
Define $\mathcal{I}_{j}$ to be
the collection of functions $\iota_j:2K_{j}\ra \|\bsA_{m,j}\|$ 
such that
$\iota_j\re K_{j}$ and $\iota_j\re[K_{j},2K_{j})$ are strictly increasing, 
 the substructure $\bsA_{m,j}\re \iota'' 2 K_{j}$ is in Amalg$(n,j)$, and moreover,
$\bsA_{m,j}\re\iota'' K_{j}\cong \bsA_{m,j}\re \iota''[K_{j},2K_{j})\cong\bsA_{n,j}$.
For each $\iota_j\in \mathcal{I}_{j}$ and $\bsX\cong\bsA_{m,j}$,
 fix the notation
$$
\iota_j(\bsX):=
(\bsX\re\iota'' K_{j},\bsX \re \iota''[K_{j},2K_{j}))=(\{x^{\iota_j(0)},\dots,x^{\iota_j(K_{j}-1)}\},
\{x^{\iota_j(K_{j})},\dots,x^{\iota_j(2K_{j}-1)}\}),
$$
the pair of structures in ${\bsX \choose \bsA_{n,j}}$ determined by $\iota_j$.

Throughout the poof of this theorem, given any structure $\bsD$ which embeds $\bsA_{n,j}$, for any $\bsX,\bsY\in {\bsD \choose\bsA_{n,j}}$,
the pair $(\bsX,\bsY)$ is considered both as an ordered pair of structures isomorphic to $\bsA_{n,j}$ as well  as the substructure of $\bsD\re(|\bsX|\cup|\bsY|)$ with all inherited relations.

\begin{clm}\label{claim.F'}
Let $j\in J_n$.
There is a structure $\bsB\in\mathcal{K}_j$ with a substructure $\bsC\in{\bsB \choose\bsA_{m,j}}$ such that
for each $\iota\in\mathcal{I}_j$,
for each $\tau\in\mathcal{I}_j$ such that $\tau(\bsA_{m,j})\cong\iota(\bsC)$
there is a $\bsV\in{\bsB\choose\bsA_{m,j}}$ such that 
$\tau(\bsV)=\iota(\bsC)$.
\end{clm}

\begin{proof}
Let $p=|\mathcal{I}_j|-1$ and enumerate $\mathcal{I}_j$ as $\lgl \iota^i:i\le p\rgl$.
The proof proceeds by amalgamation in $p$ stages, each stage $i\le p$ proceeding inductively  by amalgamating   to obtaining a $\bsB^i$ which satisfies the claim for the structure $\iota_i(\bsA_{m,j})$.

Let  $\bsW^0$  denote the substructure $\iota^0(\bsA_{m,j})$.
Let $\mathcal{I}^0$ denote the set of $\tau\in\mathcal{I}_j$ such that $\tau(\bsA_{m,j})\cong\bsW^0$, and enumerate $\mathcal{I}^0$ as $\lgl \tau^{0,k}:k\le q^0\rgl$.
Let $e^{0,0}:\bsW^0\ra\bsA_{m,j}$ be the
identity injection on $\bsW^0$,
so that
 $e^{0,0}(\bsW^0)=\iota^0(\bsA_{m,j})$.
Let $f^{0,0}:\bsW^0\ra\bsV^{0,0}$ be an embedding of $\bsW^0$ into a copy  $\bsV^{0,0}$ of $\bsA_{m,j}$ such that $f^{0,0}(\bsW^0)=\tau^{0,0}(\bsV^{0,0})$.
Let $(g^{0,0},h^{0,0},\bsB^{0,0})$ be a free amalgamation of $(\bsW^0, e^{0,0},\bsA_{m,j},f^{0,0},\bsV^{0,0})$, and
let $e^{0,1}=g^{0,0}\circ e^{0,0}$.
Thus, $e^{0,1}:\bsW^0\ra\bsB^{0,0}$.
Let $f^{0,1}:\bsW^0\ra\bsV^{0,1}$ be an embedding of $\bsW^0$ into a copy  $\bsV^{0,1}$ of $\bsA_{m,j}$ such that $f^{0,1}(\bsW^0)=\tau^{0,1}(\bsV^{0,1})$.
Let $(g^{0,1},h^{0,1},\bsB^{0,1})$ be a free amalgamation of $(\bsW^0, e^{0,1},\bsB^{0,0},f^{0,1},\bsV^{0,1})$, and
let $e^{0,2}=g^{0,1}\circ e^{0,1}$,
so that $e^{0,2}:\bsW^0\ra\bsB^{0,1}$.

Given 
$e^{0,k+1}:\bsW^0\ra\bsB^{0,k}$,
let $f^{0,k+1}:\bsW^0\ra\bsV^{0,k+1}$ be an embedding of 
$\bsW^0$ into a copy  $\bsV^{0,k+1}$ of $\bsA_{m,j}$ such that $f^{0,k+1}(\bsW^0)=\tau^{0,k+1}(\bsV^{0,k+1})$.
Let $(g^{0,k+1},h^{0,k+1},\bsB^{0,k+1})$ be a free amalgamation of $(\bsW^0, e^{0,k+1},\bsB^{0,k},f^{0,k_1},\bsV^{0,k+1})$, and
let $e^{0,k+2}=g^{0,k+1}\circ e^{0,k+1}$,
so that $e^{0,k+2}:\bsW^0\ra\bsB^{0,k+1}$.
At the end of the $q^0$ many stages of the construction, let 
$h^{0}= h^{0,q^0}\circ\dots\circ h^{0,0}$
and let $\bsB^0=\bsB^{0,q^0}$,
so that $h^{0}$ embeds  the original copy of $\bsA_{m,j}$ into $\bsB^{0}$.
This concludes the $0$-th stage of constructing $\bsB$.

For the $i$-th stage,
suppose that $0<i\le p$ and $h^{i-1}:\bsA_{m,j}\ra\bsB^{i-1}$ are given.
Let $\mathcal{I}^{i}$ denote the set of $\tau\in\mathcal{I}_j$ such that $\tau(\bsA_{m,j})\cong\bsW^{i}:=\iota^{i}(\bsA_{m,j})$, and enumerate these as $\lgl \tau^{i,k}:k\le q^{i}\rgl$.
Let $e^{i,0}:\bsW^{i}\ra\bsB^{i-1}$ be the embedding such that $e^{i,0}(\bsW^{i})=\iota^{i}(h^{i-1}(\bsA_{m,j}))$.
Let $f^{i,0}:
\bsW^{i}\ra\bsV^{i,0}$ be an embedding of $\bsW^{i}$ into a copy  $\bsV^{i,0}$ of $\bsA_{m,j}$ such that $f^{i,0}(\bsW^i)=\tau^{i,0}(\bsV^{i,0})$.
Let $(g^{i,0},h^{i,0},\bsB^{i,0})$ be a free amalgamation of $(\bsW^{i}, e^{i,0},\bsB_{i-1},f^{i,0},\bsV^{i,0})$, and
let $e^{i,1}=g^{i,0}\circ e^{i,0}$.

For $0<k \le q^i$ suppose 
$e^{i,k}:\bsW^{i}\ra\bsB^{i,k-1}$ is given.
Let $f^{i,k}:\bsW^{i}\ra\bsV^{i,k}$ be an embedding of 
$\bsW^{i}$ into a copy  $\bsV^{i,k}$ of $\bsA_{m,j}$ such that $f^{i,k}(\bsW^{i})=\tau^{i,k}(\bsV^{i,k})$.
Let $(g^{i,k},h^{i,k},\bsB^{i,k})$ be a free amalgamation of $(\bsW^{i}, e^{i,k},\bsB^{i,k-1},f^{i,k},\bsV^{i,k})$, and
let $e^{i,k+1}=g^{i,k}\circ e^{i,k}$.
At the end of the $q^{i}$ many stages of the construction, let 
$\bsB^i=\bsB^{i,q^i}$ and
$h^{i}= h^{i,q^{i}}\circ\dots\circ h^{i,0}$.
This concludes the $i$-th stage.

Repeating the procedure 
for all $i\le p$, we
 obtain a structure $\bsB:=\bsB^{p}$ satisfying the claim.
\end{proof}

\begin{clm}\label{claim.F}
Let  $B\in\mathcal{R}$, $n<\om$ and $L\sse J_n$ be given.
There is a  $C\le B$ such that for all $k\ge m$,
for
each
 pair $(\bsX_j)_{j\in L},(\bsY_j)_{j\in L}\in {(\bsC_{k,j})_{j\in L}\choose(\bsA_{n,j})_{j\in L}}$ and for  each $(\iota_j:j\in L)\in \prod_{j\in L}\mathcal{I}_j$ such that for each $j\in L$, $\iota_j(\bsA_{m,j})\cong (\bsX_j,\bsY_j)$,
there is a $(\bsZ_j)_{j\in L}\in{(\bsB_{k',j})_{j\in L}\choose (\bsA_{m,j})_{j\in L}}$ such that each  $\iota_j(\bsZ_j)=(\bsX_j,\bsY_j)$,
where $k'$ is such that $\bsC_{k,j}$ is a substructure of $\bsB_{k',j}$.
\end{clm}

\begin{proof}
By  Claim \ref{claim.F'},
for each $p\ge m$ and each $j\in L$,
there is a structure $\bsB_{p,j}^*\in\mathcal{K}_j$ containing  a substructure $\bsC_{p,j}^*\in{\bsB_{p,j}^*\choose\bsA_{p,j}}$  with the following property:
Given $\bsU_j\in{\bsC_{p,j}^*\choose\bsA_{m,j}}$ and
 $\iota_j\in\mathcal{I}_j$,
for each $\tau_j\in\mathcal{I}_j$ such that $\tau_j(\bsA_{m,j})\cong(\bsX_j,\bsY_j)$,
there is a $\bsV_j\in{\bsB^*_{p,j}\choose\bsA_{m,j}}$ such that 
$\tau_j(\bsV_j)= \iota_j(\bsU_j)$.
Since for a generating sequence, each structure in $\mathcal{K}_j$ embeds into all but finitely many $\bsA_{k,j}$,
there is a subsequence $(k_p)_{p\ge m}$ such that  each $\bsB_{p,j}^*$ embeds as a substructure of $\bsA_{k_p,j}$.

Thinning through this subsequence, for each $p\ge m$, take $\bsC_{p,j}$ to be a substructure of $\bsB_{k_p,j}$ isomorphic to $\bsA_{p,j}$ satisfying Claim \ref{claim.F'},
and let $\bsC_{p,j}=\bsB_{p,j}$ for each $p<m$.
Then we obtain a $C\le B$ which satisfies the claim.
\end{proof}

For fixed $n$ and $L\sse J_n$, we shall let
 $\mathcal{I}$ denote the set of all 
sequences $(\iota_j)_{j\in L}\in\prod_{j\in L}\mathcal{I}_{j}$.
Given  a sequence $\iota=(\iota_j:j\in L)\in\mathcal{I}$
 and $X(m)\in\mathcal{R}(m)$,
 fix the notation
$$
\iota(X(m)):=((\bsX_{m,j}\re\iota_j'' K_{n,j})_{j\in L},
(\bsX_{m,j}\re\iota_j'' [K_{n,j},2K_{n,j}))_{j\in L}).
$$
Thus, $\iota(X(m))$ is a pair of sequences of structures, each  sequence of which is isomorphic to $(\bsA_{n,j})_{j\in L}$.
Moreover, for each $j\in L$,
the pair 
$(\bsX_{m,j}\re\iota_j'' K_{n,j},\bsX_{m,j}\re\iota_j'' [K_{n,j},2K_{n,j}))$
also determines  a substructure of $\bsX_{m,j}$ which is an amalgamation of two copies of $\bsA_{n,j}$.

\begin{lem}\label{lem.1}
Let 
 $n<\om$ and $L\sse J_n$  be given. 
Let $E$ be an equivalence relation on  $\bigcup_{k\ge n} {(\bsA_{k,j})_{j\in L} \choose (\bsA_{n,j})_{j\in L}}\times{(\bsA_{k,j})_{j\in L} \choose (\bsA_{n,j})_{j\in L}}$.
Let  $m$  be large enough that for each $j\in L$,  all members of {\rm Amalg}$(n,j)$ embed into $\bsA_{m,j}$.
Then there are $C\le B \in\mathcal{R}$ and a subset $\mathcal{I}'\sse\mathcal{I}$ such that
$C\le B$ satisfy Claim \ref{claim.F} and
for all  $k\ge n$ and  all $X(n),Y(n)\in\mathcal{R}(n)|C(k)$,
\begin{equation}
(\bsX_{n,j})_{j\in L}\, E\, (\bsY_{n,j})_{j\in L}
\longleftrightarrow 
\exists U(m)\in\mathcal{R}(m)|B\ \
\exists \iota\in \mathcal{I}'\ \ 
\iota(U(m))=((\bsX_{n,j})_{j\in L},(\bsY_{n,j})_{j\in L})).
\end{equation}
\end{lem}

\begin{proof}
For each $\iota\in\mathcal{I}$,
define
$$
\mathcal{H}_{\iota}=\{r_{m+1}(X): X\in\mathcal{R}\mathrm{\ and\ } E(\iota(X(m)))\}.
$$
Each $\mathcal{H}_{\iota}$ is 
a subset of the Nash-Williams family $\mathcal{AR}_{m+1}$.
Hence, by the Abstract Nash-Williams Theorem, 
there is a $B\in\mathcal{R}$ which is homogeneous for  $\mathcal{H}_{\iota}$, for all $\iota\in\mathcal{I}$.
That is, for each $\iota\in\mathcal{I}$,
either $\mathcal{AR}_{m+1}|B\sse\mathcal{H}_{\iota}$ or else $\mathcal{AR}_{m+1}|B\cap\mathcal{H}_{\iota}=\emptyset$.
Let $\mathcal{I}'=\{\iota\in\mathcal{I}: E(\iota(B(m)))\}$.
Finally, take  $C\le B$ satisfying the conclusion of Claim \ref{claim.F} for each $j\in L$.
\end{proof}

We will prove the following statement by induction on  $M\ge 1$:
Given any   $n$ such that $J_n\ge M$,   $L\in [J_n]^M$,
 and  an equivalence relation $E$ on
$\bigcup_{k\ge n}{(\bsA_{k,j})_{j\in L} \choose (\bsA_{n,j})_{j\in L}}\times{(\bsA_{k,j})_{j\in L} \choose (\bsA_{n,j})_{j\in L}}$,
 there is a $C\in\mathcal{R}$ 
and there are index sets $I_{j}\sse K_{j}$
such that 
for all $k\ge n$ and all $(\bsX_{n,j})_{j\in L},(\bsY_n)_{j\in L}\in {(\bsC_{k,j})_{j\in L} \choose (\bsA_{n,j})_{j\in L}}$,
$$
(\bsX_{n,j})_{j\in L}\, E\, (\bsY_{n,j})_{j\in L}\mathrm{\ if\ and\ only\ if\ }  \forall j\in L,\ |\bsX_{n,j}|\, E_{I_{j}}\, |\bsY_{n,j}|.
$$
\vskip.1in

\underline{Base Case}.
$M=1$.
Let $n<\om$, $j\in J_n$, and  $L=\{j\}$. 
Let $E$ be an equivalence relation such that  $E\sse \bigcup_{k\ge n} {(\bsA_{k,j}) \choose (\bsA_{n,j})}\times{(\bsA_{k,j}) \choose (\bsA_{n,j})}$.
Let $C\le B\in\mathcal{R}$ and  $\mathcal{I}'_{j}$ satisfy Lemma  \ref{lem.1}.
Define 
$$
I_{j}=\{i\in K_{j}:\forall \iota\in \mathcal{I}',\ 
\iota_j(i)=\iota_j(K_{j}+i)\}.
$$
Since each $\iota\in\mathcal{I}'_{j}$ is a sequence consisting of only a single entry, $(\iota_j)$, we shall  abuse notation for the base case and use $\iota$ in place of $\iota_j$.
We make the convention that for each $k<\om$,  $k'$ denotes the number such that $\bsC_{k,j}$ is a substructure of $\bsB_{k',j}$.

\begin{clm}\label{claim.E}
If $\iota\in\mathcal{I}'_{j}$,
 $\tau\in\mathcal{I}_{j}$,
and $\tau(\bsA_{m,j})\cong\iota(\bsA_{m,j})$,
then $\tau\in\mathcal{I}'_{j}$.
\end{clm}

\begin{proof}
Let $\iota$ and $\tau$ be as in the hypothesis.
Let $(\bsX,\bsY)=\iota(\bsC_{m,j})$.
By Claim \ref{claim.F},
there is an $m'\ge m$ and a $\bsV\in{\bsB_{m',j}\choose \bsA_{m,j}}$ such that $\tau(\bsV)=(\bsX,\bsY)$.
Since $\iota\in\mathcal{I}'_{j}$, Lemma \ref{lem.1} implies that
$\bsX\, E\, \bsY$.
Therefore, by Lemma \ref{lem.1}, $\tau$ is also in $\mathcal{I}'_{j}$.
\end{proof}

\begin{clm}\label{claim.A}
Let  $i\in K_{j}\setminus I_{j}$
For each $l\ge m$ there are   $\bsX,\bsY\in {\bsC_{l,j}\choose \bsA_{n,j}}$ such that for each $k\in K_{j}\setminus \{i\}$,
$x^k=y^k$,
$x^i\ne y^i$,
and $\bsC_{l,j}\re(|\bsX|\cup|\bsY|)$ is a free amalgamation of $\bsX$ and $\bsY$.
Let $\iota$ be any map in $\mathcal{I}_{j}$ such that $\iota(\bsA_{m,j})\cong(\bsX,\bsY)$.
Then $\iota\in\mathcal{I}'_{j}$.
\end{clm}

\begin{proof}
First we prove a  general fact.
Let  $\sigma\in\mathcal{I}_{j}$ and $i\in K_{j}$ be such that $\sigma(i)<\sigma(K_{j}+i)$.
Let $\bsV=(\bsX,\bsY)=\sigma(\bsA_{m,j})$.
Let $k\in \|\bsV\|$ be such that $v^k=x^i$.
Take another copy $\bsW=(\bsZ,\bsY')\cong \sigma(\bsA_{m,j})$.
Let $e:\bsY\ra\bsV$ be the identity embedding
and $f:\bsY\ra\bsW$ be such that $f(\bsY)$ equals $\bsY'$.
By the OPFAP, we may freely amalgamate
$(\bsY,e,\bsV,f,\bsW)$ to some $(g,h,\bsU)$
 so that the following hold:
$\bsV\cong\bsU\re(\|\bsU\|\setminus\{k\})$ and
$\bsW\cong\bsU\re(\|\bsU\|\setminus\{k+1\})$.
In words, $\bsU$ consists exactly of copies of the substructures $(\bsX,\bsY)$ and $(\bsZ,\bsY')$ where 
the copies of $\bsY$ and $\bsY'$ coincide, and 
the copies of  $\bsX$ and $\bsZ$ in $\bsW$ differ only on their $i$-th coordinates.
Thus, by an argument similar to Claim \ref{claim.F},
possibly thinning $C$ again, 
we may assume that for each $\sigma\in\mathcal{I}_{j}$ and $i\in K_{j}$ such that $\sigma(i)<\sigma(K_{j}+i)$
and $(\bsX,\bsY)=\sigma(\bsC_{m,j})$,
there  are substructures $\bsZ$ and $\bsW$ as above coming from $\bsB_{m',j}$.
That is,  there is a $z^k$ in $\bsB_{m',j}$ so that the substructure of $\bsB_{m',j}$ restricted to universe of $\sigma(\bsC_{m,j})\cup\{z^k\}$ is isomorphic to $\bsW$.

To prove the claim,
 first note that
since   $i$ is in $K_{j}\setminus I_{j}$,
 there is a $\sigma\in\mathcal{I}'_{j}$ such that $\sigma(i)\ne \sigma(K_{j}+i)$.
Let $(\bsX,\bsY)=\sigma(\bsC_{m,j})$.
Since $\sigma$ is in $\mathcal{I}'_{j}$, it follows that $\bsX\, E\, \bsY$, by Lemma \ref{lem.1}.
Without loss of generality, assume that $x^i<y^i$.
By the previous paragraph, there are structures  $\bsZ\in{\bsB_{m',j}\choose \bsA_{n,j}}$ and 
$\bsW\in{\bsB_{m',j}\choose \bsA_{m,j}}$   such that
\begin{enumerate}
\item
for each $k\in K_{j}\setminus\{i\}$,
$z^k=x^k$,
\item
$z^i<x^i$,
\item
$\bsB_{m',j}\re(|\bsX|\cup|\bsZ|)$ is the free amalgamation of $\bsX$ and $\bsZ$, and
\item
$\sigma(\bsW)=(\bsZ,\bsY)$.
\end{enumerate}
Since $\sigma$ is in $\mathcal{I}'_{j}$, it follows that $\bsZ\, E\, \bsY$.
Hence, $\bsX\, E\, \bsZ$.

Now let $\iota\in\mathcal{I}_{j}$ be any map such that $\iota(\bsA_{m,j})\cong(\bsX,\bsZ)$; that is, the free amalgamation of two copies of $\bsA_{n,j}$ where only their $i$-th coordinates differ.
Then $\iota$ must be in $\mathcal{I}'_{j}$ by Lemma \ref{lem.1}, since $\bsX\, E\, \bsZ$ and there is a $\bsD\in{\bsB_{m',j}\choose\bsA_{m,j}}$ such that $\iota(\bsD)=(\bsX,\bsZ)$.
\end{proof}

\begin{clm}\label{claim.B}
Let  $\iota,\tau\in\mathcal{I}'_{j}$ and  $k\ge m$ be given.
Suppose there are
$\bsV,\bsW\in {\bsB_{k',j}\choose\bsA_{m,j}}$ such that $\iota(\bsV)=(\bsX,\bsY)$ and $\tau(\bsW)=(\bsX,\bsZ)$,
where 
$\bsX,\bsY,\bsZ\in {\bsC_{k,j}\choose\bsA_{n,j}}$.
Then for each $\sigma\in\mathcal{I}_{j}$ such that there is a $\bsU\in {\bsB_{k',j}\choose\bsA_{m,j}}$ such that 
$\sigma(\bsU)=(\bsY,\bsZ)$,
$\sigma$ is in $\mathcal{I}'_{j}$.

Likewise, if there are 
$\bsV,\bsW\in {\bsB_{k',j}\choose\bsA_{m,j}}$ such that $\iota(\bsV)=(\bsX,\bsY)$ and $\tau(\bsW)=(\bsY,\bsZ)$,
then for each $\sigma\in\mathcal{I}_{j}$ for which there is a $\bsU\in {\bsB_{k',j}\choose\bsA_{m,j}}$ such that 
$\sigma(\bsU)=(\bsX,\bsZ)$,
$\sigma$ is in $\mathcal{I}'_{j}$.
\end{clm}

\begin{proof}
The proof is immediate from Lemma \ref{lem.1}, the definition of $\mathcal{I}'_{j}$ and the fact that $E$ is an equivalence relation.
\end{proof}

Our strategy at this point is to find an $\eta\in\mathcal{I}'_j$ such that for each interval between two points of $I_j$, for $(\bsX,\bsY)=\eta(\bsA_{m,j})$,
all the members  of  $\bsX$ in that interval are less than all the members of $\bsY$ in that interval.
This will be done in Claim \ref{claim.C}.
That claim will set us up to show that every map in $\mathcal{I}_j$ which fixes the members of $I_j$ is actually in $\mathcal{I}'_j$.

We now give a few more definitions which will aid in the remaining proofs.
Let $q=|I_{j}|$ and enumerate $I_{j}$ in increasing order as $\{i_p:p<q\}$.
Fix the following notation for the intervals of $K_{j}$ determined by the members of $I_{j}$:
Let $I^0=[0,i_0)$,
for each $p<q-1$ let $I^{p+1}=(i_p,i_{p+1})$,
and let $I^q=(i_{q-1},K_{j})$.
Thus, $K_{j}$ is the disjoint union of $I_{j}$  and   the intervals $I^p$, $p\le q$.
Given $\iota\in\mathcal{I}'_j$ and $(\bsX,\bsY)=\iota(\bsA_{m,j})$, 
for $p\le q$ and $k,l\in I^p$,
we say that $(k,l)$ is the  {\em maximal switching pair of $\iota$ in $I^p$} if  the following holds: 
\begin{enumerate}
\item[(a)]
$l=\max\{i\in I^p:\exists i'\in I^p\, \iota(i')> \iota(K_j+i)\}$ and
\item[(b)]
$k=\min\{i'\in I^p:\iota(i')> \iota(K_j+l)\}$.
\end{enumerate}
In words,  $x^k> y^l$ in $(\bsX,\bsY)$, there are no other members of $(\bsX,\bsY)$ between them, and for every $t>l$ in $I^p$, $y^t$ is greater than every member of $\bsX$ in the interval $I^p$.
We point out that  $x^{\max(I^p)}<y^{\min(I^p)}$ in the structure $(\bsX,\bsY)=\iota(\bsA_{m,j})$ if and  only if there is no maximal switching pair for $\iota$ in the interval $I^p$.
This is the configuration we are heading for in Claim \ref{claim.C} below.

For $\iota\in\mathcal{I}'_j$, define the order relation induced by $\iota$,  $\rho_{\iota}:K_{j}\times K_{j}\ra \{<,=,>\}$,  as follows:
For $(k,l)\in K_{j}\times K_j$ and $\rho\in\{<,=,>\}$,  define $\vec{\rho}_{\iota}(k,l)=\rho$ if and only if $(\iota(k), \iota(K_j+l))=\rho$.

\begin{clm}[Maximal switching pair can be switched]\label{claim.NNL}
Let $\tau\in\mathcal{I}_j'$ and $p\le q$, and let $(k,l)$ be the maximal switching pair in $I^p$.
Then there is a $\sigma\in\mathcal{I}_j'$ such that 
for all $(s,t)\in K_j\times K_j\setminus \{(k,l)\}$, $\vec{\rho}_{\sigma}(s,t)=\vec{\rho}_{\tau}(s,t)$,
and  $\vec{\rho}_{\sigma}(k,l)=\, <$.
\end{clm}

\begin{proof}
Let $\tau\in\mathcal{I}_j'$, let $p\le q$, and let $(k,l)$ be the maximal switching pair for $\tau$ in $I^p$.
Let $(\bsX,\bsY)=\tau(\bsA_{m,j})$.
Since $l\not\in I_j$,  Claim \ref{claim.A} implies
 there is an $\iota\in\mathcal{I}'_{j}$ such that $\iota(k)=\iota(K_{j}+k)$ for all $k\in K_{j}\setminus\{l\}$,
and $\iota(l)<\iota(K_{j}+l)$.
Let $(\bsY_*,\bsZ)=\iota(\bsA_{m,j})$.
Let $e:\bsY\ra(\bsX,\bsY)$ be the identity map on $\bsY$,
and let $f:\bsY_*\ra(\bsY_*,\bsZ)$ be the identity map on $\bsY_*$.

By the OPFAP, there is a free amalgamation of $(\bsX,\bsY)$ and $(\bsY_*,\bsZ)$ with the order prescribed by the  order relation $\vec{\rho}$, which is now described.
Let $\bsU$ denote $(\bsX,\bsY)$ and $\bsV$ denote $(\bsY_*,\bsZ)$.
In words, since the only difference between $\bsY_*$ and $\bsZ$ in $\bsV$ is at their $l$-th members, we identify $\bsY_*$ with $\bsY$ and define  $\vec{\rho}$ between $\bsU$ and $\bsY_*$ as $\vec{\rho}_{\tau}$, and additionally we order $x^k<z^l$ and $z^l$ less than the least member of $\bsU$ above $x^k$.
Precisely,
for $(s,t)\in |U|\times|V|$ such that $u^s=y^i$ and $v^t=y_*^i$ for the same $i\in K_j$,
define $\vec{\rho}(s,t)=\, =$.
This induces the relation $\vec{\rho}(s',t)=\vec{\rho}_{\tau}(s',t)$ for all $s'\in |U|$ and $t\in K_j+1\setminus\{ l+1\}$. 
Let $s_k$ be the number in $|U|$ such that $u^{s_k}=x^k$.
Note that $|V|=K_j+1$ and  $v^{l+1}=z^l$.
Define $\vec{\rho}(s_k,l+1)=\, <$,
and define $\vec{\rho}(s_k+1,l+1)=\, >$.
The rest of $\vec{\rho}$ is completely determined by the above relations, since we require $\vec{\rho}$ to respect the linear orders on $|U|$ and $K_j+1$.
That is, we require that  if $\vec{\rho}(s,t)=\, <$, then $\vec{\rho}(s',t')=\, <$ for all $s'\le s$ and $t'\ge t$;
and if $\vec{\rho}(s,t)=\, >$, then $\vec{\rho}(s',t')=\, >$ for all $s'\ge s$ and $t'\le t$.




By the OPFAP,
there is a free amalgamation $(g,h,\bsW)$ of $(\bsY, e,\bsU,f,\bsV)$ respecting the order $\vec{\rho}$.
There is  a copy $\bsW'\cong\bsW$ which is  a substructure of $\bsC_{i,j}$ for some $i$,
since  each member of $\mathcal{K}_j$ embeds into all but finitely many $\bsC_{i,j}$.
Slightly abusing notation, we have that  $(g,h,\bsW')$ is a free amalgamation  of $(\bsY, e,\bsU,f,\bsV)$ respecting the order $\vec{\rho}$.

Let $\bsU'=(\bsX',\bsY')$ denote $g(\bsU)=g(\bsX,\bsY)$  and let 
 $\bsV'=(\bsY_*',\bsZ')$ denote $h(\bsV)=h(\bsY_*,\bsZ)$, substructures of $\bsW'$.
By choosing $C$ small enough within $B$, similarly to the proof in Claim \ref{claim.F}, 
 we may assume that there are $\bsD_{m,j},\bsE_{m,j}\in{\bsB_{i',j}\choose \bsA_{m,j}}$ such that $\tau(\bsD_{m,j})=\bsU'$
and $\iota(\bsE_{m,j})=\bsV'$.
Since $\tau\in\mathcal{I}'_j$ and $\bsU'=\tau(\bsD_{m,j})$ it follows that $\bsX'\, E\, \bsY'$.
Likewise, $\iota\in\mathcal{I}'_j$ and $\bsV'=\iota(\bsE_{m,j})$  imply that $\bsY'_*\, E\, \bsZ'$.
Since $\bsY'$ and $\bsY_*'$ are the same substructure of $\bsW'$ and $E$ is an equivalence relation,
we have $\bsX'\, E\, \bsZ'$.

Let $\sigma$ be any member of $\mathcal{I}_j$ such that 
$\sigma(\bsA_{m,j})\cong(\bsX',\bsZ')$.
Again, by choosing $C$ small enough within $B$, similarly to the proof in Claim \ref{claim.F}, 
we may assume that there is an $\bsF_{m,j}\in{\bsB_{i',j}\choose \bsA_{m,j}}$  such that $\sigma(\bsF_{m,j})=(\bsX',\bsZ')$.
Then $\sigma$ is in $\mathcal{I}'_j$, since $\bsX'\, E\, \bsZ'$.
Note that for any such  $\sigma$, $\vec{\rho}_{\sigma}$ is the same as $\vec{\rho}_{\tau}$ except at the pair $(k,l)$, where now $\vec{\rho}_{\sigma}(k,l)=\, <$.
Thus, there is a $\sigma\in\mathcal{I}_j'$ satisfies the claim.
\end{proof}

\begin{clm}\label{claim.C}
There is an $\eta\in \mathcal{I}'_{j}$ such that
for each $p\le q$,
$\max(\eta'' I^p)<\min(\eta''\{K_{j}+i:i\in I^p\})$,
and 
there are no relations between $\bsA_{m,j}\re \eta''( K_j\setminus I_j)$ and $\bsA_{m,j}\re \eta''\{K_j+k:k\in K_j\setminus I_j\}$.
\end{clm}

\begin{proof}
Let $p\le q$ be given and assume that $I^p$ is nonempty.
Let $\tau\in\mathcal{I}'_j$ and $(k,l)$ be the maximal switching pair  for $\tau$ in the interval $I^p$.
By finitely many applications of Claim \ref{claim.NNL},
there is a $\tau_l\in\mathcal{I}'_j$ such that for all $(s,t)\in K_j\times (K_j\setminus\{l\})$,
$\vec{\rho}_{\tau_l}(s,t)=\vec{\rho}_{\tau}(s,t)$;
for all $s\le \max(I^p)$, $\vec{\rho}_{\tau_l}(s,l)=\, <$;
and for all $s>\max(I^p)$, $\vec{\rho}_{\tau_l}(s,l)=\, >$.
If $l>\min(I^p)$, then the applications of Claim \ref{claim.NNL} constructed $\tau_l$ so that $\tau_l(k)>\tau_l(l-1)$.
Thus,  there is some $k_1\le k$ such that $(k_1,l-1)$ is the maximal switching pair for $\tau_l$ in the interval $I^p$.
By finitely many applications of Claim \ref{claim.NNL}, we obtain a $\tau_{l-1}\in\mathcal{I}'_j$ such that 
 for all $(s,t)\in K_j\times (K_j\setminus\{l-1\})$,
$\vec{\rho}_{\tau_{l-1}}(s,t)=\vec{\rho}_{\tau_l}(s,t)$;
for all $s\le \max(I^p)$, $\vec{\rho}_{\tau_{l-1}}(s,l-1)=\, <$;
and for all $s>\max(I^p)$, $\vec{\rho}_{\tau_{l-1}}(s,l-1)=\, >$.
Continuing in this manner, we eventually obtain a $\sigma^p\in\mathcal{I}'_j$ such that for all $(s,t)\in K_j\times  K_j\setminus I^p\times I^p$, 
$\vec{\rho}_{\sigma^p}(s,t)=\vec{\rho}_{\tau}(s,t)$;
and for all $(s,t)\in I^p\times I^p$, $\vec{\rho}_{\sigma^p}(s,t)=\, <$.

By induction on the intervals to build an $\eta$ as in the claim as follows:
Starting with any $\tau\in\mathcal{I}'_j$, by the previous paragraph,
there is a $\sigma^0\in\mathcal{I}'_j$ 
 such that 
for all $(s,t)\in K_j\times  K_j\setminus I^0\times I^0$, 
$\vec{\rho}_{\sigma^0}(s,t)=\vec{\rho}_{\tau}(s,t)$;
and for all $(s,t)\in I^p\times I^p$, $\vec{\rho}_{\sigma^0}(s,t)=\, <$.
Given $p<q$ and $\sigma^p$, by the previous paragraph, 
there is a $\sigma^{p+1}\in\mathcal{I}'_j$ 
 such that 
for all $(s,t)\in K_j\times  K_j\setminus I^{p+1}\times I^{p+1}$, 
$\vec{\rho}_{\sigma^{p+1}}(s,t)=\vec{\rho}_{\sigma^p}(s,t)$;
and for all $(s,t)\in I^{p+1}\times I^{p+1}$, $\vec{\rho}_{\sigma^{p+1}}(s,t)=\, <$.

Let $\sigma$ denote $\sigma^q$.
Then $\sigma\in\mathcal{I}'_j$ and 
for each $p\le q$,
$\max(\sigma'' I^p)<\min(\sigma''\{K_{j}+i:i\in I^p\})$.
However,  $\sigma(\bsA_{m,j})$ might not be the free amalgamation of $\bsA_{m,j}\re \sigma''K_j$ and $\bsA_{m,j}\re \sigma''[K_j,2K_j)$
over the structure $\bsA_{m,j}\re\sigma'' I_j$.
Take structures $(\bsX,\bsY)\cong(\bsX',\bsZ)\cong\sigma(\bsA_{m,j})$.
By the OPFAP,
there is a free  amalgamation  $\bsW$ of  $(\bsX,\bsY)$ and $(\bsX',\bsZ)$ with the following properties:
 $\bsX$ and $\bsX'$ are sent to the same substructure, call it $\bsX_*$ of $\bsW$ (hence $\bsX\re I_j$ is sent to the same substructure as $\bsX'\re I_j$).
Moreover, letting $\bsY_*,\bsZ_*$ denote the copies of $\bsY$ and $\bsZ$ in $\bsW$,
we have that 
 for all $p\le  q$, for all $i,k,l\in I^p$, $x_*^i<y_*^k<z_*^l$.
Thus, $\bsY_*\re K_j\setminus I_j$ and $\bsZ_*\re K_j\setminus I_j$ have no relations between them in $\bsW$.
Thus, the substructure $(\bsY_*,\bsZ_*)$ in $\bsW$, is the free amalgamation of two copies of $\bsA_{n,j}$ over the substructure $\bsA_{n,j}\re I_j$,
and for each $p\le q$, all the members of $\bsY_*$ in the interval $I^p$ are less than all the members of $\bsZ_*$ in $I^p$.
Since the structure $(\bsY_*,\bsZ_*)$ is an amalgamation of two copies of $\bsA_{n,j}$, there is an $\eta\in\mathcal{I}_j$ such that $\eta(\bsA_{m,j})\cong(\bsY_*,\bsZ_*)$.
Since $\sigma\in\mathcal{I}'_j$,
we have that $\bsX_*\, E\, \bsY_*$ and $\bsY_*\, E\, \bsZ_*$.
Thus, $\bsY_*\, E\, \bsZ_*$.
Therefore, $\eta$ is in $\mathcal{I}'_j$, by Lemma \ref{lem.1}.
This $\eta$ satisfies the claim.
\end{proof}

\begin{clm}\label{claim.D}
Let $k\ge m$.
Given any $\bsX,\bsY\in {\bsC_{k,j}\choose\bsA_{n,j}}$ such that for all $i\in I_{n}$, $x^i=y^i$,
there is a $\tau\in\mathcal{I}'_{j}$ such that for some $\bsW\in{\bsB_{k',j}\choose\bsA_{m,j}}$,
$\tau(\bsW)=(\bsX,\bsY)$.
\end{clm}

\begin{proof}
Let  $\bsX,\bsY\in {\bsC_{k,j}\choose\bsA_{n,j}}$ be  such that for all $i\in I_{j}$, $x^i=y^i$.
Take $\bsZ\in{\bsB_{k',j}\choose\bsA_{n,j}}$ such that
for each $i\in I_{j}$, $z^i=x^i$,
and 
 for each $p\le q$,
$\max \bsZ\re I^p<\min(\bsX\re I^p,\bsY\re I^p)$,
and  any relations between $\bsZ$ and $\bsX$ and any relations between $\bsZ$ and $\bsY$ involve only members of their universes with indices in $I_{j}$.
(By the OPFAP and possibly thinning $C$ again below $B$, such a $\bsZ$ exists.)
Let $\bsV$ be the substructure of $\bsB_{k',j}$ determined by the universe $|\bsX|\cup|\bsZ|$;
and let $\bsW$ be the substructure of $\bsB_{k',j}$ determined by the universe $|\bsY|\cup|\bsZ|$.
Let $\eta$ be the member of $\mathcal{I}'_{j}$ from Claim \ref{claim.C}.
Then both $\bsV$ and $\bsW$ are isomorphic to the structure $\eta(\bsA_{m,j})$.
Since $\eta\in\mathcal{I}'_{j}$,
we have $\bsZ\, E\, \bsX$ and $\bsZ\, E\, \bsY$.
Since $E$ is an equivalence relation, it follows that $\bsX\, E\, \bsY$.
It follows that any $\tau\in\mathcal{I}_{j}$ for which there is a $\bsD\in{\bsB_{k',j}\choose\bsA_{m,j}}$ such that $\tau(\bsD)=(\bsX,\bsY)$ is in $\mathcal{I}'_{j}$.
Therefore,
each $\tau\in\mathcal{I}_{j}$ which is fixed on indices in $I_{j}$ is also in $\mathcal{I}'_{j}$.
\end{proof}

By Claim \ref{claim.D},
the following is immediate.

\begin{clm}\label{claim.3}
For each $k\ge n$, for all $X(n),Y(n)\in\mathcal{R}(n)|C(k)$,
we have $\bsX_{n,j}\, E\, \bsY_{n,j}$ if and only
$\bsX_{n,j}\, E_{I_{j}}\, \bsY_{n,j}$.
\end{clm}

\begin{proof}
Let $k\ge n$ and  $X(n),Y(n)\in\mathcal{R}(n)|C(k)$.
If $\bsX_{n,j}\, E\, \bsY_{n,j}$, then 
 there is an $\iota\in \mathcal{I}'_{j}$ 
and a $\bsU\in{\bsB_{k',j}\choose\bsA_{m,j}}$
such that
$(\bsX_{n,j},\bsY_{n,j})=\iota(\bsU)$, by Lemma \ref{lem.1}.
Since $\iota\in\mathcal{I}'_{j}$, we have that   for each $i\in I_{j}$, $\iota(i)=\iota( K_{n,j}+i)$; hence, for each $i\in I_{j}$, $x^i_{n,j}=y^i_{n,j}$.
Therefore, $\bsX_{n,j}\, E_{I_{j}}\, \bsY_{n,j}$.
Conversely, suppose $\bsX_{n,j}\, E_{I_{j}}\, \bsY_{n,j}$.
Let $\iota\in\mathcal{I}_{j}$ 
and  $\bsW_{m,j}\in{\bsB_{k',j}\choose\bsA_{m,j}}$ be such that $\iota(\bsW_{m,j})=(\bsX_{n,j},\bsY_{n,j})$.
By Claim \ref{claim.D}, $\iota\in\mathcal{I}'_{j}$. 
Thus,   $\bsX_{n,j}\, E\, \bsY_{n,j}$.
\end{proof}

It follows from Claim \ref{claim.3} and Lemma \ref{lem.1}  that for each  $\tau\in\mathcal{I}_j$,
$\tau$ is in $\mathcal{I}'_j$ if and only if  for all $i\in I_j$, $\tau(i)=\tau(K_j+1)$.
This completes the Base Case.
\vskip.1in

\underline{Induction Hypothesis}.
Given $n$ such that $J_n\ge M$,  $N\le M$, and $L\in [J_n]^N$,
letting $E$ be an equivalence relation such that  $E\, \sse \bigcup_{k\ge n} {(\bsA_{k,j})_{j\in L} \choose (\bsA_{n,j})_{j\in L}}\times{(\bsA_{k,j})_{j\in L} \choose (\bsA_{n,j})_{j\in L}}$,
the following hold.
Fix  any $m$  large enough that for each $j\in L$, all 
amalgamations 
 of two  copies of $\bsA_{n,j}$ can be embedded into $\bsA_{m,j}$, and let  $C\le B$ and $\mathcal{I}'\sse\mathcal{I}$ be obtained as in Lemma \ref{lem.1}
and similarly as in Claims \ref{claim.F'} and \ref{claim.F}.
Letting, for $j\in L$,
\begin{equation}
I_j=\{i\in K_{j}:\forall \iota_j\in\mathcal{I}'_j, \iota_j(i)=\iota_j(K_{j}+i)\},
\end{equation}
the following hold:
\begin{enumerate}
\item[(a)]
$\mathcal{I}'=\Pi_{j\in L}\mathcal{I}'_j$,
where for each $j\in L$, $\mathcal{I}'_j=\{\iota_j:\iota\in\mathcal{I}'\}$.
\item[(b)]
When restricted below $C$,
$E=E_{(I_j)_{j\in L}}$.
\end{enumerate}

\underline{Induction Step}.
$M+1$.
Let $n<\om$ be such that $M+1\ge J_n$.
 Let $L\in [J_n]^{M+1}$, 
and let $E$ be an equivalence relation such that  $E\sse \bigcup_{k\ge n} {(\bsA_{k,j})_{j\in L} \choose (\bsA_{n,j})_{j\in L}}\times{(\bsA_{k,j})_{j\in L} \choose (\bsA_{n,j})_{j\in L}}$.
Let $m$ be large enough that for each $j\in L$, all  amalgamations  of two  copies of $\bsA_{n,j}$ can be embedded into $\bsA_{m,j}$.
Let $l=\max(L)$, and 
let $L'=L\setminus\{l\}$.
We start by fixing $B,C\in\mathcal{R}$,  with $C\le B$,  and  $\mathcal{I}'$ satisfying Lemma  \ref{lem.1}.
For each $j\in L$, let $\mathcal{I}'_j$ denote the collection of those $\iota_j\in\mathcal{I}_j$ for which there exists a $\tau=(\tau_k:k\in L)\in\mathcal{I}'$ such that $\iota_j=\tau_j$.

For  each $W\in\mathcal{R}$ and each $K_l\sse M_{l}:=\|\bsA_{m,l}\|$ such that $\bsA_{m,l}\re K_l\cong\bsA_{n,l}$,
by the induction hypothesis there is 
a $V\le W$ such that for each $X\le V$, $E$ restricted to the copies of $(\bsA_{n,j})_{j\in L}$ in $(\bsX_{m,j})_{j\in L'}{}^{\frown}(\bsX_{m,l}\re K_l)$ is canonical.
By the OPFAP and the definition of generating sequence, 
for any $K_l,K'_l\sse M_{l}$
satisfying $\bsA_{m,l}\re K_l\cong \bsA_{m,l}\re K'_l\cong\bsA_{n,l}$,
there is a $p$ large enough so that there are structures
 $\bsY_{m,l},\bsZ_{m,l}$ in ${\bsA_{p,l}\choose \bsA_{m,l}}$ with
$\bsY_{m,l}\re K_l=\bsZ_{m,l}\re K'_l$.
Thus, possibly thinning $B$, we have that 
the canonical equivalence relation on 
${(\bsB_{p,j})_{j\in L'}{}^{\frown}(\bsD_{n,l})
\choose(\bsA_{n,j})_{j\in L}}$
is the same for all $p\ge m$ and each fixed $\bsD_{n,l}\in{\bsB_{p,l}\choose\bsA_{n,l}}$.

Let $\mathcal{T}_{L'}$ denote the collection of $\tau=(\tau_j:j\in L')$ (where each $\tau_j\in\mathcal{I}_j$)
which give the canonical equivalence relation on 
${(\bsB_{p,j})_{j\in L'}{}^{\frown}(\bsC_{n,l})
\choose(\bsA_{n,j})_{j\in L}}$, for each $p\ge m$ and 
$\bsC_{n,l}\in{\bsB_{p,l}\choose\bsA_{n,l}}$.
By (a) of the induction hypothesis,
$\mathcal{T}_{L'}=\Pi_{j\in L'}\mathcal{T}_j$,
where for each $j\in L'$, $\mathcal{T}_j=\{\tau_j:\tau\in\mathcal{T}_{L'}\}$.
For each $j\in L'$,
let $H_j=\{i\in K_{j}:\forall \tau\in\mathcal{T}_{L'},\ 
\tau_j(i)=\tau_j(K_{j}+i)\}$.
By (b) of the induction hypothesis, below $B$ the canonical equivalence relation when the $l$-th coordinate is fixed is $E_{(H_j)_{j\in L'}}$.

Likewise, 
for  each $W\in\mathcal{R}$ and each collection $K_j\sse M_{j}:=\|\bsA_{m,j}\|$ ($j\in L'$) such that $\bsA_{m,j}\re K_j\cong\bsA_{n,j}$,
by the induction hypothesis, there is 
a $V\le W$ such that for each $X\le V$, $E$ restricted to the copies of $(\bsA_{n,j})_{j\in L}$ in $(\bsX_{m,j}\re K_j)_{j\in L'}{}^{\frown}(\bsX_{m,l})$ is canonical.
By the OPFAP and the  definition of a generating sequence,
it follows that for each $W\in\mathcal{R}$, there is a $V\le W$ such that for all $j\in L'$, whenever  $K_j,K'_j\sse M_{j}$
satisfy $\bsA_{m,j}\re K_j\cong\bsA_{m,j}\re K'_j\cong\bsA_{n,j}$,
then
there are $\bsY_{m,j},\bsZ_{m,j}\in{\bsW_{p,j}\choose\bsA_{n,j}}$, for some $p\ge m$,
such that $\bsY_{m,j}\re K_j=\bsZ_{m,j}\re K'_j$.
Thus, possibly thinning $B$, we may assume that 
the equivalence relation is the same canonical one on each set
${(\bsD_{n,j})_{j\in L'}{}^{\frown}(\bsB_{p,l})
\choose(\bsA_{n,j})_{j\in L}}$
 for all $p\ge m$ and each fixed $(\bsD_{n,j})_{j\in L'}\in{(\bsB_{p,j})_{j\in L'}\choose(\bsA_{n,j})_{j\in L'}}$.

Let $\mathcal{T}_{l}$ denote the collection of $\tau_{l}\in\mathcal{I}_l$ 
which give the canonical equivalence relation on 
the copies of $(\bsA_{n,j})_{j\in L}$ in 
${(\bsD_{n,j})_{j\in L'}{}^{\frown}(\bsB_{p,l})
\choose(\bsA_{n,j})_{j\in L}}$
 for all $p\ge m$ and each fixed $(\bsD_{n,j})_{j\in L'}\in{(\bsB_{p,j})_{j\in L'}\choose(\bsA_{n,j})_{j\in L'}}$.
Let $H_{l}=\{i\in K_{l}:\forall \tau_{l}\in\mathcal{T}_{l},\ 
\tau_{l}(i)=\tau_l(K_{l}+i)\}$.
Thus, below $B$, the canonical equivalence relation when the $l$-th coordinate is fixed is $E_{H_{l}}$.

By the induction hypothesis, $\mathcal{T}_{L'}=\prod_{j\in L'}\mathcal{T}_j$.
Thus, 
 $\mathcal{T}_{L'}\times\mathcal{T}_l=\prod_{j\in L}\mathcal{T}_j$.
Moreover, each $I_j$ must be contained in $H_j$, for each $j\in L$.

\begin{clm}\label{clm.3}
$\prod_{j\in L}\mathcal{T}_j\sse \mathcal{I}'$.
Hence, below $C$,
$E_{(H_{j})_{j\in L}}\sse E$.
\end{clm}

\begin{proof}
Given any $\tau=(\tau_j:j\in L)\in \prod_{j\in L}\mathcal{T}_j$ and $((\bsX_{n,j})_{j\in L},(\bsY_{n,j})_{j\in L})=\tau((\bsC_{m,j})_{j\in L})$,
we see that $(\bsX_{n,j})_{j\in L}\, E\, 
{(\bsX_{n,j})_{j\in L'}}^{\frown}\bsY_{n,l}\, E\, 
(\bsY_{n,j})_{j\in L}$.
Thus, $\tau\in\mathcal{I}'$.

Suppose that 
$(\bsX_{n,j})_{j\in L},(\bsY_{n,j})_{j\in L}\in {(\bsC_{p,j})_{j\in L}\choose(\bsA_{n,j})_{j\in L}}$
satisfy 
$(\bsX_{n,j})_{j\in L}\, E_{(H_{j})_{j\in L}}\, (\bsY_{n,j})_{j\in L}$, where $p\ge n$.
Let $\bsZ_{n,j}=\bsX_{n,j}$ for each $j\in L'$, and let $\bsZ_{n,l}=\bsY_{n,l}$.
Then $(\bsX_{n,j})_{j\in L}\, E$ $(\bsZ_{n,j})_{j\in L}$, and
$(\bsZ_{n,j})_{j\in L}\, E\, (\bsY_{n,j})_{j\in L}$.
Thus, 
 $(\bsX_{n,j})_{j\in L}\, E\, (\bsY_{n,j})_{j\in L}$, by transitivity of $E$.
\end{proof}

\begin{clm}\label{claim.10new}
Below $C$,
$\mathcal{I}'\sse\prod_{j\in L}\mathcal{T}_j$.
\end{clm}

\begin{proof}
Let $\iota:=(\iota_j:j\in L)\in\mathcal{I}'$ and let $((\bsX_{n,j})_{j\in L},(\bsY_{n,j})_{j\in L})=\iota((\bsC_{m,j})_{j\in L})$.
Then $(\bsX_{n,j})_{j\in L}\, E\, (\bsY_{n,j})_{j\in L}$.
Fixing $((\bsX_{n,j})_{j\in L'},(\bsY_{n,j})_{j\in L'})$ and running the arguments for the Base Case on coordinate $l$,
we obtain Claim \ref{claim.C} on $\bsC_{m,l}$.
Take $\eta_l\in\mathcal{T}_l$ as in Claim \ref{claim.C}.
Take  a $\bsZ_{n,l}\in{\bsB_{m',l}\choose \bsA_{n,l}}$ such that 
both $(\bsX_{n,l},\bsZ_{n,l})=\eta_l(\bsV_{m,l})$
and $(\bsY_{n,l},\bsZ_{n,l})=\eta(\bsW_{m,l})$
for some $\bsV_{m,l},\bsW_{m,l}\in{\bsB_{m',l}\choose\bsA_{n,l}}$.
In particular, 
$\bsX_{n,l}\re H_l=\bsZ_{n,l}\re H_l=\bsY_{n,l}\re H_l$.
It follows  that $(\bsX_{n,j})_{j\in L'}\, E\, {(\bsX_{n,j})_{j\in L'}}^{\frown}\bsZ_{n,l}$,
and 
$(\bsY_{n,j})_{j\in L'}\, E\, {(\bsY_{n,j})_{j\in L'}}^{\frown}\bsZ_{n,l}$.
Therefore, 
${(\bsX_{n,j})_{j\in L'}}^{\frown}\bsZ_{n,l}\, E\, 
{(\bsY_{n,j})_{j\in L'}}^{\frown}\bsZ_{n,l}$,
which implies that $(\iota_j:j\in L')\in\mathcal{T}_{L'}$.
Hence, for each $j\in L'$,
$\iota_j$ is in $\mathcal{T}_j$.

By a similar argument, say fixing
$((\bsX_{n,j})_{j\in L\setminus\{0\}},(\bsY_{n,j})_{j\in L\setminus\{0\}})$, we find that $\iota_l$ is in $\mathcal{T}_l$.
Therefore, $\iota\in\prod_{j\in L}\mathcal{T}_j$.
\end{proof}

Therefore, $\mathcal{I}'=\prod_{j\in L}\mathcal{T}_j$.
Hence, below $C$, $E$ is given by $E_{(H_j)_{j\in L}}$.
We conclude by showing that the induction hypotheses (a) and (b) are satisfied for this stage.

If $\iota_j$ is in $\mathcal{I}'$, then there was some $\tau=(\tau_l:l\in L)\in\mathcal{I}'$  such that $\iota_j=\tau_j$.
Since $\mathcal{I}'=\prod_{j\in L}\mathcal{T}_j$, $\iota_j$ must be in $\mathcal{T}_j$. 
Conversely, if $\tau_j\in\mathcal{T}_j$, then taking any $\tau_l\in \mathcal{T}_l$ for $l\in L\setminus\{j\}$,
we have that $(\tau_l:l\in L)$ is in $\mathcal{I}'$.
Thus, $\tau_j$ is in $\mathcal{I}'_j$.
Therefore, $\mathcal{I}'=\prod_{l\in L} \mathcal{I}'_l$, so (a) holds.

Define $I_j$  to be the set of all $i\in K_j$ such that for all $\iota_j\in \mathcal{I}'_j$, $\iota_j(i)=\iota_j(K_j+i)$. 
It was shown above that each $I_j\sse H_j$.
Suppose there is an $i\in H_j\setminus I_j$.
There is an $(\iota_l:l\in L)\in \mathcal{I}'$ such that $\iota_j(i)\ne K_j+i$.
Every $(\tau_l:l\in L)\in\prod_{l\in L}\mathcal{T}_l$
must have $\tau_j(i)=\tau_j(K_j+i)$.
But $\prod_{l\in L}\mathcal{T}_l$ equals $\mathcal{I}'$, a contradiction.
Therefore, each $I_j=H_j$, hence (b) holds.

This finishes the proof of the theorem.
\end{proof}

\begin{rem}
Soki\v{c} has pointed out that it seems to be sufficient to assume the weaker Order-preserving Strong Amalgamation Property,
the same definition as OPFAP except  only requiring the amalgamations  to be strong, not necessarily free.
\end{rem}


\section{General Ramsey-classification theorem for
topological Ramsey spaces constructed from  generating sequences}\label{sec.canon}

We  prove a general Ramsey-classification theorem, Theorem \ref{canonical R},
 for equivalence relations on fronts
 for the class of the topolgical Ramsey spaces introduced in Section \ref{sec.genseq}, where the \Fraisse\ classes have the OPFAP.
Theorem \ref{canonical R} extends Theorem  4.14
from \cite{Dobrinen/Todorcevic11} for canonical equivalence relations on the space $\mathcal{R}_1$ to the more general class of topological Ramsey spaces constructed from a generating sequence.
As the proof here closely follows that in 
\cite{Dobrinen/Todorcevic11}, we shall omit those  proofs which follow by   straightforward  modifications of arguments in that paper.
The essential new ingredient here is that 
the building blocks for Theorem \ref{canonical R} are the canonical equivalence relations  from Theorem \ref{FiniteERProducts}, and handling  this shall require some care.

Throughout this section, 
let $1\le J\le \om$, and $\mathcal{K}_j$, $j\in J$, be \Fraisse\ classes of finite ordered relational structures with the Ramsey property and the Order-Prescribed Free Amalgamation Property.
Let
$\lgl \bsA_k:k<\om\rgl$ be a fixed generating sequence,
and let $\mathcal{R}$ denote the topological Ramsey space $\mathcal{R}(\lgl \bsA_k:k<\om\rgl)$.
Recall that for $j\in J_k$,  $K_{k,j}$ denotes the cardinality of  the structure $\bsA_{k,j}$, and 
 for any structure $\bsB_{k,j} \cong \bsA_{k,j}$,
we let $\{b^i_{k,j}:i<K_{k,j}\}$ denote the enumeration of the universe of $\bsB_{k,j}$ in increasing order.

\begin{defn}[Canonical projection maps on blocks]\label{projections}
Let $k<\om$ be given.
For $\bsB_{k,j}\in
{\bsA_{n,j} \choose \bsA_{k,j}}$
and $I\sse K_{k,j}$, 
let 
 $\pi_I(\bsB_{k,j})= \bsB_{k,j}\upharpoonright \{b_{k,j}^i : i\in I\}$,   
the substructure of $\bsB_{k,j}$ with universe $\{b_{k,j}^i : i\in I\}$.

For  $B(k) = \lgl n,(\bsB_{k,j})_{j\in J_k}\rgl\in\mathcal{R}(k)$, 
we define the following projection maps.
Given  $I_{k,j}\subseteq K_{k,j}$, $j\in J_k$,  let
 \begin{equation}
\pi_{(I_{k,j})_{j\in J_k}}(B(k)) = \lgl n,  (\pi_{I_{k,j}}(\bsB_{k,j}))_{j\in J_k}\rgl,
\end{equation}
and
let 
\begin{equation}
\pi_{<>}(B(k)) = \lgl\rgl,
\end{equation}
 where $\lgl\rgl$ denotes the empty sequence.

We slightly abuse notation by associating $\lgl n, (\emptyset)_{j\in J_k}\rgl$ with $\lgl n\rgl$.
We define the depth projection map as
\begin{equation}
\pi_{\depth}(B(k)) = \lgl n\rgl,
\end{equation}
the depth of $B(k)$ in $\mathbb A$.
Then
when $I_{k,j}=\emptyset$ for all $j\in J_k$,
we associate  $\pi_{(I_{k,j})_{j\in J_k}}(B(k))$  with  $\pi_{\depth}(B(k))$.
Let 
\begin{equation}
\Pi(k) = \{\pi_{<>}\}\cup\{\pi_{(I_{k,j})_{j\in J_k}} : \forall j\in J_k,\ I_{k,j}\subseteq K_{k,j}\}.
\end{equation}
\end{defn}

The canonical equivalence relations on blocks are induced by the canonical projection maps as follows.

\begin{defn}[Canonical equivalence relations on blocks]\label{defn.canoneqblock}
Let $k<\om$, 
and $B(k), C(k)\in\mathcal{R}(k)$.
For  $I_{k,j}\subseteq K_{k,j}$, $j\in J_k$,
define
\begin{equation}
B(k)\, E_{(I_{k,j})_{j\in J_k}}\, C(k)
\longleftrightarrow
\pi_{(I_{k,j})_{j\in J_k}}(B(k)) = \pi_{(I_{k,j})_{j\in J_k}}(C(k)).
\end{equation}
Define
\begin{equation}
B(k)\, E_{<>}\, C(k) \longleftrightarrow \pi_{<>}(B(k)) = \pi_{<>}(C(k)).
\end{equation}
Thus, 
$E_{<>} = \mathcal R(k)\times\mathcal R(k)$.
We also define
\begin{equation} 
B(k)\, E_{\depth}\, C(k)\longleftrightarrow \pi_{\depth}(B(k)) = \pi_{\depth}(C(k)).
\end{equation}
When $I_{k,j}=\emptyset$ for all $j\in J_k$,
then $E_{\depth}$ is a simplified notation for  $E_{(I_{k,j})_{j\in J_k}}$, as  in this case, they are the same equivalence relation.

The collection of {\em canonical equivalence relations} on $\mathcal{R}(k)$ is
 \begin{equation}
\mathcal E(k) =\{E_{<>}\}\cup\{E_{(I_{k,j})_{j\in J_k}} : \forall j\in J_k,\ I_{k,j}\subseteq K_{k,j}\}.
\end{equation}
\end{defn}

For the following definitions, let $X\in\mathcal{R}$, $\mathcal{F}$ be a front on $[\emptyset, X]$, and
 $\vp$ be a function on $\mathcal{F}$.

\begin{defn}\label{defn.innerphi}
We shall say that $\vp$  is {\em inner} 
if for each $b\in\mathcal{F}$, 
$\vp(b)=\bigcup_{i<|b|}\pi_{r_i(b)}(b(i))$,
where each $\pi_{r_i(b)}$ is some member of $\Pi(i)$.
\end{defn}

Thus, for $b=\lgl \lgl n_0,(\bsB_{0,j})_{j\in J_0}\rgl,\dots,\lgl n_{k-1},(\bsB_{{k-1},j})_{j\in J_{k-1}}\rgl\rgl$,
$\vp(b)=\{\lgl\rgl\}\cup\{\lgl n_l, (\bsC_{l,j})_{j\in J_l}\rgl:l\in L\}$,
for  some subset  $L\sse k$, and some
 (possibly empty) substructures
$\bsC_{l,j}$  of $\bsB_{l,j}$.
That is, $\vp$ is inner if it picks out a subsequence of substructures from a given $b$.

For $l<|b|$, let $\vp(b)\re r_l(b)$ denote 
$\bigcup_{i<l}\pi_{r_i(b)}(b(i))$,
the {\em initial segment} of $\vp(b)$ which is obtained from $r_l(b)$.
For $b,c\in\mathcal{F}$, we shall say that $\vp(c)$ is a {\em proper initial segment} of $\vp(b)$, and  write $\vp(c)\sqsubset \vp(b)$, if there is an $l<|b|$ such that 
$\vp(c)=\vp(b)\re r_l(b)\ne \vp(b)$.

\begin{defn}
An inner map $\vp$ is {\em Nash-Williams} if
whenever $b,c\in\mathcal{F}$ and $\vp(b)\ne\vp(c)$,
then $\vp(c)\not\sqsubset\vp(b)$.

An equivalence relation $R$ on $\mathcal{F}$ is {\em canonical}  if there is an
 inner, Nash-Williams map $\vp$ on $\mathcal{F}$ such that 
for all $b,c\in\mathcal{F}$, $b\, E\, c \Llra \vp(b)=\vp(c)$.
\end{defn}

\begin{rem}
As for the topological Ramsey spaces considered in \cite{Dobrinen/Todorcevic11} and \cite{Dobrinen/Todorcevic12}, 
here too there can be different inner, Nash-Williams maps which represent the same canonical equivalence relation.
However, there will be one maximal such map, maximality being with respect to the embedding relation on the \Fraisse\ classes.
(See Remark 4.23 and Example 4.24 
in \cite{Dobrinen/Todorcevic11}.)
As the maximal such map is the one useful for classifying the initial Tukey structures below an ultrafilter associated with a Ramsey space, we consider the maximal inner Nash-Williams map to be the canonizing map.
\end{rem}

We fix the following notation, useful in this and subsequent sections.

\begin{notation}
For  $X\in\mathcal{R}$ and $b,s, t\in\mathcal{AR}$,
fix the following notation.
If $s\sqsubseteq b$, write $b\setminus s\le_{\fin} X$ if 
the blocks in $b$ not in $s$ all come from blocks of $X$;
 precisely,
if
 $b\setminus s =\lgl \lgl n_l,(\bsB_{l,j})_{j\in J_l}\rgl:k\le l<m\rgl$,
then for each $k\le l<m$, there is a block $\lgl n_l, (\bsX_{p,j})_{j\in J_p}\rgl =X(p)$, for some $p$,  such that $(\bsB_{l,j})_{j\in J_l}\le (\bsX_{p,j})_{j\in J_l}$.
For a set $\mathcal{F}\sse\mathcal{AR}$, define the following.
 Let $\mathcal F_s = \{b\in\mathcal F : s\sqsubseteq b\}$,
$\mathcal{F}_s|X=\{b\in\mathcal{F}_s:b\setminus s\le_{\fin} X\}$.

For $b(k)=\lgl n_k,(\bsB_{k,j})_{j\in J_k}\rgl\in\mathcal{R}(k)$,
we let
 $\depth_{\mathbb A}(b(k))=n_k$,
the depth of the block $b(k)$ in $\mathbb A$.
For $\mathcal{F}\sse\mathcal{AR}$,
define $\mathcal{F}_s|X/t=\{b\in \mathcal{F}_s|X: 
\depth_{\mathbb A}(b(|s|))>\depth_{\mathbb A}(t)\}$.
Similarly, let $r_{|s|+1}[s,X]/t=\{b\in r_{|s|+1}[s,X]:\depth_{\mathbb A}(b(|s|))>\depth_{\mathbb A}(t)\}$.
\end{notation}

It is important to note that $b$ being in $\mathcal{F}_s|X$
 or  $\mathcal{F}_s|X/t$ does not imply that $b\le_{\fin} X$; it only means that the blocks of $b$ above $s$ comes from within $X$.

 The next theorem is the main result of this section.

\begin{thm}\label{canonical R}
Let $1\le J\le \om$, and $\mathcal{K}_j$, $j\in J$, be \Fraisse\ classes of finite ordered relational structures with the Ramsey property and the Order-Prescribed Free Amalgamation Property.
Let
$\lgl \bsA_k:k<\om\rgl$ be a fixed generating sequence,
and let $\mathcal{R}$ denote the topological Ramsey space $\mathcal{R}(\lgl \bsA_k:k<\om\rgl)$.

Suppose $A\in\mathcal{R}$, $\mathcal F$ is a front on $[\emptyset,A]$, and $R$ is an equivalence relation on $\mathcal F$. Then there exists $C\leq A$ such that $R$ is canonical on $\mathcal F|C$.
\end{thm}

\begin{proof}
Let $f:\mathcal F\rightarrow \mathbb N$ be any map which induces $R$. 
We begin by reviewing the concepts of mixing and separating, first introduced in \cite{Proml/Voigt85} and used in a more general form in \cite{Dobrinen/Todorcevic11} and \cite{Dobrinen/Todorcevic12}.
Let  $\hat{\mathcal F}$ denote
$\{r_n(b) : b\in\mathcal F, n\leq |b|\}$,
the collection of all initial segments of members of $\mathcal{F}$. 
For  $s,t\in\hat{\mathcal{F}}$,
we shall say that
 $X$ {\em separates $s$ and $t$} if and only
for all  
$b\in\mathcal{F}_s|X/t$ and $c\in\mathcal{F}_t|X/s$,
$f(b)\ne f(c)$.
 $X$ {\em mixes $s$ and $t$} if and only if there is no $Y\le X$ which separates $s$ and $t$.
$X$ {\em decides for $s$ and $t$} if and only if either $X$ separates $s$ and $t$, or else $X$ mixes $s$ and $t$.

The proofs of the following Lemmas \ref{lem.transmix} - \ref{lema decides 2} are omitted, as they are the same as the proofs of the analogous statements in \cite{Dobrinen/Todorcevic11}.

\begin{lem}[Transitivity of mixing]\label{lem.transmix}
 For every $X\in\mathcal R$ and every $s,t,u\in\hat{\mathcal F}$, if $X$ mixes $s$ and $t$ and $X$ mixes $t$ and $u$, then $X$ mixes $s$ and $u$.
\end{lem}

Since mixing is trivially reflexive and symmetric, it is an equivalence relation.
We shall say that a property $P(s,X)$  ($s\in\mathcal{AR},X\in\mathcal{R}$) is {\em hereditary} if
whenever $P(s,X)$ holds, then $P(s,Y)$ holds for all $Y\le X$.
Likewise,  $P(s,t,X)$ is  {\em hereditary} if 
whenever $P(s,t,X)$ holds, then $P(s,t,Y)$ holds for all $Y\le X$.

\begin{lem}[Diagonalization for Hereditary Properties]\label{diagonalization argument}
\begin{enumerate}
\item 
Suppose $P(\cdot,\cdot)$ is a 
hereditary property, and  that for  every $X\in\mathcal R$ and every $s\in\mathcal{AR}|X$, there exists $Y\leq X$ such that $P(s,Y)$. 
Then for every $X\in\mathcal R$ there exists $Y\leq X$ such that $P(s,Z)$ holds, for every $s\in\mathcal{AR}|Y$ and every $Z\leq Y$.
\item 
Suppose $P(\cdot,\cdot, \cdot)$ is a hereditary property, and that for every $X\in\mathcal R$ and  all $s,t\in\mathcal{AR}|X$, there exists $Y\leq X$ such that $P(s,t,Y)$ holds.
 Then, for every $X\in\mathcal R$ there exists $Y\leq X$ such that $P(s,t,Z)$ holds, for all $s,t\in\mathcal{AR}|Y$ and every $Z\leq Y$.
\end{enumerate}
\end{lem}

\begin{lem}\label{lema decides 2}
For each $A\in\mathcal R$ there is a $B\leq A$ such that $B$ decides for all $s,t\in\hat{\mathcal F}|B$.
\end{lem}

Possibly shrinking $A$, we may assume that $A\in\mathcal R$ satisfies Lemma \ref{lema decides 2}.
We now introduce some notation useful for arguments applying the Nash-Williams Theorem.

\begin{notation}\label{notn.leftmostcopy}
For $i\le k<\om$, we define the projection map  $\pi_{\bsA_i}:\mathcal{R}(k)\ra\mathcal{R}(i)$ as follows.
For $X(k)=\lgl n_k,(\bsX_{k,j})_{j\in J_k}\rgl \in\mathcal R(k)$,
let 
 $\pi_{\bsA_i}(X(k))$ denote 
$\lgl n_k,(\bsY_{i,j})_{j\in J_i}\rgl$,
where 
for each $j\in J_i$,
$\bsY_{i,j}$ is  the projection of $\bsX_{k,j}$ to 
the lexicographic leftmost copy of $\bsA_{i,j}$ within  $\bsX_{k,j}$.
\end{notation}

\begin{clm}\label{mixed iff related}
 For each $s\in (\hat{\mathcal F}\setminus\mathcal F)|A$ and each $X\leq A$, there is a $Z\leq X$ and an equivalence relation $E_s\in\mathcal E(|s|)$ such that 
the following holds:
Whenever
$x,y\in \mathcal{R}(|s|)|Z/s$,
 letting $a=s^{\frown}x$ and $b=s^{\frown}y$, 
 we have that
$Z$ mixes $a$ and $b$ if and only if $x\,E_s\, y$.
\end{clm}

\begin{proof}
Let $n=|s|$ and $X\le A$ be given. 
Let $R_s$ be the following relation on $\mathcal R(n)|A/s$. For all $x,y\in \mathcal{R}(n)|A/s$,
\begin{equation}
x\, R_s\, y\  \Llra\ A\ \mbox{mixes } s^{\frown}x \mbox{ and }\ s^{\frown} y.
\end{equation}
Define 
$\mathcal X = \{Y\le X: A$ mixes $s^{\frown}Y(n)$ 
 and  $s^{\frown}\pi_{\bsA_n}(Y(n+1))\}$. 
By the Abstract Nash-Williams Theorem, there is a $Y\le X$ such that $[\emptyset,Y]\subseteq\mathcal X$ or $[\emptyset,Y]\cap\mathcal X = \emptyset$.

Suppose $[\emptyset,Y]\subseteq\mathcal X$. 
Then for all 
$x,y\in\mathcal{R}(n)|Y/s$,
we have $x\, R_s\, y$.
Fix $x,y\in\mathcal{R}(n)|Y/s$,
let $a=s^{\frown}x$, $b=s^{\frown}y$, 
 and take $Z_1,Z_2\le Y$ such that
$Z_1(n)=x$,
$Z_2(n)=y$,
and $Z_1(n+1)=Z_2(n+1)$.
Then $x\,R_s\,y$ follows from the fact that $Z_1,Z_2\in\mathcal X$ and by transitivity of mixing.
 In this case the proof of the claim finishes by taking $Z=Y$ and $E_s=E_{<>}$.

Suppose now that $[\emptyset,Y]\cap\mathcal X = \emptyset$. 
Then for all 
$x,y\in\mathcal{R}(n)|Y/s$,
 we have
$x\, R_s\, y\rightarrow\ \depth_{\mathbb A}(x)=\depth_{\mathbb A}(y)$.
Let $m$ be large enough that all possible configurations of isomorphic copies of $(\bsA_{n,j})_{j\in J_n}$ can be embedded into $(\bsA_{m,j})_{j\in J_n}$.  
Let $\mathcal{I}_n$ denote the collection of all sequences $(I_{n,j})_{j\in J_n}$, where each $I_{n,j}\sse K_{n,j}$.
(Recall that $K_{n,j}$ is the cardinality of the structure $\bsA_{n,j}$ from the fixed generating sequence.)
For each $\mathbb I\in\mathcal{I}_n$,
 define
\begin{equation}
\mathcal Y_\mathbb I  = \{Z\le Y : \forall x,y\in \mathcal{R}(n)|Z(m)/s\ 
(A \mbox{ mixes } s^{\frown} x\mbox{ and } s^{\frown}y\lra \pi_{\mathbb I } (x)=\pi_{\mathbb I } (y))\}
\end{equation}
Let $\mathcal Y = [\emptyset,Y]\setminus\bigcup_{\mathbb I}\mathcal Y_{\mathbb I}$. 
Notice that $\mathcal Y$ along with the $\mathcal Y_\mathbb I$, $\mathbb I\in\mathcal{I}_n$, form a finite clopen cover of $[\emptyset,Y]$.
 By the Abstract Nash-Williams Theorem, there is $Z\le Y$ such that either $[\emptyset,Z]\subseteq \mathcal Y_{\mathbb I}$ for some $\mathbb I\in\mathcal{I}_n$, or else $[\emptyset,Z]\subseteq \mathcal Y$. 
By Theorem \ref{FiniteERProducts}, the latter case is impossible.
Thus, fix $Z\le Y$ and $\mathbb I\in\mathcal{I}_n$ such that $[\emptyset,Z]\subseteq \mathcal Y_{\mathbb I}$. 
If at least one of the $I_{n,j}$'s is nonempty then  $E_s=E_{\mathbb I}$.
Otherwise, $E_s=E_{\depth}$.
\end{proof}

The following is obtained from  Claim \ref{mixed iff related} and  Lemma \ref{diagonalization argument}.

\begin{clm}\label{claim decide}
 There is a $B\leq A$ such that for each $s\in (\hat{\mathcal F}\setminus\mathcal F)|B$, there is an equivalence relation $E_s\in\mathcal E(|s|)$ satisfying the following:
 For all $a,b\in r_{|s|+1}[s,B]$, $B$ mixes $a$ and $b$ if and only if $a(|s|)\,E_s\,b(|s|)$.
\end{clm}

Fix $B$ as in Claim \ref{claim decide}. 
For $s\in (\hat{\mathcal F}\setminus\mathcal F)|B$ and $n=|s|$, let $E_s$ denote the member of $\mathcal E(n)$  as guaranteed by Claim \ref{claim decide}. 
We say that $s$ is $E_s$-\textit{mixed} by $B$; that is,
 for all $a,b\in r_{n+1}[s,B]$, $B$ mixes $a$ and $b$ if and only if $a(n)\,E_s\,b(n)$. 
Let $\pi_s$ denote the projection which defines $E_s$. 
 Given $a\in\mathcal F|B$, define 
\begin{equation}
\varphi(a) = \bigcup_{i<|a|}\pi_{r_i(a)}(a(i)).
\end{equation}

The proof of the next claim follows in a straightforward manner from the definitions.
We omit the proof, as it is essentially the same as the proof of Claim 4.17 of \cite{Dobrinen/Todorcevic11}.

\begin{clm}\label{claim 3}
The following are true for all $X\leq B$ and all $s,t\in\hat{\mathcal F}| B$.
\begin{enumerate}
\item 
Suppose $s\notin\mathcal F$. Given $a,b\in r_{|s|+1}[s,X]$, if $X$ mixes $a$ and $t$, and $X$ also mixes $b$ and $t$, then $a(|s|)\,E_s\, b(|s|)$. 
\item
 If $X$ separates $s$ and $t$, then for every $a\in\hat{\mathcal F}\cap r_{|s|+1}[s,X]/t$ and every $b\in \mathcal F\cap r_{|t|+1}[t,X]/s$, $X$ separates $a$ and $b$.
\item
 Suppose $s\notin\mathcal F$. 
Then $\pi_s=\pi_{<>}$ if and only if $X$ mixes $s$ and $a$, for all $a\in \hat{\mathcal F}\cap r_{|s|+1}[s,X]$.
\item 
Suppose $s\notin\mathcal F$. 
Then $\pi_s=\pi_{\depth}$ if and only if for all $a,b\in \hat{\mathcal F}\cap r_{|s|+1}[s,X]$, if $X$ mixes  $a$ and  $b$ then $\depth_X(a)=\depth_X(b).$
\item
 If $s\sqsubseteq t$ and $\varphi(s)=\varphi(t)$, then $X$ mixes $s$ and $t$. 
\end{enumerate}
\end{clm}

The next proposition is the crucial step in the proof of the theorem.
It follows the same outline as  Claim 4.18 of \cite{Dobrinen/Todorcevic11}, but more needs to be checked for the general setting of topological Ramsey spaces constructed from a generating sequence.
The key  to this proof is that blocks are composed of sequences of members of \Fraisse\ classes, 
and the definition of generating sequence allows us to find blocks where all possible order configurations of some fixed finite collection of structures occur.

\begin{prop}\label{mix}
Assume that  $s,t\in (\hat{\mathcal F}\setminus\mathcal F)|B$ are mixed by $B$.
Let $k=|s|$ and $l=|t|$.
Then the following hold.
\begin{itemize}
\item[{(a)}] 
$\pi_s=\pi_{<>}$ if and only if  $\pi_t=\pi_{<>}$.
\item[{(b)}] 
$\pi_s=\pi_{\depth}$ if and only if  $\pi_t=\pi_{\depth}$.
\item[{(c)}] 
$\pi_s=\pi_{(I_{s,j})_{j\in J_k}}$
if and only if
$\pi_t=\pi_{(I_{t,j})_{j\in J_l}}$.
\end{itemize}
In the case of (c),
the set
$\{j\in J_k:I_{s,j}\ne\emptyset\}$ must equal $\{j\in J_l: I_{t,j}\ne\emptyset\}$, and 
the projected  substructures are isomorphic.
That is, if $\lgl i,(\bsS_{k,j})_{j\in J_k}\rgl=\pi_s(Z(k))$ and $\lgl i',(\bsT_{|t|,j})_{j\in J_l}\rgl=\pi_t(Z'(l))$,
then
for each $j\in J_k\cap J_l$,
the structures $\bsS_{k,j}$ and $\bsT_{l,j}$ are isomorphic;
in addition,
 for each $j\in J_k\setminus J_l$, $\bsS_{k,j}=\emptyset$,
and for each $j\in J_l\setminus J_k$, $\bsT_{l,j}=\emptyset$.

Furthermore, there is a $C\leq B$ such that for all $s,t\in\hat{\mathcal F}|C$, 
if $C$ mixes $s$ and $t$, then
for 
every $a\in\hat{\mathcal F}\cap r_{k+1}[s,C]/t$ and every $b\in \hat{\mathcal F}\cap r_{l+1}[t,C]/s$, $C$ mixes $a$ and $b$ if and only if $\pi_s(a(k))=\pi_t(b(l))$.
\end{prop}

\begin{proof}
Suppose $s,t\in  (\hat{\mathcal F}\setminus\mathcal F)|B$ are mixed by $B$, and let  $k=|s|$ and $l=|t|$.

(a) 
Suppose $\pi_s=\pi_{<>}$ and $\pi_t\neq \pi_{<>}$.
 By (1) of Claim \ref{claim 3}, 
$B$ mixes $s$ with at most one $E_t$ equivalence class of extensions of $t$.
Since $\pi_t\neq \pi_{<>}$,
there is a $Y\in[\max(k,l),B]$ such that for every $b\in r_{l+1}[t,Y]/s$, $Y$ separates $s$ and $b$. 
But then $Y$ separates $s$ and $t$, contradiction.
Thus, $\pi_t$ must also be $\pi_{<>}$.
By a similar argument, we conclude  that $\pi_s=\pi_{<>}$
if and only if
$\pi_t=\pi_{<>}$.

(b)
will follow from the argument for (c), in the case when all $I_{s,j}$ and $I_{t,j}$ are empty.

 (c) 
Suppose now that both $\pi_s$ and $\pi_t$ are not $\pi_{<>}$. 
 Let $m = \max(k,l)$, and let $n>m$ be large enough that all amalgamations  of two  copies of $(\bsA_{m,j})_{j\in J_m}$ can be embedded into $(\bsA_{n,j})_{j\in J_m}$. 
Let 
$$
\mathcal Z_0 = \{X\in[m,B]: B\ \mbox{separates } s^{\frown} X(k)\mbox{ and } t^{\frown}\pi_{\bsA_l}(X(n))\}
$$
and
$$
\mathcal Z_1 = \{X\in[m,B] : B\ \mbox{separates } s^{\frown}\pi_{\bsA_k}(X(n))\mbox{ and } t^{\frown} X(l)\}.
$$
Applying the Abstract Nash-Williams Theorem twice, 
we obtain an $X\in [m,B]$ such that 
$[m,X]$ is homogeneous for both $\mathcal{Z}_0$ and $\mathcal{Z}_1$.
Since we are assuming that both 
$\pi_s$ and $\pi_t$ are different from $\pi_{<>}$,
it must be the case that $[m,X]\sse \mathcal{Z}_0\cap\mathcal{Z}_1$.
Thus, 
for all $a\in r_{k+1}[s,X]/t$ and  $b\in  r_{l+1}[t,X]/s$, if $a$ and $b$ are mixed by $B$,  then $\depth_{B}(a) =  \depth_{B}(b)$.

Let $\mathcal{I}_k$ denote the collection of all sequences of the form $(I_{j})_{j\in J_k}$, where each $I_{j}\sse K_{n,j}$ and $\pi_{(I_{j})_{j\in J_k}}(B(n))\in\mathcal{R}(k)$.
Likewise,
let $\mathcal{I}_l$ denote the collection of all sequences of the form $(I_{j})_{j\in J_l}$, where each $I_{j}\sse K_{n,j}$ and $\pi_{(I_{j})_{j\in J_l}}(B(n))\in\mathcal{R}(l)$.

For each pair $\mathbb S \in\mathcal{I}_k$ and $\mathbb T\in\mathcal{I}_l$, let
\begin{equation}
\mathcal X_{\mathbb S, \mathbb T} = \{Y\leq X : B\ \mbox{ mixes }\ s^{\frown}\pi_{\mathbb S}(Y(n))\mbox{ and } t^{\frown}\pi_{\mathbb T}(Y(n))\}.
\end{equation}
Applying the Abstract Nash-Williams Theorem finitely-many times, we find a $Y\leq X$ which is homogeneous for
$\mathcal X_{\mathbb S, \mathbb T}$, for
all
 pairs $(\mathbb S,\mathbb T)\in\mathcal{I}_k\times\mathcal{I}_l$.

\begin{subclaimn}\label{subclaim.important}
For each pair $(\mathbb S,\mathbb T)\in\mathcal{I}_k\times\mathcal{I}_l$, 
if $\pi_s\circ\pi_{\mathbb S}(Y(n))\neq\pi_t\circ\pi_{\mathbb T}(Y(n))$, 
then $[\emptyset,Y]\cap\mathcal X_{\mathbb S,\mathbb T}=\emptyset$.
\end{subclaimn}

\begin{proof}
Suppose $\pi_s\circ\pi_{\mathbb S}(Y(n))\neq\pi_t\circ\pi_{\mathbb T}(Y(n))$.
Let 
$S(k)$ denote
$\pi_{\mathbb S}(Y(n))$ which is in $\mathcal{R}(k)$,
and let  $T(l)$ denote 
$\pi_{\mathbb T}(Y(n))$ which is in $\mathcal{R}(l)$.
Then there is a  $d$ and there are some substructures $\bsS_{k,j}\in{\bsY_{n,j}\choose\bsA_{k,j}}$, $j\in J_k$,
and $\bsT_{l,j}\in{\bsY_{n,j}\choose\bsA_{l,j}}$, $j\in J_l$ such that
$S(k)=\lgl d, (\bsS_{k,j})_{j\in J_k}\rgl$ 
and
$T(l)=\lgl d, (\bsT_{l,j})_{j\in J_l}\rgl$.
$\pi_s\circ\pi_{\mathbb S}(Y(n))=\pi_s(S(k))=\lgl d, (\bsS'_{k,j})_{j\in J_k}\rgl$ for some substructures $\bsS'_{k,j}\le \bsS_{k,j}$;
likewise,
$\pi_t\circ\pi_{\mathbb T}(Y(n))=\pi_t(T(l))=\lgl d, (\bsT'_{l,j})_{j\in J_l}\rgl$ for some substructures $\bsT'_{l,j}\le \bsT_{l,j}$.
Since  $\pi_s\circ\pi_{\mathbb S}(Y(n))\neq\pi_t\circ\pi_{\mathbb T}(Y(n))$,
 one of the following must happen:
 (i) there is some $j\in J_k\cap J_l$ such that $\bsS'_{k,j}\ne \bsT'_{k,j}$;
or (ii) there is a $j\in J_k\setminus J_l$ such that $\bsS'_{k,j}\ne\emptyset$;
or (iii)  there is a $j\in J_l\setminus J_k$ such that $\bsT'_{k,j}\ne\emptyset$.

In case (i), without loss of generality, assume that 
$|\bsS'_{k,j}|\setminus|\bsT'_{l,j}|\ne\emptyset$ for some $j\in J_k\cap J_l$;
that is,
the universe of $\bsS'_{k,j}$  is not contained within the universe of $\bsT'_{l,j}$.
Since $\bsS'_{k,j}$ and $\bsT'_{k,j}$ are substructures of $\bsY_{n,j}$,
their universes are subsets of the universe of $\bsY_{n,j}$.
Recall that $K_{n,j}$ is the cardinality of the universe of $\bsY_{n,j}$, and that we enumerate  the members of the universe $|\bsY_{n,j}|$  in increasing order as $\{y^i_{n,j}:i\in K_{n,j}\}$.
Let $p\in K_{n,j}$  be such that $y_{n,j}^p\in|\bsS'_{k,j}|\setminus|\bsT'_{k,j}|$.
Take $q$ large enough that there are $W(n),V(n)\in\mathcal{R}(n)|Y(q)$ such that
for all $i\in J_n\setminus \{j\}$,
$\bsW_{n,i}=\bsV_{n,i}$,
and
the universes of $\bsW_{n,j}$ and $\bsV_{n,j}$ differ only  on the members $w^p_{n,j}$ and $v^p_{n,j}$.
This is possible by the definition of a generating sequence; in particular, because $\mathcal{K}_j$ is a \Fraisse\ class.

Let $U(k)=\pi_{\mathbb S}(W(n))$, $U'(k)=\pi_{\mathbb S}(V(n))$,
$Z(l)=\pi_{\mathbb T}(W(n))$, and $Z'(l)=\pi_{\mathbb T}(V(n))$.
Then $\pi_t(Z(l))=\pi_t(Z'(l))$, which implies that $B$ mixes $t^{\frown} Z(l)$ and $t^{\frown} Z'(l)$.
If $[\emptyset, Y]\sse\mathcal{X}_{\mathbb S,\mathbb T}$,
then
it follows that $B$ mixes $s^{\frown} U(k)$ and $t^{\frown} Z(l)$, 
and $B$ mixes $s^{\frown} U'(k)$ and $t^{\frown} Z'(l)$.
By transitivity of mixing,
$B$ mixes $s^{\frown} U(k)$ and $s^{\frown} U'(k)$.
But $\pi_s(U(k))\ne \pi_s(U'(k))$, since 
$w^p_{n,j}\in
\pi_s(U(k))\setminus \pi_s(U'(k))$ 
(and also $v^p_{n,j}\in \pi_s(U'(k))\setminus\pi_s(U(k))$).
Thus, $U(k)$ is not ${E}_s$ related to $U'(k)$,
so $B$ does not mix   $s^{\frown} U(k)$ and $s^{\frown} U'(k)$, by 
 Claim \ref{claim decide}, a contradiction.
Therefore, 
it must be the case that $[\emptyset,Y]\cap\mathcal{X}_{\mathbb S,\mathbb T}=\emptyset$.

In case (ii), 
if  there is a $j\in J_k\setminus J_l$ with $\bsS'_{k,j}\ne\emptyset$,
then this implies that $J_k>J_l$.
Take $W(n),V(n)\in\mathcal{R}(n)|Y(q)$, for some $q$ large enough, such
that $\bsW_{n,j}$ and $\bsV_{n,j}$ have disjoint universes,
and for all $i\in J_n\setminus\{j\}$,
$\bsW_{n,i}=\bsV_{n,i}$.
Let $U(k)=\pi_{\mathbb S}(W(n))$, $U'(k)=\pi_{\mathbb S}(V(n))$,
$Z(l)=\pi_{\mathbb T}(W(n))$, and $Z'(l)=\pi_{\mathbb T}(V(n))$.
Then $Z(l)=Z'(l)$;
so in particular, $B$ mixes $t^{\frown}\pi_t\circ \pi_{\mathbb T}(W(n))$
and $t^{\frown}\pi_t\circ\pi_{\mathbb T}(V(n))$.
Again, if $[\emptyset, Y]\sse\mathcal{X}_{\mathbb S,\mathbb T}$,
then
 $B$ mixes $s^{\frown} U(k)$ and $t^{\frown} Z(l)$, 
and $B$ mixes $s^{\frown} U'(k)$ and $t^{\frown} Z'(l)$.
By transitivity of mixing,
$B$ mixes $s^{\frown} U(k)$ and $s^{\frown} U'(k)$.
But $\pi_s(U(k))\ne \pi_s(U'(k))$, since 
the universes of $\bsW_{n,j}$ and $\bsV_{n,j}$ are disjoint, and the $j$-th structures in $\pi_s(U(k))$ and $\pi_s(U'(k))$
are not empty and not equal.
Thus, $B$ does not mix   $s^{\frown} U(k)$ and $s^{\frown} U'(k)$, by 
 Claim \ref{claim decide}, a contradiction.
Therefore,  $[\emptyset,Y]\cap\mathcal{X}_{\mathbb S,\mathbb T}=\emptyset$.

By a similar argument as in Case (ii),
we conclude that $[\emptyset,Y]\cap\mathcal{X}_{\mathbb S,\mathbb T}=\emptyset$ in Case (iii) as well.
Thus, in all cases, 
 $[\emptyset,Y]\cap\mathcal{X}_{\mathbb S,\mathbb T}=\emptyset$ .
\end{proof}

It follows from the Subclaim that whenever 
$a\in r_{k+1}[s,Y]/t$, $b\in r_{l+1}[t,Y]/s$ and
 $B$ mixes $a$ and $b$, then $\pi_s(a(k))=\pi_t(b(l))$.
Thus, (b) and (c) follow.


Now we prove there is a $C\leq Y$ such that for all $s,t\in\hat{\mathcal F}|C$, 
if $C$ mixes $s$ and $t$, then
for 
every $a\in\hat{\mathcal F}\cap r_{k+1}[s,C]/t$ and every $b\in \hat{\mathcal F}\cap r_{l+1}[t,C]/s$, $C$ mixes $a$ and $b$ if and only if $\pi_s(a(k))=\pi_t(b(l))$.
Since the forward direction holds below $Y$, it only remains to 
find a $C\le Y$ such that, below $C$, whenever $\pi_s(a(k))=\pi_t(b(l))$, then $C$ mixes $a$ and $b$.

Let  $(\mathbb S,\mathbb T)\in\mathcal{I}_k\times\mathcal{I}_l$ be a pair such that  
$\pi_s\circ\pi_{\mathbb S}(Y(n))=\pi_t\circ\pi_{\mathbb T}(Y(n))$.
It suffices to show that $[\emptyset, Y]\sse \mathcal{X}_{\mathbb S,\mathbb T}$.
Assume also, towards a contradiction, that  $[\emptyset,Y]\cap\mathcal X_{\mathbb S,\mathbb T}=\emptyset$. 
Let $\mathbb S',\mathbb T'$ be a pair in
 $\mathcal{I}_k\times\mathcal{I}_l$ 
satisfying $\pi_s\circ\pi_{\mathbb S'} (Y(n))=\pi_t\circ\pi_{\mathbb T'}(Y(n))$. 
We will prove that $[\emptyset,Y]\cap\mathcal X_{\mathbb S',\mathbb T'}=\emptyset$. 
Take $V(n),W(n)\in\mathcal{R}(n)|Y$ such that 
$\pi_s\circ \pi_{\mathbb S}(V(n)) =\pi_s\circ\pi_{\mathbb S'}(W(n))$ and 
$\pi_t\circ\pi_{\mathbb T}(V(n))=\pi_t\circ\pi_{\mathbb T'}(W(n))$.
This is possible since all $\mathcal{K}_j$ are \Fraisse\ classes with the OPFAP and  we are assuming that 
 $\pi_s\circ\pi_{\mathbb S} (Y(n))=\pi_t\circ\pi_{\mathbb T}(Y(n))$
 and
 $\pi_s\circ\pi_{\mathbb S'} (Y(n))=\pi_t\circ\pi_{\mathbb T'}(Y(n))$. 
Then $Y$ mixes $s^{\frown}\pi_{\mathbb S}(V(n))$ and $s^{\frown}\pi_{\mathbb S'}(W(n))$;
and $Y$ mixes $t^{\frown}\pi_{\mathbb T}(V(n))$ and $t^{\frown}\pi_{\mathbb T'}(W(n))$.
Since 
$Y$ separates  $s^{\frown}\pi_{\mathbb S}(V(n))$ and $t^{\frown}\pi_{\mathbb T}(V(n))$,
and since mixing is transitive,
it
 follows that $Y$ must separate $s^{\frown}\pi_{\mathbb S'}(W(n))$ and $t^{\frown}\pi_{\mathbb T'}(W(n))$.
Thus, $[\emptyset, Y]\cap\mathcal{X}_{\mathbb S',\mathbb T'}=\emptyset$.

This, along with the Subclaim, implies that for all pairs $(\mathbb S,\mathbb B)$ in $\mathcal{I}_k\times\mathcal{I}_l$,
$[\emptyset, Y]\cap\mathcal{X}_{\mathbb S,\mathbb T}=\emptyset$.
But this implies that $Y$ separates $s$ and $t$, a contradiction.
Therefore, 
for all pairs
$(\mathbb S,\mathbb T)\in\mathcal{I}_k\times\mathcal{I}_l$ such that  
$\pi_s\circ\pi_{\mathbb S}(Y(n))=\pi_t\circ\pi_{\mathbb T}(Y(n))$,
we have $[\emptyset,Y]\sse \mathcal{X}_{\mathbb S,\mathbb T}$.
Thus, whenever $U(k)\in \mathcal{R}(k)|Y/(s,t)$ and $V(l)\in\mathcal{R}(l)|Y/(s,t)$ satisfy $\pi_s(U(k))=\pi_t(V(l))$,
then $Y$ mixes $s^{\frown}U(k)$ and $t^{\frown}V(l)$.
By Lemma \ref{lema decides 2}, we get $C\leq Y$ for which the Proposition  holds.
\end{proof}

By very slight, straightforward  modifications to the proofs of Claims 4.19 - 4.21 in \cite{Dobrinen/Todorcevic11}, we obtain the following claim.

\begin{clm}\label{finish}
For all $a,b\in\mathcal{F}|C$,
$a\, R\, b$ if and only if $\vp(a)=\vp(b)$.
Moreover, $\vp(a)\not\sqsubset\vp(b)$.
\end{clm}

By its definition, $\vp$ is inner, and by Claim \ref{finish}, $\vp$ is Nash-Williams and canonizes the equivalence relation $R$.
\end{proof}

\begin{rem}
We point out that the entire proof of Theorem \ref{canonical R} used only instances of the Abstract Nash-Williams Theorem, and not the full power of the Abstract Ellentuck Theorem.
\end{rem}

The following corollary of Theorem \ref{canonical R}
is proved in exactly the same way as Theorem 4.3 in \cite{Dobrinen/Todorcevic11}.

\begin{cor}\label{canonical on finite fronts}
Let $1\le J\le \om$, and $\mathcal{K}_j$, $j\in J$, be \Fraisse\ classes of finite ordered relational structures with the Ramsey property and the Order-Prescribed  Free Amalgamation Property.
Let
$\lgl \bsA_k:k<\om\rgl$ be a fixed generating sequence,
and let $\mathcal{R}$ denote the topological Ramsey space $\mathcal{R}(\lgl \bsA_k:k<\om\rgl)$.

For any $n$, $A\in \mathcal{R}$, and equivalence relation $R$ on $\mathcal{AR}_n|A$,
there is a $B\le A$ such that 
$R$ is canonical on $\mathcal{AR}_n|B$.
This means 
there are equivalence relations  $E_i\in\mathcal E(i)$, $i< n$, such that for all $a,b\in\mathcal{AR}_n|B$ 
$$a\, R\, b\ \mbox {if and only if }\, \forall i<n, a(i)\, E_i\, b(i).$$
\end{cor}


\section{Basic Tukey reductions for selective and Ramsey filters\\
 on general topological Ramsey spaces}\label{sec.basic}

We first remind the reader of the basic definitions of the Tukey theory of ultrafilters. Suppose that $\mathcal{U}$ and $\mathcal{V}$ are ultrafilters. A function $f$ from $\mathcal{U}$ to $\mathcal{V}$ is \emph{cofinal} if every cofinal subset of $(\mathcal{U},\supseteq)$ is mapped by $f$ to a cofinal subset of $(\mathcal{V}, \supseteq)$. We say that $\mathcal{V}$ is \emph{Tukey reducible to $\mathcal{U}$} and write $\mathcal{V} \le_{T} \mathcal{U}$ if there exists a cofinal map $f:\mathcal{U} \rightarrow \mathcal{V}$. If $\mathcal{U} \le_{T}\mathcal{V}$ and $\mathcal{V} \le_{T}\mathcal{U}$ then we write $\mathcal{V}\equiv_{T} \mathcal{U}$ and say that $\mathcal{U}$ and $\mathcal{V}$ are \emph{Tukey equivalent.} The relation $\equiv_{T}$ is an equivalence relation and $\le_{T}$ is a partial order on its equivalence classes.  The equivalence classes are called \emph{Tukey types}.
(See the recent survey paper \cite{Dobrinen13} for more background on Tukey types of ultrafilters.)

When restricted to ultrafilters, the Tukey reducibility relation is a coarsening of the Rudin-Keisler reducibility relation. 
If $h(\mathcal{U})=\mathcal{V}$, then the map sending $X\in\mathcal{U}$ to $h''X \in\mathcal{V}$ witnesses Tukey reducibility. 
Thus,  if $\mathcal{V} \le _{RK} \mathcal{U}$, then $\mathcal{V}\le_{T} \mathcal{U}$.

The work in this section will set up some of the machinery for answering this and Questions 1, 2, and 3  from the Introduction;
we do that in the next section.
The main results in this section, Proposition \ref{BasicReductionTukeyProp}
and Theorem \ref{p-point Tukey theorem}, are proved for general topological Ramsey spaces, in the hope that they may be more generally applied in the future.

An ultrafilter $\mathcal{U}$ on a countable base $X$ has \emph{continuous Tukey reductions} if whenever a non-principal ultrafilter $\mathcal{V}$ is Tukey reducible to $\mathcal{U}$, then every monotone cofinal map $f:\mathcal{U} \rightarrow \mathcal{V}$ is continuous with respect to the subspace topologies on $\mathcal{U}$ and $\mathcal{V}$ inherited from $2^{X}$ when restricted to some cofinal subset of $\mathcal{U}$. The next theorem has become an important tool in the study of the Tukey structure of ultrafilters Tukey reducible to some p-point ultrafilter.

\begin{thm}[Dobrinen and Todorcevic, \cite{Dobrinen/Todorcevic10}] \label{Ppoint Reduction} If $\mathcal{U}$ is a p-point ultrafilter on $\omega$, then $\mathcal{U}$ has continuous Tukey reductions. 
\end{thm}
In fact, by  
results of Dobrinen
(see Theorem 2.7  in  \cite{DobrinenCanonicalMaps15}, which first appeared in   the unpublished  \cite{Dobrinen10}), every ultrafilter  Tukey reducible to some p-point has continuous Tukey reductions.

 In the previous sections of this paper we restricted consideration to  spaces constructed from generating sequences. 
In this section we consider all  topological Ramsey spaces $\mathcal{R}$ such that $\mathcal{R}$ is  closed in $\mathcal{AR}^{\bN}$, $(\mathcal{R}, \le)$ is a partial order, and $(\mathcal{R},\le, r)$ satisfies axioms $\bf{A.1}-\bf{A.4}$.
In Theorem \ref{p-point Tukey theorem},
we  generalize  Theorem $\ref{Ppoint Reduction}$ to filters selective for a topological Ramsey space.

\begin{notation}
 In order to avoid repeating certain phrases, we let $(\mathcal{R}, \le, r)$ denote a fixed triple satisfying axioms ${\bf A.1}-{\bf A.4}$ which is closed in the subspace topology it inherits from $\mathcal{AR}^{\mathbb{N}}$. Furthermore, we assume that $(\mathcal{R},\le)$ is a partial order and has a top element which we denote by $\mathbb{A}$. By the Abstract Ellentuck Theorem, $\mathcal{R}$ forms a topological Ramsey space. If $\mathcal{C}$ is a subset of $\mathcal{R}$ we let $\mathcal{C} \ge_{T} \mathcal{V}$ denote the statement $(\mathcal{C},\ge)\ge_{T} (\mathcal{V}, \supseteq)$.
\end{notation}
We omit the proof of the next fact since if follows by a straightforward generalization of the proof of Fact 6 from \cite{Dobrinen/Todorcevic10}.

\begin{fact}\label{Tim-monotone reduction} Assume that $\mathcal{C}\subseteq \mathcal{R}$ and $\mathcal{V}$ is an ultrafilter on $\omega$. If $\mathcal{C}\ge_{T} \mathcal{V}$, then there is monotone cofinal map $f: \mathcal{C} \rightarrow \mathcal{V}$.
\end{fact}
The notion of a selective filter for a topological Ramsey space was introduced along with the relation of almost-reduction by Mijares in \cite{mijares2007notion}.
The notion of almost reduction on a topological Ramsey space was introduced by  Mijares in \cite{mijares2007notion}. The relation of almost reduction generalizes the relation of almost inclusion $\subseteq^{*}$ on $\mathcal{P}(\omega)$ to arbitrary topological Ramsey spaces.  The relation of \emph{almost reduction on $\mathcal{R}$} is defined as follows: $X \le^{*} Y$ if and only if there exists $a\in \mathcal{AR}$ such that $\emptyset \not = [a, X] \subseteq [a, Y]$.
Fix the following notation:
For any fixed $\mathbb A\in \mathcal{R}$,
 for  $n<\omega$ and $X,Y\le\mathbb A$,
define $X/ n \le Y$ if and only if there exists an $a\in \mathcal{AR}|Y$ with $\depth_{\mathbb{A}}(a) \le  n$, $\emptyset \not= [a, X] \subseteq [a,Y]$. 
In particular, if $X/n\le Y$, then $X\le^* Y$.
If a topological Ramsey space has a maximum member, we let $\mathbb A$ denote that member.
Otherwise, we may without loss of generality fix some $\mathbb A\in \mathcal{R}$ and work below $\mathbb A$.

\begin{fact}
 For each $X$ and $Y$ in $\mathcal{R}$, $X\le^{*}Y$ if and only if there exists $i<\omega$ such that $X/r_i(X) \le Y$.
\end{fact}

\begin{defn}
A subset $\mathcal{C}\subseteq \mathcal{R}$ is a $\emph{selective filter}$ on $(\mathcal{R},\le)$ if $\mathcal{C}$ is a maximal  filter on $(\mathcal{R},\le)$ and for each decreasing sequence $X_{0}\ge X_{1} \ge X_{2} \ge \dots $ of elements of $\mathcal{C}$ there exists $X\in \mathcal{C}$ such that for all $i<\omega$, $X/r_{i}(X) \le X_{i}$.
\end{defn}

Axiom ${\bf A.3}$ implies that for each decreasing sequence $X_{0}\ge X_{1} \ge X_{2} \ge \dots $ of elements of $\mathcal{R}$ there exists $X\in \mathcal{R}$ such that for all $i<\omega$, $X/r_{i}(X) \le X_{i}$. Thus, assuming MA or CH it is possible to construct a selective filter on $(\mathcal{R},\le)$. Forcing with $\mathcal{R}$ using almost reduction adjoins a filter on $(\mathcal{R},\le)$ satisfying a localized version of the Abstract Nash-Williams theorem for $\mathcal{R}$. 
By  work of Mijares in \cite{mijares2007notion} every ultrafilter generic for this forcing is a selective filter on $(\mathcal{R},\le)$.

Recall that $\mathcal{R}$ is assumed to be closed in the subspace topology it inherits from $\mathcal{AR}^{\mathbb{N}}$. A sequence $(X_{n})_{n<\omega}$ of elements of $\mathcal{R}$ converges to an element $X\in \mathcal{R}$ if and only if for each $k<\omega$ there is an $m<\omega$ such that for each $n\ge m$, $r_{k}(X_{n}) = r_{k}(X).$ A function $f: \mathcal{R}\rightarrow \mathcal{P}(\omega)$ is \emph{continuous} if and only if for each convergent sequence $(X_{n})_{n<\omega}$  in $\mathcal{R}$ with $X_{n} \rightarrow X$, we also have $f(X_{n})\rightarrow f(X)$ in the topology obtained by identifying $\mathcal{P}(\omega)$ with $2^{\bN}$. A function $f:\mathcal{C}\rightarrow \mathcal{V}$ is said to be \emph{continuous} if it is continuous with respect to the topologies on $\mathcal{C}$ and $\mathcal{V}$ taken as subspaces of $\mathcal{AR}^{\mathbb{N}}$ and $2^{\bN}$, respectively. The next definition is a generalization of notion of basic Tukey reductions for an ultrafilter on $\omega$, (see Definition 2.2  and Lemma 2.5 in \cite{DobrinenCanonicalMaps15}), to filters on $\mathcal{R}$.

\begin{defn}\label{defn.basic}
 Assume that $\mathcal{C}\subseteq \mathcal{R}$ is a filter on $(\mathcal{R},\le)$. \emph{$\mathcal{C}$ has basic Tukey reductions} if whenever $\mathcal{V}$ is a non-principal ultrafilter on $\omega$ and  $f:\mathcal{C} \rightarrow \mathcal{V}$ is a monotone cofinal map, there is an $X\in \mathcal{C}$ and  a  monotone map $\tilde{f}:\mathcal{R} \rightarrow \mathcal{P}(\om)$    such that
\begin{enumerate}
\item  $\tilde{f}$ is continuous with respect to the metric topology on $\mathcal{AR}^{\bN}$;
\item $\tilde{f}\upharpoonright ( \mathcal{C} \upharpoonright X)=f \upharpoonright ( \mathcal{C} \upharpoonright X)$;
\item
$\tilde{f}$ is generated by a finitary map
 $\hat{f}: \mathcal{AR} \rightarrow [\omega]^{<\omega}$ satisfying
\begin{enumerate}
\item For each $k<\omega$ and each $s\in \mathcal{AR}$, if $\depth_{\mathbb{A}}(s) \le k$ then $\hat{f}(s) \subseteq k;$
\item 
$s\sqsubseteq t \in \mathcal{AR}$ implies that $\hat{f}(s) \sqsubseteq \hat{f}(t);$
\item 
$\hat{f}$ is monotone, that is, if $s, t \in \mathcal{AR}$ with $s\le_{\fin} t$, then $\hat{f}(s) \subseteq \hat{f}(t)$; and 
\item For each $Y\in \mathcal{R}$, $\tilde{f}( Y) = \bigcup_{k<\omega} \hat{f}(r_{k}(Y))$.
\end{enumerate}
\end{enumerate}
\end{defn}

The next proposition provides an important application of the notion of basic Tukey reductions for $\mathcal{C}$ and helps reduce the characterization of the ultrafilters on $\omega$ Tukey reducible to $(\mathcal{C},\ge)$ to the study of canonical equivalence relations for fronts on $\mathcal{C}$. It is the generalization of Proposition 5.5 from \cite{Dobrinen/Todorcevic11} to our current setting. 

\begin{defn}
If $\mathcal{C}\subseteq \mathcal{R}$ and $\mathcal{F}\subseteq \mathcal{AR}$ then we will say that \emph{$\mathcal{F}$ is a front on $\mathcal{C}$} if and only if for each $C\in \mathcal{C}$, there exists ${s} \in\mathcal{F}$ such that ${s} \sqsubseteq X$; and for all pairs $s\ne t$ in $\mathcal{F}$, $s\not\sqsubset t$.
\end{defn}

\begin{prop} \label{Proposition 5.5 Analogue} 
Assume that $\mathcal{C}\subseteq \mathcal{R}$ is a  filter on $(\mathcal{R},\le)$  which  has basic Tukey reductions, and suppose $\mathcal{V}$ is a non-principal  ultrafilter on $\omega$ with $ \mathcal{V}\le_T \mathcal{C}$. 
Then there is a front $\mathcal{F}$ on $\mathcal{C}$ and a function $f:\mathcal{F}\rightarrow \omega$ such that for each $Y\in \mathcal{V}$,
 there exists $X\in \mathcal{C}$ such that $f( \mathcal{F}|X) \subseteq Y$.
 Furthermore, if $\mathcal{C}\upharpoonright \mathcal{F}$ is a base for an ultrafilter on $\mathcal{F}$,
 then $\mathcal{V} = f(\left< \mathcal{C}\upharpoonright\mathcal{F} \right>).$
\end{prop}

\begin{proof} 
Suppose that $\mathcal{C}$ and $\mathcal{V}$ are given and satisfy the assumptions of the proposition. 
By Fact $\ref{Tim-monotone reduction}$, there is a monotone map $g: \mathcal{C} \rightarrow \mathcal{V}$. Since $\mathcal{C}$ has basic Tukey reductions, there is a continuous monotone cofinal map $g':\mathcal{C} \rightarrow \mathcal{V}$  and a function $\hat{g}: \mathcal{AR} \rightarrow [\omega]^{<\omega}$ satisfying (1)-(3) in the definition of basic Tukey reductions. Let $\mathcal{F}$ consist of all $r_{n}(Y)$ such that $Y \in \mathcal{C}$ and $n$ is minimal such that $\hat{g}(r_{n}(Y)) \not = \emptyset$. 
By the properties of $\hat{g}$, $\min(\hat{g}(r_{n}(Y))) = \min( g( Y))$. By its definition $\mathcal{F}$ is a front on $\mathcal{C}$. Define a new function $f:\mathcal{F} \rightarrow \omega$ by $f(b) = \min(\hat{g}(b))$, for each $b\in \mathcal{F}$.

Since $g'$ is a monotone cofinal map, the $g'$-image of $\mathcal{C}$ in $\mathcal{V}$ is a base for $\mathcal{V}$. From the construction of $f$, we see that for each $X\in \mathcal{C}$, $f(\mathcal{F}|X) = \{ f(a) : a\in \mathcal{F}|X\} \subseteq g'(X)$.
 Therefore, for each $Y\in\mathcal{V}$ there exists $X\in \mathcal{C}$ such that $f(\mathcal{F}|X) \subseteq Y$.
 We remind the reader of
the following useful fact
(see Fact 5.4 from \cite{Dobrinen/Todorcevic11}).

\begin{fact}\label{EqualityFact} Suppose $\mathcal{V}$ and $\mathcal{U}$ are proper ultrafilters on the same countable base set, and for each $V\in \mathcal{V}$ there is a $U \in \mathcal{U}$ such that $U \subseteq V$. Then $\mathcal{U} = \mathcal{V}$.
\end{fact}

Suppose that $\mathcal{C} \upharpoonright \mathcal{F}$ generates an ultrafilter on $\mathcal{F}$,
 and let $\left<\mathcal{C}\upharpoonright \mathcal{F} \right>$ denote the ultrafilter it generates. Then the Rudin-Keisler image $f(\left< \mathcal{F}\upharpoonright \mathcal{C}\right>)$ is an ultrafilter on $\omega$ generated by the base $\{ f(\mathcal{F}|X) : X\in \mathcal{X}\}$. Hence, Fact $\ref{EqualityFact}$ implies that $f(\left< \mathcal{F}\upharpoonright \mathcal{C}\right>)=\mathcal{V}$. 
\end{proof}

If a selective filter $\mathcal{C}$ on $(\mathcal{R},\le)$ has the property that, for each front $\mathcal{F}$ on $\mathcal{C}$, $\mathcal{C}\upharpoonright \mathcal{F}$ generates an ultrafilter on $\mathcal{F}$, 
then Proposition
\ref{Proposition 5.5 Analogue} 
 shows that every nonprincipal ultrafilter Tukey-reducible to $\mathcal{C}$ is a Rudin-Keisler image of $\mathcal{C}\upharpoonright \mathcal{F}$, for some front on $\mathcal{C}$. This provides motivation for studying the notion of a Nash-Williams filter on $(\mathcal{R},\le)$. The next definition is an adaptation of Definition $5.1$ ($1$) from \cite{Dobrinen/Todorcevic11} to our current setting.

\begin{defn}\label{defn.NWfilter} 
A maximal filter $\mathcal{C} \subseteq \mathcal{R}$
is a \emph{Nash-Williams filter on $(\mathcal{R},\le)$}
if  for each 
front $\mathcal{F}$ on $\mathcal{C}$ and each 
$\mathcal{H} \subseteq \mathcal{F}$, there is a $C\in \mathcal{C}$  such that
either  $\mathcal{F}|C\sse\mathcal{H}$
or else $\mathcal{F}|C\cap\mathcal{H}=\emptyset$.
\end{defn}

It is clear that any Nash-Williams filter is also a Ramsey filter on $(\mathcal{R},\le)$
(recall Definition \ref{def.Ramseymaxfilter}), and hence is a maximal filter. 
The Abstract Nash-Williams Theorem for $\mathcal{R}$ can be used in conjunction with MA or CH to construct a Nash-Williams filter on $(\mathcal{R} ,\le)$. Furthermore, forcing with $\mathcal{R}$ using almost reduction adjoins a Nash-Williams filter on $(\mathcal{R}, \le)$. 
By work of Mijares in \cite{mijares2007notion}, any Ramsey filter on $(\mathcal{R}, \le)$ is a selective filter on $(\mathcal{R}, \le)$. 
Thus, any Nash-Williams filter is a selective filter on $(\mathcal{R},\le)$. 
Trujillo in \cite{Trujillo2013selective} has shown that (assuming CH or MA, or by forcing) there are topological Ramsey spaces for which there are maximal filters which are selective but not Ramsey for those spaces.
We omit the proof of the next theorem as it follows from a straightforward generalization of the proof of Trujillo 
 in \cite{TrujilloThesis}
  for the special case of the space $\mathcal{R}_{1}$ (recall Example \ref{ex.BT4.84.9} in Section \ref{sec.uf}).

\begin{thm}\label{thm.Ramsey=NW}
 Let $1\le J\le\omega$ and $\mathcal{K}_{j}$, $j\in J$, be a collection of Fra\"{i}ss\'{e} classes of finite ordered relational structures such that each $\mathcal{K}_{j}$ satisfies the Ramsey property. 
Let $\left< \mathbf{A}_{k} : k<\omega \right>$ be a generating sequence,
and let $\mathcal{R}$ denote $\mathcal{R}(\left< \mathbf{A}_{k} : k<\omega \right>)$.
 Suppose that $\mathcal{C}$ is a filter on $(\mathcal{R},\le)$. 
Then 
$\mathcal{C}$ is  Nash-Williams for $\mathcal{R}$
 if and only if $\mathcal{C}$ is Ramsey
for $\mathcal{R}$.
\end{thm}

The next  fact is the analogue of Fact 5.3  from \cite{Dobrinen/Todorcevic11}. 
We omit the proof as it follows by similar arguments.

\begin{fact}\label{analogueFact5.3}  Suppose $\mathcal{C}\subseteq\mathcal{R}$ is a Nash-Williams filter on $(\mathcal{R},\le)$. If $\mathcal{C}'$ is any cofinal subset of $\mathcal{C}$, and $\mathcal{F} \subseteq \mathcal{AR}$ is any front on $\mathcal{C}'$, then $\mathcal{C}' \upharpoonright \mathcal{F}$ generates an ultrafilter on $\mathcal{F}$.
\end{fact}

The next proposition is one of the keys in the general mechanism for classifying initial Tukey structures and the Rudin-Keisler structures within them.
We only sketch the proof here, as it is the same proof as that of Proposition 5.5  in \cite{Dobrinen/Todorcevic11}.

\begin{prop}\label{BasicReductionTukeyProp}
Assume that that $\mathcal{C}\subseteq \mathcal{R}$ is a Nash-Williams filter on $(\mathcal{R},\le)$. 
Suppose $\mathcal{C}$ has basic Tukey reductions and $\mathcal{V}$ is a non-principal an ultrafilter on $\omega$ with $\mathcal{C} \ge_{T} \mathcal{V}$. 
Then there is a front $\mathcal{F}$ on $\mathcal{C}$ and a function $f:\mathcal{F}\rightarrow \omega$ such that $\mathcal{V} = f(\left< \mathcal{C}\upharpoonright\mathcal{F} \right>).$
\end{prop}

\begin{proof} Suppose that $\mathcal{V}$ is Tukey reducible to some Nash-Williams filter $\mathcal{C}$ on $(\mathcal{R},\le)$. Assume that $\mathcal{C}$ has Basic Tukey reductions. Theorem $\ref{Proposition 5.5 Analogue}$ and Fact $\ref{analogueFact5.3}$ imply that there is a front $\mathcal{F}$ on $\mathcal{C}$ and a function $f:\mathcal{F}\rightarrow \omega$ such that $\mathcal{V} = f(\left< \mathcal{C} \upharpoonright \mathcal{F}\right>)$.
\end{proof}

We now introduce some notation needed for its definition and for the proof of the main theorem of this section.

\begin{notation}\label{notn.thm53}
If there is a maximum member of $\mathcal{R}$, let  $\mathbb A$ denote it.
Otherwise, fix some $\mathbb A\in\mathcal{R}$ and relative everything that follows to $[0,\mathbb A]$.
 For each $X,Y \le \mathbb A$,  define
\begin{equation}
d(X) = \{ \depth_{\mathbb{A}}(r_{i}(X)): i <\omega\}.
\end{equation}
 Define $\rho:[0,\mathbb A]\times \omega \rightarrow \mathcal{AR}$ to be the map such that for each $X\le \mathbb A$ and each $n<\omega$, $\rho(X,n) = r_{i}(X)$, where $i$ is the unique natural number such that $\depth_{\mathbb{A}}(r_{i}(X)) \le  n < \depth_{\mathbb{A}}(r_{i+1}(X))$. 
\end{notation}

By ${\bf A.1}$ (c) and ${\bf A.2}$ (b), for each $X\le\mathbb A$, $d(X)$ is infinite.
Also note that
for each $s\in \mathcal{AR}|X$, 
 $X/s \le Y$ if and only if $\depth_{X}(s) = i$ and $X/r_{i}(X) \le Y$.
In particular,
 $X/r_{i}(X) \le Y$ if and only if $X/\depth_{\mathbb{A}}(r_{i}(X)) \le Y$.

The next theorem is the main result of this section.
It extends to all topological Ramsey spaces previous  results in \cite{Dobrinen/Todorcevic10} for the Milliken space $\FIN^{[\infty]}$ and in \cite{Dobrinen/Todorcevic11} and \cite{Dobrinen/Todorcevic12} for the $\mathcal{R}_{\al}$  spaces ($1\le\al<\om_1$).
It will be used in conjunction with Proposition \ref{BasicReductionTukeyProp} in the next section to identify initial structures in the Tukey types of ultrafilters.

\begin{thm} \label{p-point Tukey theorem}
If $\mathcal{C}$ is a selective filter on $(\mathcal{R},\le)$ and $\{d(X): X\in \mathcal{C}\}$ generates a nonprincipal ultrafilter on $\omega$ then, $\mathcal{C}$ has basic Tukey reductions.
\end{thm}

\begin{proof}
Suppose that $\mathcal{V}$ is an ultrafilter on $\omega$ Tukey reducible to $\mathcal{C}$, and $f:\mathcal{C}\rightarrow\mathcal{V}$ is a monotone cofinal map witnessing $\mathcal{C}\ge_{T} \mathcal{V}$. For each $k<\omega$, let $P_{k}(\cdot, \cdot)$ be the following proposition:
For $s\in\mathcal{AR}$ and $X\in \mathcal{R}$, $P_{k}(s,X)$ holds if and only if for each $Z\in \mathcal{C}$ such that $s \sqsubseteq Z$ and $Z/s \le X$,  $k\not \in f(Z)$. Let $\mathcal{C}$ be a selective filter for $(\mathcal{R},\le)$. Assume that $\{d(X): X\in \mathcal{C}\}$ generates an nonprincipal ultrafilter on $\omega$.

\begin{clm} 
There is an $\bar{X}\in \mathcal{C}$ such that $f\upharpoonright (\mathcal{C} \upharpoonright \bar{X}) : \mathcal{C} \upharpoonright \bar{X} \rightarrow \mathcal{V}$ is continuous.
\end{clm}

\begin{proof}
We begin by constructing a decreasing sequence in $(\mathcal{C},\le)$. 
Let $X_{0}= \mathbb{A}$. 
Given $n>0$ and $X_{i}\in \mathcal{C}$ 
 for all $i< n$, 
we will choose $X_{n}\in \mathcal{C}$ such that
\begin{enumerate}
\item 
$X_n\le X_{n-1}$,
\item 
$\rho(X_{n}, n) = \emptyset$,
\item 
For each $s$ in $\mathcal{AR}$ with $\depth_{\mathbb{A}}(s) \le n$ and each $k\le n$, if there exists $Y'\in\mathcal{C}$ such that $\rho(Y',n)=s$ and $k\not \in f(Y')$, then $P_{k}(s, X_{n})$ holds.
\end{enumerate}

By axiom ${\bf A.2}$ (a),  the set $\{ s \in \mathcal{AR} : \depth_{\mathbb{A}}(s) \le  n\}$ is finite.
 Let $s_{1}, s_{2}, \dots s_{i_{n}}$ be an enumeration of $\{ s \in \mathcal{AR} : \depth_{\mathbb{A}}(s) \le n\}$. Since $\mathcal{C}$ is a maximal filter and $\{d(X): X\in \mathcal{C}\}$ forms a nonprincipal ultrafilter on $\omega$, there exists a $W_{0} \in \mathcal{C}$ such that $W_{0} \le X_{n-1}$ and $\rho(W_{0}, n) = \emptyset$. Now suppose that there exists $Y\in \mathcal{C}$ such that $\rho(Y, n) = s_{1}$ and $k \not \in f(Y)$. Take $Y_{1}$ to be in $\mathcal{C}$ such that $\rho(Y_{1}, n) = s_{1}$ and $k \not \in f(Y_{1})$. 
Since $\mathcal{C}$ is a filter on $(\mathcal{C},\le)$ there exists $W_1\in \mathcal{C}$ such that $W_1\le Y_{1}, W_{0}$. 
If there is no $Y\in \mathcal{C}$ such that $\rho(X,n)=s_{1}$ and $k\not\in f(Y)$,
then let $W_1=W_0$.
For the induction step, suppose that for $1\le l < i_{n}$ and $W_{0} \ge W_{1} \ge \dots \ge W_{l}$ are given and in $\mathcal{C}$. If there is a $Y\in \mathcal{C}$ such that $s_{l+1} = \rho(Y,n)$ and $k\not \in f(Y)$, then take some $Y_{l}\in \mathcal{C}$ and let $W_{l+1}\in\mathcal{C}$ such that $W_{l+1}\le W_{l}, Y_{l+1}$.
Otherwise, let $W_{l+1} = W_{l}$. After $i_{n}$ many steps let $X_{n} = W_{i_{n}}$.

We check that $X_{n}$ satisfies properties (1) - (3). 
(1) By construction $X_{n} \le X_{n-1}$. 
(2) Since $\rho(W_{0}, n) = \emptyset$ and $X_{n}\le W_{0}$, we have $\rho(X_{n},n)=\emptyset$.
(3) Let $s$ be an element of $\mathcal{AR}$ such that $\depth_{\mathbb{A}}(s) \le n$. It follows that there exists $1\le l \le i_{n}$ such that $s = s_{l}$. If there is a $Y'\in \mathcal{C}$ such that $s=s_{l}=\rho(Y',n)$ and $k\not \in f(Y')$ then $W_{l}$ was taken so that $W_{l} \le W_{l-1}, Y_{l}'$. Hence, if $Z \in \mathcal{C}$, $s \sqsubseteq Z$ and $Z/s \le X_{n}$ then $Z \le Y'_{l}$.
Since $f$ is monotone and $k\not\in f(Y'_{l})$, it must be the case that $P_{k}(s, X_{n})$ holds.

Since $\mathcal{C}$ is selective for $(\mathcal{R},\le)$, there exists $Y \in \mathcal{C}$ such that for each $i<\omega$, $Y/r_{i}(Y)\le X_{i}$. Let $\{y_{0}, y_{1}, \dots\}$ denote the increasing enumeration of $d(Y)$. Let $A = \bigcup [y_{2i+1}, y_{2i+2})$. Without loss of generality, assume that $A$ is not in the ultrafilter generated by $\{d(X): X\in \mathcal{C}\}$. 
 Let $\bar{X}$ be an element of $\mathcal{C}$ such that $\bar{X} \le Y$ and $d(\bar{X}) \subseteq \omega \setminus A$. 
We show that $f\upharpoonright(\mathcal{C} \upharpoonright \bar{X})$ is continuous
by showing that there is a strictly increasing sequence $(m_{k})_{k<\omega}$ such that for each $Z \in \mathcal{C}\upharpoonright \bar{X}$, the initial segment $f(Z) \cap (k+1)$ of $f(Z)$ is determined by the initial segment $\rho(Z, m_{k})$ of $Z$.

For each $k<\omega$, let $i_{k}$ denote the least $i$ for which $y_{2i+1} \ge k$. 
Let $W \in \mathcal{C} \upharpoonright \bar{X}$ be given and let $s = \rho(W, y_{2{i}_{k}+1})$.  
Since $d(\bar{X}) \cap [ y_{2i_{k}+1}, y_{2i_{k}+2}) = \emptyset,$ it follows that $\rho(\bar{X}, y_{2i_{k}+1})=\rho(\bar{X}, y_{2{i}_{k}+2})$. 
Notice that $W\le \bar{X}$, $\bar{X}/y_{2i_{k}+2}  \le Y$, $Y/y_{2i_{k}+2}\le X_{2i_{k}+1}$ and $\rho(\bar{X},y_{2i_{k}+1})=\rho(\bar{X}, y_{2{i}_{k}+2})$
 From this it follows that $k \not \in f(W)$ if and only if $P_{k}(s,  X_{2i_{k}+1})$, which holds if and only if $P_{k}(s, \bar{X})$ holds. 
Let $m_{k}= y_{2i_{k}+2}$.
Then  $f\upharpoonright( \mathcal{C} \upharpoonright \bar{X})$ is continuous, since the question of whether or not $k\in f(W)$ is determined by the finite initial segment $\rho(W, m_{k})$ along with $\bar{X}$.
\end{proof}

 Extend $f\re(\mathcal{C}\upharpoonright \bar{X})$ to 
a function $f':\mathcal{C}\ra\mathcal{V}$ by defining $f'(X) = \bigcup\{ f(Y) : Y \in \mathcal{C}\upharpoonright \bar{X}$ and $Y \le X\}$, for $X \in \mathcal{C}$. 
Notice that  
$f':\mathcal{C}\rightarrow \mathcal{V}$ is monotone and
$f'\re(\mathcal{C}\re\bar{X})= f\re(\mathcal{C}\re\bar{X})$.
Further, for each $X\in \mathcal{C}$ and $k<\omega$, $k\not \in f'(X)$ if and only if for all $Y \in \mathcal{C}\upharpoonright \bar{X}$ with $Y \le X$, $k\not\in f'(X)$,
 and this holds if and only if $P_{k}( \rho( X, m_{k}), \bar{X})$ holds. 
Thus,
$f':\mathcal{C} \rightarrow \mathcal{V}$ is continuous, as whether $k\in f(X)$ is determined by the finite initial segment $\rho( X, m_{k})$ along with $\bar{X}$. 
Now define  $\hat{f}:\mathcal{AR} \rightarrow [\omega]^{<\omega}$ by $\check{f}(s) = \{ k \le \depth_{\mathbb{A}}(s) : \neg P_{k}(s, \bar{X})\}$, for $s\in \mathcal{AR}$; and 
define $\tilde{f}:\mathcal{R}\ra\mathcal{P}(\om)$ by
$\tilde{f}(Y)= \bigcup_{n<\omega} \hat{f}(r_{n}(Y))$, for $Y \in \mathcal{R}$.
Then $\hat{f}$ satisfies (3) in Definition \ref{defn.basic} and $\tilde{f}$ is continuous.
Notice  that   $f'(Y)= \bigcup_{n<\omega} \hat{f}(r_{n}(Y))$, for $Y \in \mathcal{C}$,
hence implying that $\tilde{f}\re\mathcal{C}= f'\re\mathcal{C}$.
Thus,  $\tilde{f}\re(\mathcal{C}\re\bar{X})= f'\re(\mathcal{C}\re\bar{X})$.
\end{proof}

\begin{cor} \label{cor.p-point Tukey theorem}
Let $\lgl \bsA_k:k<\om\rgl$ be a generating sequence as in Definition  \ref{defn.A_k}.
If $\mathcal{C}$ is a selective filter on $(\mathcal{R}(\lgl \bsA_k:k<\om\rgl),\le)$ such that $\{d(X): X\in \mathcal{C}\}$ generates a nonprincipal ultrafilter on $\omega$, then 
for each ultrafilter $\mathcal{V}$ Tukey reducible to $\mathcal{C}$,
$\mathcal{V}$ has basic Tukey reductions.
\end{cor}

\begin{proof}
This follows from Theorem \ref{p-point Tukey theorem}
and   the proof of Theorem  2.6 in \cite{DobrinenCanonicalMaps15}.
\end{proof}

The next result will be used in Section  \ref{sec.initstruc}  to identify initial structures in the Tukey types of p-point ultrafilters.

\begin{thm}\label{thm.Tim Step1} 
Suppose $\mathcal{C}\sse\mathcal{R}$ is a Nash-Williams filter on $(\mathcal{R},\le)$  and $\{d(X):X\in\mathcal{C}\}$ generates an ultrafilter on $\om$.
Then an ultrafilter $\mathcal{V}$ on $\om$ is Tukey reducible to $\mathcal{C}$ if and only if $\mathcal{V} =f( \left< \mathcal{C} \upharpoonright \mathcal{F} \right>)$ for some front $\mathcal{F}$ on $\mathcal{C}$ and some function $f: \mathcal{F} \rightarrow \omega$.
\end{thm}

\begin{proof}$(\Rightarrow)$ Suppose that $\mathcal{C}$ is Ramsey for $\mathcal{R}$. Proposition \ref{BasicReductionTukeyProp} and Theorem \ref{p-point Tukey theorem} show that if $\mathcal{V}$ is a non-principal ultrafilter on $\omega$ Tukey reducible to $\mathcal{C}$ then there is a front $\mathcal{F}$ on $\mathcal{C}$ and a function $f:\mathcal{F}\rightarrow \omega$ such that $\mathcal{V} =f( \left< \mathcal{C} \upharpoonright \mathcal{F} \right>)$.

$(\Leftarrow)$ Suppose that $\mathcal{F}$ is a front on $\mathcal{C}$, $f:\mathcal{F}\rightarrow \omega$ and $\mathcal{V}=f( \left< \mathcal{C} \upharpoonright \mathcal{F} \right>)$. The map sending $X \in \mathcal{C}$ to $f'' \mathcal{F}|X$ is a monotone cofinal map from $(\mathcal{C},\ge)$ to $(\mathcal{V},\supseteq)$. Thus, $\mathcal{V} \le_{T} \mathcal{C}$.
\end{proof}

When  $\mathcal{R}$ is a topological Ramsey space constructed from a generating sequence,
Theorem \ref{thm.Ramsey=NW} implies that 
the hypotheses of 
Theorem \ref{thm.Tim Step1} can be weakened to assuming that $\mathcal{C}$ is Ramsey.

The next fact shows that many topological Ramsey spaces give rise to selective filters with basic Tukey reductions.

\begin{fact} \label{selectiveExistence} Suppose that $\mathcal{R}$ has the property that for each $X \in \mathcal{R}$ and each $A \subseteq \omega$ there exists $Y\le X$  in $\mathcal{R}$ such that either $d(Y) \subseteq A$ or $d(Y) \subseteq \omega \setminus A$. Then assuming CH, MA or forcing with $\mathcal{R}$ using almost reduction, there exists a selective filter on $(\mathcal{R}, \le)$ with the property that $\{ d(X) : X\in \mathcal{C}\}$ generates a nonprincipal ultrafilter on $\omega$. 
\end{fact}

In \cite{Dobrinen/Todorcevic11} and \cite{Dobrinen/Todorcevic12}, Dobrinen and Todorcevic   introduced topological Ramsey spaces $\mathcal{R}_{\al}$,     
 $\al<\om_1$, which distill key properties of forcings of Laflamme in \cite{Laflamme89} and with associated ultrafilters with initial Tukey structure exactly that of a decreasing chain of order type $\al+1$. 
For $1\le n<\om$,
the space $\mathcal{R}_n$ is constructed from a certain tree of height $n+1$ which forms the top element of the space. When $n>1$, these spaces are not constructed from generating sequences.

Trujillo has shown in \cite{Trujillo2013selective} that there is a topological Ramsey space $\mathcal{R}_{n}^{\star}$ constructed from $\mathcal{R}_{n}$, such that forcing with $\mathcal{R}_{n}^{\star}$ using almost reduction adjoins a selective filter $\mathcal{C}$ on $(\mathcal{R}_{n}, \le)$ which is not a Ramsey filter on $(\mathcal{R}_{n},\le)$. Furthermore, it can be shown that $\{d(X): X\in \mathcal{C}\}$ generates an ultrafilter on $\omega$. 
Forcing with the space $\mathcal{R}_{n}^{\star}$ using almost reduction, or assuming CH or MA,
one can
 construct a selective but not Ramsey maximal filter on $(\mathcal{R}_{n},\le)$.
Such a filter has the property  that $\{ d(X) : X\in \mathcal{C}\}$ generates an ultrafilter on $\omega$. Theorem $\ref{p-point Tukey theorem}$ implies that these non-Ramsey filters on $(\mathcal{R}_{n},\le)$ have basic Tukey reductions.
 Using a similar argument, the work of Trujillo in \cite{Trujillo2013selective} shows that for each positive $n$,
using forcing or  assuming CH or MA, there is a selective but not Ramsey filter on $(\mathcal{H}^{n},\le)$ with basic Tukey reductions. (Recall  $\mathcal{H}^{n}$ from Example \ref{ex.H}.)

If $\mathcal{R}$ is constructed from some generating sequence then 
Theorems \ref{p-point Tukey theorem}  and \ref{thm.Tim Step1}
reduce the identification of ultrafilters on $\omega$
which are Tukey reducible to a Ramsey filter $\mathcal{C}$ associated with  $(\mathcal{R},\le)$ to the study of Rudin-Keisler reduction on ultrafilters on base sets which are  fronts on $\mathcal{C}$.
 In the next section we show that the 
Ramsey-classification Theorem \ref{canonical R}
can be localized  to
 equivalence relations on fronts on a Ramsey filter on $(\mathcal{C},\le)$.
We then use it identify initial structures in the Tukey types of ultrafilters Tukey reducible to any Ramsey filter associated with a Ramsey space constructed from a generating sequence.


\section{Initial structures in the Tukey and Rudin-Keisler types of p-points}\label{sec.initstruc}

The  structure of the Tukey types of ultrafilters  (partially ordered by $\contains$)  was  studied in \cite{Dobrinen/Todorcevic10}.
In that paper, it is shown that large chains, large antichains, and diamond configurations embed into the Tukey types of p-points.  However, this left open the question of what  the exact structure of all Tukey  types below a given p-point is.
Recall that we use the terminology 
{\em initial Tukey structure} below an ultrafilter $\mathcal{U}$ to refer to the structure of the Tukey types of {\em all} nonprincipal ultrafilters Tukey reducible to $\mathcal{U}$ 
(including $\mathcal{U}$).

In \cite{Raghavan/Todorcevic12},
Todorcevic showed that the initial Tukey structure below a Ramsey ultrafilter on $\om$ consists  exactly of  one Tukey type, namely that of the Ramsey ultrafilter.
In \cite{Dobrinen/Todorcevic11} and \cite{Dobrinen/Todorcevic12}, Dobrinen and Todorcevic  showed that for each $1\le \al<\om_1$, there are Ramsey spaces with associated ultrafilters which have initial 
Tukey and initial Rudin-Keisler structures which are decreasing chains of order type $\al+1$.  This left open the following questions from the Introduction,  which we restate here.

\begin{question1}
What are the possible initial Tukey structures for ultrafilters on a countable base set?
\end{question1}

\begin{question2}
What are the possible initial Rudin-Keisler structures for ultrafilters on a countable base set?
\end{question2}

\begin{question3}
 For a given ultrafilter $\mathcal{U}$, what is the structure of the Rudin-Keisler ordering of the isomorphism classes of ultrafilters Tukey reducible to $\mathcal{U}$?
\end{question3}

In this section, we answer Questions 1  - 3 for all 
Ramsey filters associated with a Ramsey space constructed from a generating sequence with \Fraisse\ classes which have the Order-Prescribed Free Amalgamation Property.
The results in 
Theorems \ref{Tim-InitialStructures}  and \ref{TukeyReducibleTheorem2}
show the surprising  fact that the structure of the \Fraisse\ classes used for the generating sequence have bearing on the initial Rudin-Keisler structures, but not on the intial Tukey structures.

In this section we use topological Ramsey spaces constructed from generating sequences to identify some initial  structures in the Tukey types of p-points. 
The next theorem is one of the main results, and will be proved at the end of this section.

\begin{thm} \label{Tim-InitialStructures} 
Let  $\mathcal{C}$ be a Ramsey filter on 
a Ramsey space constructed from a generating sequence for \Fraisse\ classes of ordered relational structures with the Ramsey property and the OPFAP.
\begin{enumerate}
\item
If $J<\om$,
then the initial Tukey structure of all ultrafilters Tukey reducible to $\mathcal{C}$ is exactly $\mathcal{P}(J)$.
\item
If  $J\le\om$, then 
  the Tukey ordering of the p-points Tukey reducible to $\mathcal{C}$ is isomorphic to the partial order $([J]^{<\omega}, \subseteq)$. 
\end{enumerate}
\end{thm}

From Theorem \ref{Tim-InitialStructures},  the following corollary is immediate.

\begin{cor} \label{cor.Tim-InitialStructures} 
It is consistent with ZFC that the following statements hold.
\begin{enumerate}
\item 
Every finite Boolean algebra appears as  the initial Tukey structure below some p-point.
\item 
The structure of the Tukey types of p-points  contains the partial order $([\omega]^{<\omega}, \subseteq)$ as an initial structure.
\end{enumerate}
\end{cor}

The archetype for the proofs and results in this section comes from work in \cite{Dobrinen/Todorcevic11} showing that the initial Tukey structure below the ultrafilter associated with the space $\mathcal{R}_1$ is exactly a chain of length 2.
  (See Theorem 5.18 in \cite{Dobrinen/Todorcevic11}  and results leading up to it.)
The outline of that proof is now presented, as it will be followed in this section in more generality.
\vskip.1in

\noindent \bf Outline of Proof of Theorem Theorem 5.18 in \cite{Dobrinen/Todorcevic11}.
 \rm
Recall that the space $\mathcal{R}_1$ in \cite{Dobrinen/Todorcevic11} is exactly the
topological Ramsey space $\mathcal{R}(\left<\mathbf{A}_{k} : k<\omega\right>)$ where $J=1$ and for each $k<\omega$, $\mathbf{A}_{k,0}$ is a linear order of cardinality $k$.  
Let  $\mathcal{C}$ be a maximal filter Ramsey for $\mathcal{R}_1$ and $\mathcal{U}_1$ be the ultrafilter it generates on the leaves of the base tree.

Theorem 5.18   in \cite{Dobrinen/Todorcevic11} was obtained in six main steps. (1) Theorem 20 from \cite{Dobrinen/Todorcevic10}, every p-point has basic monotone reductions, was used to show that all ultrafilters Tukey reducible to $\mathcal{U}_{1}$ are of the form $f(\left<\mathcal{C}\upharpoonright \mathcal{F}\right>)$ for some front on $\mathcal{C}$. 
(2) A localized version of the Ramsey-classifcation theorem for equivalence relations on fronts on $\mathcal{C}$ was shown to hold. 
(3) For each $n<\omega$, it was shown that the filter $\mathcal{Y}_{n+1}$ on the base set $\mathcal{R}_{1}(n)$ generated by $\mathcal{C}\upharpoonright\mathcal{R}_{1}(n)$ is a p-point ultrafilter. 
Furthermore, it was shown that $\mathcal{Y}_{1} <_{RK} \mathcal{Y}_{2} <_{RK} \dots$.  
(4) The localized Ramsey-classification theorem and the  canonical equivalence  relations were used to show that all ultrafilters Tukey reducible to $\mathcal{U}_{1}$ are isomorphic to an ultrafilter of $\vec{\mathcal{W}}$-trees, where $\hat{\mathcal{S}}\setminus \mathcal{S}$ is a well-founded tree, $\vec{\mathcal{W}}= ( \mathcal{W}_{s} : s \in\hat{\mathcal{S}}\setminus \mathcal{S})$, and each $\mathcal{W}_{s}$ is isomorphic to $\mathcal{Y}_{n+1}$ for some $n<\omega$ or isomoprhic to $\mathcal{U}_{0}$. (5) The theory of uniform fronts was used to show that each ultrafilter generated by a $\vec{\mathcal{W}}$-tree is isomorphic to a countable Fubini product from among the ultrafilters $\mathcal{Y}_{n}$, $n<\omega$.
 (6) The result on Fubini products was used to show that the Tukey structure of the non-principal ultrafilters on $\omega$ Tukey reducible to $\mathcal{U}_{1}$ is isomorphic to the two element Boolean algebra and that the p-points Tukey reducible to $\mathcal{U}_{1}$ are exactly $\{\mathcal{Y}_{n}: n<\omega\}$.
\vskip.1in

In order to avoid repeating phrases we fix some notation for the remainder of the section.
Fix $1\le J\le\omega$ and $\mathcal{K}_{j}$, $j\in J$, a collection of Fra\"{i}ss\'{e} classes of finite ordered relational structures such that each $\mathcal{K}_{j}$ satisfies the Ramsey property and the OPFAP.
Let $\mathcal{K}$ denote $(\mathcal{K}_{j})_{j\in J}$.
Let $\left< \mathbf{A}_{k} : k<\omega \right>$ be a generating sequence, and  let $\mathcal{R}$ denote the topological Ramsey space $\mathcal{R}(\left< \mathbf{A}_{k} : k<\omega \right>$).

Theorem \ref{thm.Tim Step1}  
verifies that step (1) can be carried out for any Ramsey filter on $(\mathcal{R},\le)$.
 In the remainder of this section, we show that analogues of steps (2) - (6) can be carried out for any Ramsey filter on $(\mathcal{R},\le)$. 
The first part  of step (2), proving the Ramsey-classifiication  theorem for $\mathcal{R}$,  was  obtained in  Theorem $\ref{canonical R}$. 
We complete step (2) by showing that a localized version of Theorem \ref{canonical R} holds for Ramsey filters on $(\mathcal{R},\le)$.   
The analogue of step (3) is not as straightforward.
First we introduce $\mathcal{K}_{\fin}$
 and then associate to each $\mathbf{B}\in \mathcal{K}_{\fin}$ a p-point ultrafilter $\mathcal{U}_{\mathbf{B}}$ (see Notation \ref{notn.importantt}). 
Then we show that the Rudin-Keisler structure of these p-points is  isomorphic to  $\tilde{\mathcal{K}}_{\fin}$,
the collection of equivalence classes of members of 
$\mathcal{K}_{\fin}$, partially ordered by embeddability.
Steps (4) and (5) are then generalized,
the only difference being that the nodes of the $\vec{\mathcal{W}}$-trees are taken to be the p-points $\mathcal{U}_{\mathbf{B}}$, $\mathbf{B}\in \mathcal{K}_{\fin}$, from step (3). 
Step (6) will be completed at the end of the section by proving Theorem $\ref{Tim-InitialStructures}$ and identifying initial structures in the Tukey types of ultrafilters.

 The next theorem completes step (2) for topological Ramsey spaces constructed from generating sequences.

\begin{thm}\label{Tim-LocalCanonical}
Let  $\mathcal{C}$ be a Ramsey filter on 
a Ramsey space constructed from a generating sequence for \Fraisse\ classes of ordered relational structures with the Ramsey property and the OPFAP.
 If $\mathcal{C}$ is a Ramsey filter on $(\mathcal{R}, \le)$, then for any front $\mathcal{F}$ on $\mathcal{R}$ and any equivalence relation $R$ on $\mathcal{F}$, there exists a $C \in \mathcal{C}$ such that $R$ is canonical on $\mathcal{F}|C$.
\end{thm}

\begin {proof}
Since $\mathcal{C}$ satisfies the Abstract Nash-Williams Theorem for $\mathcal{R}$, and $\mathcal{C}$  is also selective for $\mathcal{R}$, by Lemma 3.8 in \cite{mijares2007notion}. 
Thus,  the   proof of Theorem \ref{canonical R}
can be relativized to $\mathcal{C}$.
\end{proof}

Next we complete step (3) for the general case by first identifying the p-points to be used as the nodes of the $\vec{\mathcal{W}}$-trees we encounter in step (4), and then determining the Rudin-Keisler structure among these p-points.

\begin{fact} 
If $\mathcal{C}\subseteq \mathcal{R}$ is a Ramsey filter on $(\mathcal{R},\le)$, then for each $n<\omega$, $\mathcal{C}\upharpoonright \mathcal{R}(n) = \{ \mathcal{R}(n) | C : C \in \mathcal{C}\}$ generates an ultrafilter on base set $\mathcal{R}(n)$.
\end{fact}

\begin{notationn}\label{notn.importantt}
Suppose that $\mathcal{C}\subseteq \mathcal{R}$ is a Ramsey filter on $(\mathcal{R},\le)$. 
For each $n<\omega$, define $\mathcal{Y}_{n+1}$ to be the ultrafilter on $\mathcal{R}(n)$ generated by $\mathcal{C}\upharpoonright \mathcal{R}(n).$  
Define $\mathcal{Y}_{0} = \pi_{\depth} (\mathcal{Y}_{1})$ and $\mathcal{Y}_{\left< \right>} = \pi_{\left< \right>}( \mathcal{Y}_{1})$. 
Let 
\begin{equation}
\mathcal{K}_{\fin} = \{ (\mathbf{B}_{j})_{j\in K} :  K \in [J]^{<\omega}\mathrm{\ and \ } (\mathbf{B}_{j})_{j\in K} \in (\mathcal{K}_{j})_{j\in K}\}.
\end{equation}
For $\mathbf{B}=(\mathbf{B}_{j})_{j\in K}$ and $\mathbf{C}=(\mathbf{C}_{j})_{j\in L}$ in $\mathcal{K}_{\fin}$, define  $\mathbf{B} \le\mathbf{C}$ if and only if $K \subseteq L$ and for all $j\in K$, $\mathbf{B}_{j} \le \mathbf{C}_{j}.$
Let $\tilde{\mathcal{K}}_{\fin}$ denote the collection of equivalence classes of members of $\mathcal{K}_{\fin}$.
Then  $\tilde{\mathcal{K}}_{\fin}$ is partially ordered by $\le$.

For  $\mathbf{B}=(\mathbf{B}_{j})_{j\in K}\in\mathcal{K}_{\fin}$ with $K\not=\emptyset$, define the following.
\begin{enumerate}
\item 
Define $J_{\mathbf{B}}$ to be $K$ and define $$\mathcal{B}(\mathbf{B}) = \bigcup_{n<\omega} \left \{ \left<n, (\mathbf{C}_{j})_{j\in J_{\mathbf{B}}}\right>: \forall j\in J_{\mathbf{B}}, \mathbf{C}_{j} \in{\mathbf{A}_{n,j}\choose \mathbf{B}_{j}} \right \}.$$
\item 
Applying the joint embedding property once for each $j\in J_{\mathbf{B}}$ and using the definition of generating sequence, there is an $n$ such that for each $j\in K$, $\mathbf{B}_{j} \le  \mathbf{A}_{n,j}$. Define $n(\mathbf{B})$ to be the smallest natural number $n$ such that for each $j\in J_{\mathbf{B}}$, $\mathbf{B}_{j} \le \mathbf{A}_{n,j}.$ 
\item 
Let $\mathbb{I}(\mathbf{B})$ denote the sequence $(I_{j})_{j\in J_{n(\mathbf{B})}}$ such that $\pi_{\mathbb{I}(\mathbf{B})}(\mathbb{A}(n(\mathbf{B}))) =\left< n(\mathbf{B}), \mathbf{C}\right>$,
where $\mathbf{C}$ is the lexicographical-least element of ${(\mathbf{A}_{n(\mathbf{B}), j})_{j\in J_{\mathbf{B}}} \choose \mathbf{B}}$.
\item 
We use the slightly more compact notation $\pi_{\mathbf{B}}$ to denote
the map $\pi_{\mathbb{I}(\mathbf{B})}$.
\item  
Let $\mathcal{U}_{\mathbf{B}}$ denote the ultrafilter $\pi_{\mathbf{B}}(\mathcal{Y}_{n(\mathbf{B})+1})$ on the base set $\mathcal{B}(\mathbf{B})$.
\end{enumerate}
 We let $\emptyset$ denote the sequence in $\mathcal{K}_{\fin}$ with $K=\emptyset$. 
\end{notationn}

The next proposition describes the configuration
 of the ultrafilters $\mathcal{U}_{\mathbf{B}}$ with $\mathbf{B} \in \mathcal{K}_{\fin}$ and the projection ultrafilters $\pi_{\mathbb{I}}(\mathcal{Y}_{i})$ with $ i<\omega$ and $\pi_{\mathbb{I}}$ a projection map on $\mathcal{R}(i)$, with respect to the Rudin-Keisler ordering. For the remainder of the section, if $\pi_{\mathbb{I}}$ is a projection map on $\mathcal{R}(i)$ with $\mathbb{I}=(I_{j})_{j\in J_{i}}$,
 then we let $J_{\mathbb{I}}$ denote the set $\{j\in J_{i} : I_{j}\not = \emptyset\}$.
Recall that we write $\mathcal{U}\cong\mathcal{V}$ to denote that the two ultrafilters are Rudin-Keisler equivalent.

\begin{prop} \label{Tim RKstructure}
Suppose that $\mathcal{C}\subseteq \mathcal{R}$ is a
 Ramsey filter on $(\mathcal{R},\le)$.
\begin{enumerate}
\item 
$\mathcal{Y}_{0}$ is a Ramsey ultrafilter and $\mathcal{Y}_{1}$ is not a Ramsey ultrafilter.
\item 
For each $n<\omega$, $\mathcal{Y}_{n+1} = \mathcal{U}_{\mathbf{A}_{n}}$.
\item 
For each $m<\omega$ and each projection map $\pi_{\mathbb{I}}$ with domain $\mathcal{R}(m)$,
 there exists $\mathbf{B}\in \mathcal{K}_{\fin}$ such that $\pi_{\mathbb{I}}(\mathcal{Y}_{m+1}) \cong \mathcal{U}_{\mathbf{B}}$.
\item
 For each $\mathbf{B}\in \mathcal{K}_{\fin}$, $\mathcal{U}_{\mathbf{B}}$ is a rapid p-point.
\item 
For $\mathbf{B}$ and $\mathbf{C}$ in $\mathcal{K}_{\fin}$, $ \mathcal{U}_{\mathbf{B}} \le_{RK} \mathcal{U}_{\mathbf{C}}$ if and only if $\mathbf{B} \le \mathbf{C}\mbox{ in } \mathcal{K}_{\fin}.$
\end{enumerate}
\end{prop}

\begin{proof}
($1$) $\mathcal{Y}_{0}=\pi_{\depth}(\mathcal{Y}_{1})$ is a Ramsey ultrafilter since $\{ \pi_{\depth} '' \mathcal{R}(0)\upharpoonright C : C\in \mathcal{R}\}$ is identical to the Ellentuck space. $\mathcal{Y}_{1}$ is not Ramsey since the map $\pi_{\depth}$ is not one-to-one on any element of $\mathcal{Y}_{1}$.  ($2$) For each $n<\omega$, $n(\mathbf{A}_{n})=n$ and $\pi_{\mathbf{A}_{n}}$ is the identity map on $\mathcal{B}(\mathbf{A}_{n})$. Thus, $\mathcal{U}_{\mathbf{A}_{n}} = \pi_{\mathbf{A}_{n}}( \mathcal{Y}_{n+1}) = \mathcal{Y}_{n+1}$.

($3$) Suppose that $\pi_{\mathbb{I}}$ is a projection map with domain $\mathcal{R}(m)$. 
Let $\mathbf{B}=(\mathbf{B}_{i})_{i\in J_\mathbb{I}}$ be the substructure of $(\mathbf{A}_{m,j})_{j\in J_{\mathbb{I}}}$ such that $\pi_{\mathbb{I}}(\mathbb{A}(m)) = \left< m, \mathbf{B}\right>$. 
Let $n= n(\mathbf{B})$, and let $\mathbf{C}=(\mathbf{C}_{i})_{i\in J_{\mathbb{I}}}$ be the  substructure of $(\mathbf{A}_{n,j})_{j\in J_{\mathbb{I}}}$ such that $\pi_{\mathbf{B}}(\mathbb{A}(n)) = \left< n, \mathbf{C}\right>$.  
We will show that $\mathcal{U}_{\mathbf{B}} = \pi_{\mathbf{B}}(\mathcal{Y}_{n+1})\cong \pi_{\mathbb{I}}(\mathcal{Y}_{m+1})$.

Let $f:\mathbf{B} \rightarrow (\mathbf{A}_{m,j})_ {j\in J_{\mathbb{I}}}$ be the embedding with range $\mathbf{B}$ and $g: \mathbf{B}\rightarrow(\mathbf{A}_{n,j})_{ j\in J_{\mathbb{I}}}$ be the embedding with range $\mathbf{C}$. 
By the amalgamation property for $\mathcal{K}_{j}$, $j\in J_{\mathbb{I}}$, and the definition of generating sequence, there exist $k<\omega$ and embeddings $r:(\mathbf{A}_{n,j})_{ j\in J_{\mathbb{I}}} \rightarrow(\mathbf{A}_{k,j})_{ j\in J_{\mathbb{I}}}$ and $s:(\mathbf{A}_{m,j})_{ j\in J_{\mathbb{I}}} \rightarrow (\mathbf{A}_{k,j})_{ j\in J_{\mathbb{I}}}$
 such that $r \circ f = s \circ g$. 
Let $\mathbf{F}$, $\mathbf{G}$ and $\mathbf{H}$ denote the substructures of $(\mathbf{A}_{k,j})_{ j\in J_{\mathbb{I}}}$ generated by the ranges of $r\circ f$, $s$, and $r$, respectively.  
Let $\pi_{\mathbb{M}}$, $\pi_{\mathbb{N}}$ and $\pi_{\mathbb{F}}$ denote projection maps on $\mathcal{R}(k)$ such that $J_{\mathbb{M}}=J_{\mathbb{N}}=J_{\mathbb{F}}=J_{\mathbb{I}}$, $\pi_{\mathbb{M}}(\mathbb{A}(k))= \left< k , \mathbf{G}\right>$, $\pi_{\mathbb{N}}(\mathbb{A}(k))= \left< k , \mathbf{H}\right>$ and $\pi_{\mathbb{F}}(\mathbb{A}(k))= \left< k , \mathbf{F}\right> $.  
Since $r \circ f = s\circ g$, it follows that for all $y \in \mathcal{AR}_{k+1}$, $\pi_{\mathbf{B}}\circ \pi_{\mathbb{N}}(y(k))=\pi_{\mathbb{I}}\circ\pi_{\mathbb{M}} (y(k))= \pi_{\mathbb{F}}(y(k))$.

Let $X\in \mathcal{C}$ and consider the set $\mathcal{G}=\{x \in \mathcal{AR}_{n+1}: \exists  y \in \mathcal{R}(k)|X, \  \pi_{\mathbb{F}}(y)= \pi_{\mathbf{B}}(x(n))\} $. 
Since $\mathcal{C}$ satisfies the Abstract Nash-Williams Theorem it follows that there exists a $Y\le X$ in $\mathcal{C}$ such that 
either $\mathcal{G} \cap \mathcal{AR}_{n+1}|Y = \emptyset$ or  $\mathcal{AR}_{n+1}|Y\subseteq \mathcal{G}$. 
Since there exists $z \in \mathcal{AR}_{n+1}|Y$ such that  $\pi_{\mathbb{F}}(Y(k))= \pi_{\mathbf{B}}(z(n))$ it must be the case that $\mathcal{AR}_{n+1}|Y\subseteq \mathcal{G}$. 
By Fact $\ref{EqualityFact}$  it follows that $\pi_{\mathbb{F}}(\mathcal{Y}_{k}) \cong \pi_{\mathbf{B}}(\mathcal{Y}_{n})$. By a similar argument, we also have $\pi_{\mathbb{F}}(\mathcal{Y}_{k}) \cong \pi_{\mathbb{I}}(\mathcal{Y}_{m})$. 
Thus, $ \mathcal{U}_{\mathbf{B}}=\pi_{\mathbf{B}}(\mathcal{Y}_{n}) \cong \pi_{\mathbb{I}}(\mathcal{Y}_{m})$.

($4$) Let $K$ be a finite subset of $J$ and $\mathbf{B}=(\mathbf{B}_{j})_{j\in K}\in\mathcal{K}_{\fin}$. Suppose that $X_{0}\supseteq X_{1} \supseteq X_{2} \supseteq \cdots$ is a sequence of sets in $\mathcal{U}_{\mathbf{B}}$. Then there exists a sequence $C_{0} \ge C_{1} \ge C_{2} \ge \cdots$ of elements of $\mathcal{C}$ such that for each $i<\omega$, $\pi_{\mathbf{B}}'' (\mathcal{R}(n(\mathbf{B}))\upharpoonright C_{i})\subseteq X_{i}$. Since every Ramsey filter on $(\mathcal{R},\le)$ is also a selective filter for $(\mathcal{R},\le)$, there exists $C\in \mathcal{C}$ such that for each $i<\omega$, $C/r_{i}(C) \le C_{i}$. 
Since each $\mathcal{K}_{j}$, $j\in K$, consists of finite structures and $K$ is finite, it follows that for each $i<\omega$, $\pi_{\mathbf{B}}'' (\mathcal{R}(n)\upharpoonright C)\subseteq^{*} \pi_{\mathbf{B}}'' (\mathcal{R}(n)\upharpoonright C_{i})$. Therefore $\mathcal{U}_{\mathbf{B}}$ is a p-point.

Let $h: \omega \rightarrow \omega$ be a strictly increasing function. Linearly order $\mathcal{B}(\mathbf{B})$ so that $\left< i, \mathbf{C}\right>$ comes before $\left< j, \mathbf{D}\right>$ whenever $i<j$. For each $B\in \mathcal{R}$, there is a $C\le B$ such that $\pi_{\depth}(C(n-1)) > h(1),$ $\pi_{\depth}(C(n(\mathbf{B}))) > h ( 1+ |\mathcal{B}(\mathbf{B})\upharpoonright \left<n(\mathbf{B}), \mathbf{A}_{n(\mathbf{B})}\right>| )$, and in general, for $k>n(\mathbf{B})$,
\begin{equation}
\pi_{\depth}( C(k)) > h ( \sum_{i=n}^{k} | \mathcal{B}(\mathbf{B})\upharpoonright \left<i, \mathbf{A}_{i}\right>|).
\end{equation}
Since $\mathcal{C}$ is selective for $\mathcal{R}$, there is a $C \in \mathcal{C}$ with this property, which yields that $\mathcal{U}_{\mathbf{B}}$ is rapid. 

($5$)  ($\Leftarrow$) Suppose that $\mathbf{B}=(\mathbf{B}_{j})_{j\in K}$ and $\mathbf{C}=(\mathbf{C}_{j})_{j\in L}$ are elements of $\mathcal{K}_{\fin}$ and $\mathbf{B} \le \mathbf{C}$. Let $\mathbb{I}(\mathbf{C})= (I_{j})_{j\in J_{n(\mathbf{C})}}$. Then $K\subseteq L$ and there is a sequence $\mathbb{I} = (I'_{j})_{j\in J_{n(\mathbf{C})}}$ such that for each $j\in J_{n(\mathbf{C})}$, $I'_{j} \subseteq I_{j}$, and the structure $\mathbf{D}$ in $\mathcal{K}_{\fin}$ such that $\pi_{\mathbb{I}}(\mathbb{A}(n(\mathbf{C})))= \left< n(\mathbf{C}), \mathbf{D}\right>$ is isomorphic to $\mathbf{B}$. By the work in part (3) of this proposition, $\pi_{\mathbb{I}}(\mathcal{Y}_{n(\mathbf{C})+1})\cong \pi_{\mathbf{B}}(\mathcal{Y}_{n(\mathbf{B})+1})=\mathcal{U}_{\mathbf{B}}.$ Since for each $j\in J_{n(\mathbf{C})}$, $I'_{j}\subseteq I_{j}$, we have $\pi_{\mathbb{I}}(\mathcal{Y}_{n(\mathbf{C})+1} )\le_{RK} \pi_{\mathbf{C}}(\mathcal{Y}_{n(\mathbf{C})+1}) = \mathcal{U}_{\mathbf{C}}$. Hence, $\mathcal{U}_{\mathbf{B}} \le_{RK} \mathcal{U}_{\mathbf{C}}.$

 ($5$) ($\Rightarrow$)   Next suppose that $(\mathbf{B}_{j})_{j\in K}$ and $(\mathbf{C}_{j})_{j\in L}$ are elements of $\mathcal{K}_{\fin}$ and $\mathcal{U}_{\mathbf{B}} \le_{RK} \mathcal{U}_{\mathbf{C}}$.

\begin{lem}\label{lem.important}
For each nonprincipal ultrafilter $\mathcal{V}$ on $\omega$ with $\mathcal{V}\le_{RK} \mathcal{U}_{\mathbf{C}}$, there exists $\mathbf{D}\in \mathcal{K}_{\fin}$ such that $\mathbf{D} \le \mathbf{C}$ and $\mathcal{V} \cong \mathcal{U}_{\mathbf{D}}$. 
Furthermore, if $\mathcal{V} \cong \mathcal{U}_{\mathbf{C}}$ then $J_{\mathbf{D}}=J_{\mathbf{C}}$ and for all $j\in J_{\mathbf{C}}$, $\mathbf{C}_{j} \cong \mathbf{D}_{j}$.
\end{lem}

\begin{proof} 
Suppose that $\mathcal{V}$ is a nonprincipal ultrafilter on $\omega$ such that $\mathcal{V} \le_{RK} \mathcal{U}_{\mathbf{C}}$. 
Then there is a function $\theta:\mathcal{B}(\mathbf{C})\rightarrow \omega$ such that $\theta(\mathcal{U}_{\mathbf{C}})= \mathcal{V}$. 
Since $\theta \circ \pi_{\mathbf{C}} : \mathcal{R}(n(\mathbf{C})) \rightarrow \omega$, Theorem $\ref{Tim-LocalCanonical}$ implies that there exist an $X\in \mathcal{C}$ and a projection map $\pi_{\mathbb{I}}$ on $\mathcal{R}(n(\mathbf{C}))$ such that for all $y,z \in \mathcal{R}(n(\mathbf{C}))\upharpoonright X$, 
\begin{equation}\label{Tim-equation1}
\theta \circ \pi_{\mathbf{C}}(y) = \theta \circ \pi_{\mathbf{C}}(z) \mbox{ if and only if } \pi_{\mathbb{I}}(y) =  \pi_{\mathbb{I}}(z).
\end{equation}
Suppose $\mathbb{I} = ( I_{j})_{j\in J_{n(\mathbf{C})}}$ and $\mathbb{I}(\mathbf{C}) = ( I_{j}')_{j\in J_{n(\mathbf{C})}}$.
 Let $\mathbf{D}\in \mathcal{K}_{\fin}$ such that $\pi_{\mathbb{I}}(\mathbb{A}(n(\mathbf{C}))) = \left< n(\mathbf{C}), \mathbf{D}\right >$. 
If there exists $j\in J_{n(\mathbf{C})}$ such that $I_{j}' \not \subseteq I_{j}$ or there exists $j\in J_{\mathbb{I}}$ such that $\mathbf{D}_{j} \not \le \mathbf{C}_{j}$, then there exist $s,t\in \mathcal{R}(n(\mathbf{C}))\upharpoonright X$ such that $\pi_{\mathbb{I}}(s) \not = \pi_{\mathbb{I}}(t)$ and $\pi_{\mathbf{C}}(s) =\pi_{\mathbf{C}}(t)$. 
However, this is a contradiction to equation $(\ref{Tim-equation1})$. Therefore, $J_{\mathbf{D}}\subseteq J_{\mathbf{C}}$ and for all $j \in  J_{n(\mathbf{D})}$, $\mathbf{D}_{j} \le \mathbf{C}_{j}$, \emph{i.e.} $\mathbf{D} \le \mathbf{C}$. Additionally, equation ($\ref{Tim-equation1}$) shows that $\mathcal{U}_{\mathbf{D}}\cong \theta(\mathcal{U}_{\mathbf{C}})= \mathcal{V}.$

Next suppose that $\mathcal{V} \cong \mathcal{U}_{\mathbf{C}}$. Then there exists $Y\in \mathcal{C}$ such that $Y \le X$ and $\theta$ is injective on $\pi_{\mathbf{C}}''(\mathcal{R}(n(\mathbf{C}))\upharpoonright Y).$ 
If there is a $j\in J_{n(\mathbf{C})}$ such that $I_{j} \not \subseteq I_{j}'$ or there is a $j\in J_{\mathbb{I}}$ such that $\mathbf{C}_{j} \not \le \mathbf{D}_{j}$, 
then there are $s,t\in \mathcal{R}(n(\mathbf{C}))\upharpoonright Y$ such that $\pi_{\mathbb{I}}(s) = \pi_{\mathbb{I}}(t)$ and $\pi_{\mathbf{C}}(s) \not =\pi_{\mathbf{C}}(t)$. 
However, this contradicts the fact that $\theta$ is injective on $\pi_{\mathbf{C}}''(\mathcal{R}(n(\mathbf{C}))\upharpoonright Y)$. Therefore, $J_{\mathbf{C}}\subseteq J_{\mathbf{D}}$ and for all $j \in  J_{n(\mathbf{C})}$, $\mathbf{C}_{j} \le \mathbf{D}_{j}$,
that is, $\mathbf{C} \le \mathbf{D}$. 
Thus, $J_{\mathbf{D}}= J_{\mathbf{C}}$ and for all $j\in J_{\mathbf{C}}$, $\mathbf{C}_{j} \cong \mathbf{D}_{j}.$ \end{proof}

 Since $\mathcal{U}_{\mathbf{B}} \le_{RK} \mathcal{U}_{\mathbf{C}}$, Lemma \ref{lem.important} shows that $\mathbf{B} \le \mathbf{C}$.
\end{proof}

In what follows, 
we omit any proofs of results which follow the exact same argument as their counterparts in the proof of Theorem 5.10 in \cite{Dobrinen/Todorcevic11}.
The following makes use of the correspondence between iterated Fubini products of ultrafilters and so-called ultrafilters of $\vec{\mathcal{W}}$-trees on a flat-top front, 
(see Definition 3.2 and Facts 3.4 and 3.4 in \cite{DobrinenCanonicalMaps15}).
A uniform front is, in particular, a flat-top front, and the projection of the uniform front $\mathcal{C}|C$ in the next theorem  to $\hat{\mathcal{S}}$ will also be a flat-top front.

\begin{thm}\label{TukeyReducibleTheorem}
Suppose that $\mathcal{C}$ is a Ramsey filter on $(\mathcal{R}, \le )$.  
If $\mathcal{V}$ is a non-principal ultrafilter and $\mathcal{C}\ge_{T} \mathcal{V}$, then $\mathcal{V}$ is isomorphic to a Fubini 
iterate of p-points from among $\mathcal{U}_{\mathbf{B}}$, 
$\mathbf{B} \in \mathcal{K}_{\fin}$.
Precisely,
$\mathcal{V}$ is isomorphic to an 
ultrafilter of $\vec{\mathcal{W}}$-trees, where $\hat{\mathcal{S}}\setminus \mathcal{S}$ is a well-founded tree, $\vec{\mathcal{W}}= ( \mathcal{W}_{s} : s \in\hat{\mathcal{S}}\setminus \mathcal{S})$, and each $\mathcal{W}_{s}$ is isomorphic to $\mathcal{U}_{\mathbf{B}}$, for some $\mathbf{B} \in \mathcal{K}_{\fin}$.
\end{thm}

\begin{proof} 
Suppose that $\mathcal{C}$ and $\mathcal{V}$ are given and satisfy the assumptions of the theorem. 
By Proposition $\ref{Proposition 5.5 Analogue}$ and Lemma $\ref{Tim-LocalCanonical}$ there are a front $\mathcal{F}$ on $\mathcal{C}$, a function $f:\mathcal{F} \rightarrow \omega$, and a $C\in \mathcal{C}$ such that the following hold:
\begin{enumerate}
\item The equivalence relation induced by $f$ on $\mathcal{F}|C$ is canonical.
\item $\mathcal{V} = f( \left< \mathcal{C}\upharpoonright \mathcal{F}\right>).$
\end{enumerate}
A straightforward induction argument on the rank of fronts, along with the fact that $\mathcal{F}$ is Ramsey, shows that there is a $C'\le C\in\mathcal{C}$ such that $\mathcal{F}| C'$ is a uniform front on $\mathcal{C}|C'$.
From now on, we  will abuse notation and let $\mathcal{F}$ denote $\mathcal{F}|C'$ and $\mathcal{C}$ denote $\mathcal{C}|C'$.

Let $\mathcal{S} = \{ \varphi(t): t\in \mathcal{F}\}$, where 
$\vp$ is the inner Nash-Williams map from Theorem \ref{canonical R} 
which represents the canonical equivalence relation.
 The filter $\mathcal{W}$ on the base set $\mathcal{S}$ generated by $\varphi(\mathcal{C}\upharpoonright \mathcal{F})$ is an ultrafilter, and $\mathcal{W}\cong \mathcal{V}$.
 We omit the proof of this fact, since it follows from exactly the same argument as its counterpart in the proof of Theorem 5.10 in \cite{Dobrinen/Todorcevic11}.

Recall from the proof of Theorem \ref{canonical R}  that for each $t\in\mathcal{F}$ and $i<|t|$, there is a projection map $\pi_{r_i(t)}$ defined on $\mathcal{R}(i)$ such that $\vp(t)=\bigcup_{i<|t|}\pi_{r_i(t)}(t(i))$.
We  now extend $\vp$ to a map on all of $\hat{\mathcal{F}}$ by defining
$\vp(r_j(t))=\bigcup_{i<j}\pi_{r_i(t)}(t(i))$,
for $t\in\mathcal{F}$ and $j\le |t|$.

 Let  $\hat{\mathcal{S}}$ denote the collection of all initial segments of elements of $\mathcal{S}$.
Thus, $\hat{\mathcal{S}}=\{\vp(w):w\in\hat{\mathcal{F}}\}$. $\hat{\mathcal{S}}$ forms a well-founded tree under the ordering $\sqsubseteq$. 
For $s\in\hat{\mathcal{S}}\setminus\mathcal{S}$,
define $\mathcal{W}_{s}$ to be the filter 
on the base set $\{\pi_{r_j(t)}(u):u\in\mathcal{R}(j)\}$
generated by the sets $\{\pi_{r_{j}(t)}(u): u \in \mathcal{R}(j)|X/r_j(t)\}$, $X\in \mathcal{C}$, for any (all) $t \in \mathcal{F}$ such that $s \sqsubseteq \varphi(t)$ and $j<|t|$ maximal such that 
$\vp(r_j(t))=s$. 
The proof of the next claim follows exactly as in  \cite{Dobrinen/Todorcevic11}.

\begin{clm}
For each $s\in \hat{\mathcal{S}}\setminus\mathcal{S}$, $\mathcal{W}_{s}$ is an ultrafilter which is generated by the collection of $\{ \pi_{r_{j}(t)}(u) : u \in \mathcal{R}(j)|X\}$, $X\in \mathcal{C}$, for any $t\in \mathcal{F}$ and $j<|t|$ maximal such that $\varphi(r_{j}(t))=s$.
\end{clm}

The proof of the next claim is included, as it differs from its counterpart in the proof of Theorem 5.1 in \cite{Dobrinen/Todorcevic11}.

\begin{clm}\label{claim.new}
Let $s\in \hat{\mathcal{S}}\setminus\mathcal{S}$. Then $\mathcal{W}_{s}$ is isomorphic to $\mathcal{U}_{\mathbf{B}}$ for some $\mathbf{B} \in\mathcal{K}_{\fin}$.
\end{clm}

\begin{proof} 
Fix $t\in \mathcal{F}$ and $j<|t|$ with $j$ maximal such that $\varphi(r_{j}(t))=s$. 
Suppose that $\pi_{r_{j}(t)} = \pi_{\depth}$. 
Then for each $X \in \mathcal{C}$, $\{ \pi_{r_{j}(t)}(u) : u \in \mathcal{R}(j)|X\} = \pi_{\depth}( \mathcal{R}(j)|X) \in \mathcal{Y}_{0}$. 
Since $\mathcal{W}_{s}$ is non-principal, $\mathcal{W}_{s} = \mathcal{Y}_{0}=\mathcal{U}_{\emptyset}$, by Fact $\ref{EqualityFact}$. 
If $\pi_{r_{j}(t)} = \pi_{\mathbb{I}}$,
 then for each $X\in \mathcal{C}$, $\{\pi_{r_{j}(t)}(u) : u \in \mathcal{R}(j)|X/t\}\subseteq \{\pi_{\mathbb{I}}(u) : u \in \mathcal{R}(j)|X\} \in \pi_{\mathbb{I}}( \mathcal{Y}_{j+1})$. 
Thus, by Fact $\ref{EqualityFact}$, $\mathcal{W}_{s}= \pi_{\mathbb{I}}( \mathcal{Y}_{j+1})$. 
By Proposition $\ref{Tim RKstructure}$ (3), there exists $\mathbf{B} \in \mathcal{K}_{\fin}$ such that $\mathcal{W}_{s}\cong\mathcal{U}_{\mathbf{B}}$.
\end{proof}

The proof of the next claim follows as in \cite{Dobrinen/Todorcevic11}.

\begin{clm}
$\mathcal{W}$ is the ultrafilter generated by $\vec{\mathcal{W}}$-trees, where $\vec{\mathcal{W}}= ( \mathcal{W}_{s} : s\in \hat{\mathcal{S}}\setminus\mathcal{S}).$
\end{clm}

The previous claims show that $\mathcal{V}$ is isomorphic to the ultrafilter $\mathcal{W}$ on the base $\mathcal{S}$ generated by the $\vec{\mathcal{W}}$-trees, where for each $s\in \hat{\mathcal{S}}\setminus\mathcal{S}$, $\mathcal{W}_{s}$ is isomorphic to $\mathcal{U}_{\mathbf{B}}$ for some $\mathbf{B} \in\mathcal{K}_{\fin}$.
By the correspondence of ultrafilters of $\vec{\mathcal{W}}$-trees on $\mathcal{S}$ and iterated Fubini products,
we conclude that $\mathcal{V}$ is isomorphic to a Fubini iterate of p-points from among 
$\mathcal{U}_{\mathbf{B}}$, $\mathbf{B} \in\mathcal{K}_{\fin}$.
\end{proof}

\begin{thm} \label{TukeyReducibleTheorem2} 
Suppose that $\mathcal{C}$ is a Ramsey filter on $(\mathcal{R}, \le )$.
 The Rudin-Keisler ordering of the p-points Tukey reducible to $\mathcal{C}$ is isomorphic to the partial order $(\tilde{\mathcal{K}}_{\fin}, \le)$.
In particular, if $|J|<\om$, then the initial  Rudin-Keisler structure below 
$\mathcal{C}$ is isomorphic to the partial order $(\tilde{\mathcal{K}}_{\fin}, \le)$.
\end{thm}

\begin{proof}
The first statement follows from Theorem \ref{TukeyReducibleTheorem}, the correspondence between 
and iterated Fubini products of ultrafilters
(see  Facts 3.4 and 3.4 in \cite{DobrinenCanonicalMaps15}),
and the fact that a Fubini product of ultrafilters is never a p-point.
If  $|J|<\om$, then $\{\mathcal{AR}_1|X:X\in\mathcal{C}\}$ generates a p-point on the base set $\mathcal{AR}_1$, and in this case, every ultrafilter Rudin-Keisler reducible to $\mathcal{C}$ is a p-point. 
\end{proof}

\begin{prop}\label{prop.Tukeyequiv}
Suppose that $\mathbf{B}$ and $\mathbf{C}$ are in $\mathcal{K}_{\fin}$. 
Then
$ \mathcal{U}_{\mathbf{B}} \le_{T} \mathcal{U}_{\mathbf{C}}$ if and only if $J_{\mathbf{B}}\subseteq J_{\mathbf{C}}$.
Hence, $ \mathcal{U}_{\mathbf{B}} \equiv_{T} \mathcal{U}_{\mathbf{C}}$ if and only if $J_{\mathbf{B}}= J_{\mathbf{C}}$.
\end{prop}

\begin{proof}
 Assume $J_{\mathbf{B}} \subseteq J_{\mathbf{C}}$. By Proposition $\ref{Tim RKstructure}$, $\mathcal{U}_{(\mathbf{A}_{0,j})_{j\in J_{\mathbf{B}}}} \le_{T} \mathcal{U}_{\mathbf{B}}$, since $(\mathbf{A}_{0,j})_{j\in J_{\mathbf{B}}}\le \mathbf{B}$. Define $g :\mathcal{C}\upharpoonright\mathcal{B}((\mathbf{A}_{0,j})_{j\in J_{\mathbf{B}}}) \rightarrow \mathcal{C}\upharpoonright\mathcal{B}(\mathbf{B})$ by $g(\mathcal{B}((\mathbf{A}_{0,j})_{j\in J_{\mathbf{B}}})\upharpoonright X) = \mathcal{B}(\mathbf{B})\upharpoonright X$. 
$g$ is well-defined on a cofinal subset of $\mathcal{U}_{(\mathbf{A}_{0,j})_{j\in J_{\mathbf{B}}}}$, since from the set $\mathcal{B}((\mathbf{A}_{0,j})_{j\in J_{\mathbf{B}}})\upharpoonright X$ one can reconstruct $\{\left< k_{n}, (\mathbf{X}_{n,j})_{j\in J_{\mathbf{B}}} \right> : n<\omega\}$. 
Since $g$ is a monotone cofinal map from a cofinal subset of $\mathcal{U}_{(\mathbf{A}_{0,j})_{j\in J_{\mathbf{B}}}}$ into a cofinal subset of $\mathcal{U}_{\mathbf{B}}$, we have $\mathcal{U}_{\mathbf{B}}\le_{T} \mathcal{U}_{(\mathbf{A}_{0,j})_{j\in J_{\mathbf{B}}}}$. Hence $\mathcal{U}_{\mathbf{B} }\equiv_{T} \mathcal{U}_{(\mathbf{A}_{0,j})_{j\in J_{\mathbf{B}}}}$. By a similar argument, $\mathcal{U}_{\mathbf{C}} \equiv_{T} \mathcal{U}_{(\mathbf{A}_{0,j})_{j\in J_{\mathbf{C}}}}$. 
Since $J_{\mathbf{B}}\subseteq J_{\mathbf{C}}$,
 Proposition $\ref{Tim RKstructure}$ implies that $\mathcal{U}_{(\mathbf{A}_{0,j})_{j\in J_{\mathbf{B}}}}\le_{T}  \mathcal{U}_{(\mathbf{A}_{0,j})_{j\in J_{\mathbf{C}}}}$. Therefore $\mathcal{U}_{\mathbf{B}} \le_{T} \mathcal{U}_{\mathbf{C}}$.

 Now suppose that $J_{\mathbf{B}} \not \subseteq J_{\mathbf{C}}$. 
Since $J_{\mathbf{C}}$ is finite, the p-point ultrafilter $\mathcal{U}_{(\mathbf{A}_{0,j})_{j\in J_{\mathbf{C}}}}$ is Tukey equivalent to a Ramsey filter on $(\mathcal{R}(\left< (\mathbf{A}_{k,j})_{j\in J_{\mathbf{C}}}: k<\omega \right>), \le)$. 
By Theorem  $\ref{TukeyReducibleTheorem}$, if $\mathcal{V}$ is a p-point and $\mathcal{V} \le_{T} \mathcal{U}_{(\mathbf{A}_{0,j})_{j\in J_{\mathbf{C}}}}$ then $\mathcal{V}$ is isomorphic to some $\mathcal{U}_{\mathbf{D}}$ for some $\mathbf{D} \in \bigcup_{L\subseteq J_{\mathbf{C}}} (\mathcal{K}_{j})_{j\in L}$ or isomorphic to $\mathcal{Y}_{0}$. 
By Proposition $\ref{Tim RKstructure}$ (5), $\mathcal{U}_{(\mathbf{A}_{0,j})_{j\in J_{\mathbf{B}}}} \not \le_{RK} \mathcal{U}_{\mathbf{D}}$ for each $\mathbf{D} \in \bigcup_{J'\subseteq J_{\mathbf{C}}} (\mathcal{K}_{j})_{j\in J'}$, and also $\mathcal{U}_{(\mathbf{A}_{0,j})_{j\in J_{\mathbf{B}}}} \not \le_{RK} \mathcal{Y}_{0}$. So  $\mathcal{U}_{(\mathbf{A}_{0,j})_{j\in J_{\mathbf{B}}}}\not \le_{T} \mathcal{U}_{(\mathbf{A}_{0,j})_{j\in J_{\mathbf{C}}}}$. Therefore $\mathcal{U}_{\mathbf{B}} \not \le_{T} \mathcal{U}_{\mathbf{C}}$.
\end{proof}

\noindent \it Final Argument for Proof of Theorem \ref{Tim-InitialStructures}. \rm
For now, let $J$ be either finite or $\om$.
We prove (2) first.
Let $\mathcal{V}$ be any p-point Tukey reducible to $\mathcal{C}$.
Since a Fubini product of ultrafilters is never a p-point, it follows from Theorem \ref{TukeyReducibleTheorem2}  that $\mathcal{V}$ is isomorphic to $\mathcal{U}_{\mathbf{B}}$ for some $\mathbf{B}\in \mathcal{K}_{\fin}$. 
In particular, 
$\mathcal{V}$ is 
 Tukey equivalent to $\mathcal{U}_{\mathbf{B}}$.
 Proposition \ref{prop.Tukeyequiv} shows that the Tukey type of $\mathcal{U}_{\mathbf{B}}$ is completely determined by 
the index set $J_{\bsB}$.
Therefore, (2) holds.

Now suppose $J$ is finite.
By Theorem \ref{TukeyReducibleTheorem2}, each ultrafilter $\mathcal{V}$ Tukey reducible to $\mathcal{C}$ is 
isomorphic to a Fubini iterate of ultrafilters from among 
$\mathcal{U}_{\mathbf{B}}$, $\mathbf{B}\in \mathcal{K}_{\fin}$.
Since $J$ is finite, it follows from 
Proposition \ref{prop.Tukeyequiv} that there are only finitely many Tukey types of p-points Tukey reducible to $\mathcal{C}$.
By  Corollary 37 in \cite{Dobrinen/Todorcevic10},
for each
$\mathcal{U}_{\mathbf{B}}$, its Tukey type is the same as the Tukey type of any Fubini power of itself, since  $\mathcal{U}_{\mathbf{B}}$ is a rapid p-point.
Thus, each Fubini iterate from among the 
$\mathcal{U}_{\mathbf{B}}$, $\mathbf{B}\in \mathcal{K}_{\fin}$, has Tukey type equal to some such 
$\mathcal{U}_{\mathbf{B}}$.
Therefore, 
the initial Tukey structure below $\mathcal{C}$ is exactly $\mathcal{P}(J)$.
\qed

We finish by pointing out the implications the theorems in this section have for the  specific examples in Section \ref{sec.uf}.

\begin{example}[$n$-arrow, not $(n+1)$-arrow ultrafilters]
Let $J=1$ and fix $n\ge 2$. 
Recall that the space 
$\mathcal{A}_{n}$ is defined to be the space $\mathcal{R}(\left<\mathbf{A}_{k} : k<\omega \right>)$, 
where $\left<\mathbf{A}_{k} : k<\omega \right>$ is some generating sequence for 
$\mathcal{K}_{0}$, the class of all finite $(n+1)$-clique-free ordered graphs.

 Suppose  $\mathcal{C}$ is a Ramsey filter on $(\mathcal{A}_{n}, \le)$.
Then 
$\mathcal{U}_{\mathcal{A}_n}$, defined to be the filter on base set $\mathcal{R}(0)$ generated by the sets $\mathcal{R}(0)|C$, $C\in\mathcal{C}$, is a p-point ultrafilter  which is Tukey equivalent to $\mathcal{C}$.
By Theorem $\ref{TukeyReducibleTheorem2}$,
 the Rudin-Keisler structure of the p-points Tukey reducible to $\mathcal{U}_{\mathcal{A}_n}$ is isomorphic to 
the collection of all equivalence classes of members of $\mathcal{K}_0$, partially ordered by embedability.
By Theorem \ref{Tim-InitialStructures}, the initial Tukey structure below $\mathcal{U}_{\mathcal{A}_n}$ is exactly a chain of length $2$.
\end{example}

\begin{example}[Hypercube spaces]
Let $J\le \omega$ and for each $j\in J$, let $\mathcal{K}_{j}$ be the class of finite linear orders. Let $\left<\mathbf{A}_{k} : k<\omega \right>$ be a generating sequence such that for each $k<\omega$ and each $j\in J_k$, $\mathbf{A}_{k,j}$ is a $k$-element linear order. Recall that the space $\mathcal{H}^{J}$ is defined to be the space $\mathcal{R}(\left<\mathbf{A}_{k} : k<\omega \right>)$.

If $\mathcal{C}$ is a Ramsey filter on $(\mathcal{H}^{J}, \le)$ then by Theorem $\ref{TukeyReducibleTheorem2}$,
 the Rudin-Keisler structure of the p-points Tukey reducible to $\mathcal{C}$ is isomorphic to the partial order
 $(\tilde{\mathcal{K}}_{\fin}, \le)$. 

If $J<\omega$ then $(\tilde{\mathcal{K}}_{\fin},\le)$ is isomorphic to the partial order $(\omega^{J},\le)$ via the map sending  $\mathbf{B} \mapsto (\|\mathbf{B}_{0}\|,\|\mathbf{B}_{1}\|, \dots, $ $\|\mathbf{B}_{J-1}\|)$. (If  $j\not \in J_{\mathbf{B}}$, then we assume $\|\mathbf{B}_{j}\|= 0$.) Moreover, the initial Tukey structure below $\mathcal{C}$ is exactly $\mathcal{P}(J)$.

Let $C_{0}(\omega)$ denote the collection of sequences of natural numbers which are eventually zero. (For $(x_{i})_{i<\omega}$ and $(y_{i})_{i<\omega}$ in $C_{0}(\omega)$, $(x_{i})_{i<\omega} \le (y_{i})_{i<\omega}$ iff for all $i<\omega$, $x_{i}\le y_{i}$.) 
If $J=\omega$, then $(\tilde{\mathcal{K}}_{\fin},\le)$ is isomorphic to the partial order $(C_{0}(\omega),\le)$ via the map $\mathbf{B} \mapsto (\|\mathbf{B}_{0}\|,\|\mathbf{B}_{1}\|, \dots, \|\mathbf{B}_{J_{\mathbf{B}}}\|, 0, 0,0, \dots)$.
Further, the structure of the Tukey types of the p-points Tukey reducible to $\mathcal{C}$  is exactly $[\om]^{<\om}$.
\end{example}

These examples are just prototypes of what can be achieved by topological Ramsey spaces constructed from generating sequences.  Based on the work in this paper, many examples of topological Ramsey spaces can be constructed, with associated ultrafilters having a wide range of partition properties and initial Rudin-Keisler and Tukey structures.


\bibliographystyle{amsplain}
\bibliography{referencesDMT0108}

\end{document}